\documentclass[11pt]{amsart}

\usepackage{subfiles}

\usepackage{longtable}

\usepackage{tabu}
\usepackage[margin=1in]{geometry}
\usepackage[dvipsnames]{xcolor}   		             		
\usepackage{graphicx}			
\usepackage{amssymb}
\usepackage{mathrsfs}
\usepackage{amsthm}
\usepackage{amsmath}
\usepackage{stmaryrd}
\usepackage{tikz}
\usepackage{tikz-cd}
\usepackage{accents}
\usepackage{upgreek}
\usepackage{enumerate}
\usepackage{bm}
\usepackage{mathtools}
\usepackage[all]{xy}
\usepackage{caption}

\usepackage{url}
\usepackage{float}
\usepackage{todonotes}
\usepackage{bbm}

\usepackage[full]{textcomp}
\usepackage[cal=cm]{mathalfa}
\usepackage{xparse}
\usepackage{comment}
\usepackage{cite}

\tikzset{
  commutative diagrams/.cd, 
  arrow style=tikz, 
  diagrams={>=stealth}
}

\newcommand{\runningexample}{\noindent \textbf{Main example.} }

\theoremstyle{definition}
\newtheorem{innercustomthm}{Theorem}
\newenvironment{customthm}[1]
  {\renewcommand\theinnercustomthm{#1}\innercustomthm}
  {\endinnercustomthm}
  
\makeatletter
\def\@tocline#1#2#3#4#5#6#7{\relax
  \ifnum #1>\c@tocdepth 
  \else
    \par \addpenalty\@secpenalty\addvspace{#2}%
    \begingroup \hyphenpenalty\@M
    \@ifempty{#4}{%
      \@tempdima\csname r@tocindent\number#1\endcsname\relax
    }{%
      \@tempdima#4\relax
    }%
    \parindent\z@ \leftskip#3\relax \advance\leftskip\@tempdima\relax
    \rightskip\@pnumwidth plus4em \parfillskip-\@pnumwidth
    #5\leavevmode\hskip-\@tempdima
      \ifcase #1
       \or\or \hskip 1em \or \hskip 2em \else \hskip 3em \fi%
      #6\nobreak\relax
    \dotfill\hbox to\@pnumwidth{\@tocpagenum{#7}}\par
    \nobreak
    \endgroup
  \fi}
\makeatother

\usetikzlibrary{calc}
\usetikzlibrary{fadings}
\usetikzlibrary{decorations.pathmorphing}
\usetikzlibrary{decorations.pathreplacing}

\newcounter{marginnote}
\DeclareMathAlphabet{\mathpzc}{OT1}{pzc}{m}{it}

\usepackage[backref]{hyperref}
\hypersetup{
  colorlinks   = true,          
  urlcolor     = blue,          
  linkcolor    = purple,          
  citecolor   = blue             
}

\theoremstyle{definition}
\newtheorem{theorem}{Theorem}[section]

\newtheorem{conjecture}[theorem]{Conjecture}
\newtheorem{corollary}[theorem]{Corollary}
\newtheorem{lemma}[theorem]{Lemma}
\newtheorem{proposition}[theorem]{Proposition}
\newtheorem{remark}[theorem]{Remark}
\newtheorem{assumption}[theorem]{Assumption}
\newtheorem*{runningexample*}{Running example}

\newtheorem*{aside*}{Aside}

\newtheorem{construction}[theorem]{Construction}

\newtheorem{definition}[theorem]{Definition}
\newtheorem{example}[theorem]{Example}

\newtheorem{proposition-definition}[theorem]{Proposition-Definition}

\DeclareMathOperator{\ev}{ev}

\DeclareMathOperator{\Hom}{Hom}

\newcommand{\RR}{\mathbb{R}}
\newcommand{\EE}{\mathbf{E}}
\newcommand{\LL}{\mathbf{L}}

\newcommand{\sqC}{\scalebox{0.8}[1.2]{$\sqsubset$}}
\newcommand{\st}{\star}

\newcommand{\op}[1]{\operatorname{#1}}
\newcommand{\ol}[1]{\overline{#1}}
\newcommand{\ul}[1]{\underline{#1}}

\newcommand{\bcd}{\begin{center}\begin{tikzcd}}
\newcommand{\ecd}{\end{tikzcd}\end{center}}

\newcommand{\Tlog}{\op{T}^{\op{log}}}
\newcommand{\TTlog}{\mathbf{T}^{\op{log}}}
\newcommand{\TT}{\mathbf{T}}

\newcommand{\e}{\mathrm{e}}
\newcommand{\Aaff}{\mathbb{A}}
\newcommand{\C}{\mathbb{C}}
\newcommand{\T}{\mathrm{T}}

\newcommand{\PP}{\mathbb{P}}
\newcommand{\OO}{\mathcal{O}}
\newcommand{\N}{\mathbb{N}}
\newcommand{\Z}{\mathbb{Z}}

\newcommand{\virt}{\op{virt}}

\newcommand{\HH}{\mathcal{H}}
\newcommand{\Speck}{\operatorname{Spec}\kfield}
\newcommand{\kfield}{\Bbbk}

\newcommand{\Norm}{\mathrm{N}}
\newcommand{\LogOb}{\mathrm{LogOb}}
\newcommand{\Euler}{\mathrm{e}}

\newcommand{\Mcal}{\mathcal{M}}
\newcommand{\Acal}{\mathcal{A}}

\newcommand{\Xcal}{\mathcal{X}}

\newcommand{\Ccal}{\mathcal{C}}
\newcommand{\Mfrak}{\mathfrak{M}}

\newcommand{\Rder}{\mathbf{R}^\bullet}
\newcommand{\Lder}{\mathbf{L}^\bullet}

\newcommand{\Kup}{\mathsf{K}}
\newcommand{\Fup}{\mathsf{F}}

\newcommand{\Log}{\operatorname{Log}}

\newcommand{\FF}{\mathbf{F}}

\newcommand{\Cstar}{\C^\times}

\newcommand{\Spec}{\operatorname{Spec}}
\newcommand{\acts}{\curvearrowright}

\NewDocumentCommand{\compatibilitydatum}{m m m m m m O{} O{} O{}}{
\begin{equation*} \begin{tikzcd}[ampersand replacement=\&]
  \: \arrow{r} \& {#1} \arrow{r} \arrow{d}{#7} \& {#2} \arrow{r} \arrow{d}{#8} \& {#3} \arrow{r}{[1]} \arrow{d}{#9} \& \: \\
  \: \arrow{r} \& {#4} \arrow{r} \& {#5} \arrow{r} \& {#6} \arrow{r} \& \:
\end{tikzcd} \end{equation*}}

\NewDocumentCommand{\commutingsquare}{m m m m o O{} O{} O{} O{}}{
\begin{equation}\begin{tikzcd}[ampersand replacement=\&] \label{#5}
  #1 \arrow{r}{#6} \arrow{d}{#7} \& #2 \arrow{d}{#8} \\
  #3 \arrow{r}{#9} \& #4
\end{tikzcd}\IfValueTF{#5}{\label{#5}}{} \end{equation}}

\NewDocumentCommand{\cartesiansquare}{m m m m O{} O{} O{} O{}}{
\begin{equation*}\begin{tikzcd}[ampersand replacement=\&]
  #1 \arrow{r}{#5} \arrow{d}{#6} \arrow[dr, phantom, "\square"] \& #2 \arrow{d}{#7} \\
  #3 \arrow{r}{#8} \& #4
\end{tikzcd} \end{equation*}}

\NewDocumentCommand{\cartesiansquarelabel}{m m m m m O{} O{} O{} O{}}{
\begin{tikzcd}[ampersand replacement=\&]
  #1 \arrow{r}{#6} \arrow{d}{#7} \arrow[dr, phantom, "\square"] \& #2 \arrow{d}{#8} \\
  #3 \arrow{r}{#9} \& #4
\end{tikzcd}\IfValueTF{#5}{\label{#5}}{}
}

\NewDocumentCommand{\triangleofspaces}{m m m O{} O{} O{}}{
\begin{tikzcd} [ampersand replacement=\&]
#1 \arrow{r}{#4} \arrow[bend right]{rr}{#5} \& #2 \arrow{r}{#6} \& #3
\end{tikzcd}}

\begin{document}
 
\title{Tangent curves to degenerating hypersurfaces}
\author{Lawrence Jack Barrott and Navid Nabijou}

\begin{abstract}
We study the behaviour of rational curves tangent to a hypersurface under degenerations of the hypersurface. Working within the framework of logarithmic Gromov--Witten theory, we extend the degeneration formula to the logarithmically singular setting, producing a virtual class on the space of maps to the degenerate fibre. We then employ logarithmic deformation theory to express this class as an obstruction bundle integral over the moduli space of ordinary stable maps. This produces new refinements of the logarithmic Gromov--Witten invariants, encoding the degeneration behaviour of tangent curves. In the example of a smooth plane cubic degenerating to the toric boundary we employ localisation and tropical techniques to compute these refinements. Finally, we leverage these calculations to describe how embedded curves tangent to a smooth cubic degenerate as the cubic does; the results obtained are of a classical nature, but the proofs make essential use of logarithmic Gromov--Witten theory.
\end{abstract}

\maketitle
\tableofcontents

\newpage

\null
\vfill

\begin{figure}[h]
\begin{tikzpicture}

\draw [thick] (0,0) -- (7,0);
\draw [fill] (1,0) circle[radius=2pt];
\draw (1,0) node[below]{\small{$\eta$}};
\draw [fill] (5,0) circle[radius=2pt];
\draw (5,0) node[below]{\small{$0$}};

\draw [thick] (0.5,0.5) -- (1.5,1) -- (1.5,2.5) -- (0.5,2) -- (0.5,0.5);
\draw [thick,blue] (0.75,0.75) -- (1.25,2);
\draw [blue] (1.05,1.5) node[left]{\small$2$};
\draw (1.5,1.6) node[right]{$9\left(\frac{3}{4}\right)$};
\draw [thick] (0.5,3) -- (1.5,3.5) -- (1.5,5) -- (0.5,4.5) -- (0.5,3);
\draw [thick,blue] (1,4) circle[radius=10pt];
\draw (1.5,4.1) node[right]{$27$};

\draw [thick] (4.5,0.5) -- (5.5,1) -- (5.5,2.5) -- (4.5,2) -- (4.5,0.5);
\draw [thick,blue] (4.75,0.75) -- (5.25,2);
\draw [blue] (5.05,1.5) node[left]{\small$2$};
\draw [thick] (4.5,3) -- (5.5,3.5) -- (5.5,5) -- (4.5,4.5) -- (4.5,3);
\draw [thick,blue] (4.6,3.2) -- (5.1,4.6);
\draw [thick,blue] (4.9,4.6) -- (5.4,3.7);

\draw [->] (2.3,4.1) -- (4.2,4.1);
\draw (3.45,4.1) node[above]{$18$};

\draw [->] (2.3,3.9) -- (4.2,2);
\draw (3.45,2.9) node[above]{$9$};

\draw [->] (2.8,1.6) -- (4.2,1.6);
\draw (3.45,1.6) node[above]{$9$};

\draw (5.5,4.1) node[right]{$18=18$};
\draw (5.5,1.6) node[right]{$\dfrac{63}{4} = 9+9\left(\frac{3}{4}\right)$};

\end{tikzpicture}	
\caption{Degree $2$ degeneration (\S \ref{sec: deg 2 degeneration}). Total invariant is: $27+9(3/4) = 18+63/4 = 135/4$.}
\end{figure}
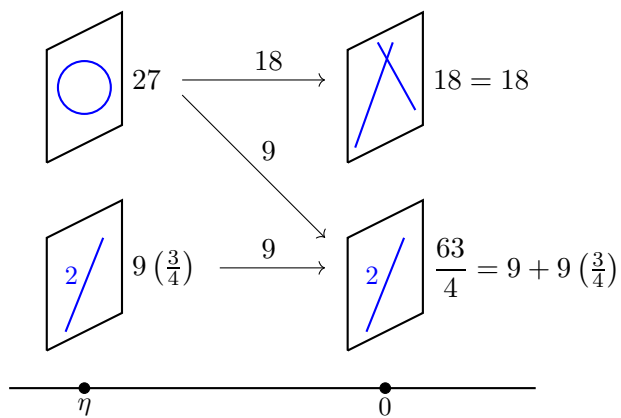\bigskip

\begin{figure}[h]
\begin{tikzpicture}

\draw [thick] (-2,-0.5) -- (7,-0.5);
\draw [fill] (0,-0.5) circle[radius=2pt];
\draw (0,-0.5) node[below]{\small{$\eta$}};
\draw [fill] (5,-0.5) circle[radius=2pt];
\draw (5,-0.5) node[below]{\small{$0$}};

\draw [thick] (-0.5,0) -- (0.5,0.5) -- (0.5,2) -- (-0.5,1.5) -- (-0.5,0);
\draw [thick,blue] (-0.25,0.25) -- (0.25,1.5);
\draw [blue] (0.05,1) node[left]{\small$3$};
\draw (0.5,1) node[right]{$9\left(\frac{10}{9}\right)$};
\draw [thick] (-0.5,2.9) -- (0.5,3.5) -- (0.5,5) -- (-0.5,4.5) -- (-0.5,2.9);
\draw [color=blue,thick] (0.4,4.7) to [out=260,in=0] (-0.2,3.7) to [out = 180,in=270] (-0.45,3.9) to [out=90,in=180] (-0.2,4.2) to [out=0, in=100] (0.4,3.5);

\draw (0.55,3.95) node[right]{$234$};

\draw [thick] (4.5,0) -- (5.5,0.5) -- (5.5,2) -- (4.5,1.5) -- (4.5,0);
\draw [thick,blue] (4.75,0.25) -- (5.25,1.5);
\draw [blue] (5.05,1) node[left]{\small$3$};
\draw [thick] (4.5,2.5) -- (5.5,3) -- (5.5,4.5) -- (4.5,4) -- (4.5,2.5);
\draw [thick,blue] (4.6,2.7) -- (5.1,4.1);
\draw [blue] (4.9,3.4) node[left]{\small$2$};
\draw [thick,blue] (4.9,4.1) -- (5.4,3.2);
\draw [thick] (4.5,5) -- (5.5,5.5) -- (5.5,7) -- (4.5,6.5) -- (4.5,5);
\draw [thick,blue] (4.6,5.2) -- (5.1,6.6);
\draw [thick,blue] (4.9,6.6) -- (5.4,5.7);
\draw [thick,blue] (4.55,5.35) -- (5.4,5.9);

\draw [->] (1.7,4.2) -- (4.2,5.7);
\draw (3,5.05) node[above]{$27$};

\draw [->] (1.7,3.9) -- (4.2,3.3);
\draw (3,3.65) node[above]{$162$};

\draw [->] (1.7,3.6) -- (4.2,1.3);
\draw (3,2.55) node[above]{$45$};

\draw [->] (2,1) -- (4.2,1);
\draw (3,1) node[below]{$9$};

\draw (5.5,6.1) node[right]{$27=27$};
\draw (5.5,3.6) node[right]{$162=162$};
\draw (5.5,1.1) node[right]{$55 = 45+9\left(\frac{10}{9}\right)$};

\end{tikzpicture}	
\caption{Degree $3$ degeneration (\S \ref{sec: deg 3 degeneration}). Total invariant is: $234+9(10/9) = 27+162+55 = 244$.}
\end{figure}
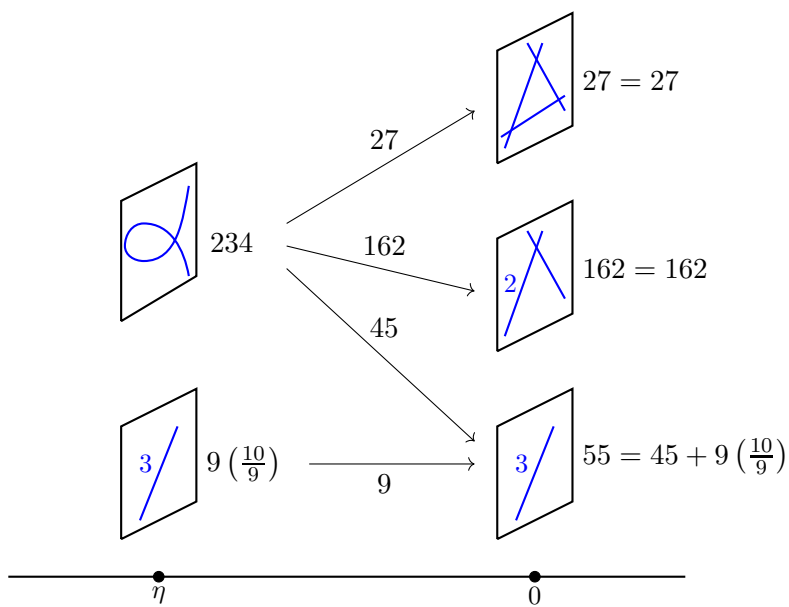

\vfill

\newpage

\section{Introduction}

\noindent Degeneration is a core technique in modern enumerative geometry. The basic idea is to degenerate a given target variety $X$ to a simpler one:
\begin{equation*} X \rightsquigarrow X_0. \end{equation*}
Under suitable conditions, enumerative invariants of the general fibre $X$ can be reconstructed from those of the central fibre $X_0$ \cite{Li2,AbramovichChenGrossSiebertDegeneration,ACGSPunctured,RangExpansions}. If the central fibre is sufficiently simple --- for example, if it decomposes into a union of toric varieties meeting transversely --- then its invariants can in turn be computed directly.

As such, degeneration is usually viewed solely as a method: it calculates the desired invariants on the general fibre in terms of invariants on some auxiliary central fibre. It has been tremendously successful at this task, underpinning many major results in the field: for a sample, see \cite{MaulikPandharipande,OkounkovPandharipandeMatrix,GPS,PandharipandePixtonQuintic}.

There is another aspect of degeneration, however, which has been mostly overlooked: the invariants of the central fibre provide \emph{refinements} of the invariants of the general fibre. This is because the moduli space associated to the central fibre typically has multiple virtual irreducible components. These refinements are geometrically meaningful: they provide information about how algebraic curves in $X$ degenerate as $X$ does.

In this paper, we investigate the geometric meaning of these refinements in the novel context of hypersurface degenerations. We examine rational curves with maximal contact order to a given hypersurface and study their behaviour as the hypersurface degenerates. Using the machinery of logarithmic Gromov--Witten theory, we explicitly calculate the aforementioned refinements, and use them to answer classical (and previously open) questions. Along the way, we establish a general degeneration formula for logarithmically singular families, and develop a new virtual push-forward technique which we exploit to calculate the virtual class on the central fibre.

\subsection{Logarithmically singular degeneration formula (\S \ref{sec: degeneration formula})} The degenerations we wish to study are logarithmically singular, and therefore fall outside the scope of the usual degeneration formula \cite{AbramovichChenGrossSiebertDegeneration, ACGSPunctured}. Our first main result is an extension of the degeneration formula to logarithmically singular families in genus zero.

\begin{customthm}{A}[Theorems~\ref{prop: POT over A1} and \ref{thm: degeneration formula}] \label{thm: degeneration formula introduction} Let $\Xcal$ be a logarithmically smooth scheme, and let $\Xcal \to \Aaff^{\! 1}$ be a projective, surjective and logarithmically flat morphism, where the base is equipped with the trivial logarithmic structure. Choose discrete data for a moduli space of genus zero logarithmic stable maps $\Kup^{\log}(\Xcal)$. Then, there is a perfect obstruction theory for the morphism
\begin{equation*} \Kup^{\log}(\Xcal) \to \Log \Mfrak_{0,n} \times \Aaff^{\! 1}	
\end{equation*}
defining a family of virtual fundamental classes on the fibres of $\Kup^{\log}(\Xcal) \to \Aaff^{\! 1}$
\begin{equation*} [\Kup^{\log}(\Xcal_t)]^{\virt} \end{equation*}
satisfying the conservation of number principle. If a fibre $\Xcal_t$ is logarithmically smooth, then this class coincides with the usual virtual fundamental class for the space of logarithmic stable maps; otherwise, the class we construct is new.\end{customthm}

We are most interested in cases where the general fibre $\Xcal_{t \neq 0}$ is logarithmically smooth but the central fibre $\Xcal_0$ is not. In this situation, $\Xcal \to \Aaff^{\! 1}$ will not be logarithmically smooth.

These hypotheses incorporate at least two distinct classes of examples. The first are hypersurface degenerations, in which $\Xcal = X \times \Aaff^{\! 1}$ with divisorial logarithmic structure induced by a degenerating family of hypersurfaces in $X$. This is the situation we focus on in this paper.

The second are degenerations of varieties, where we take all logarithmic structures to be trivial. This latter class includes the types of degenerations appearing in the classical degeneration formula, however equipped with the trivial logarithmic structure instead of the more standard logarithmic structure encoding the degeneration. It would be interesting to compare the central fibre contributions of Theorem~\ref{thm: degeneration formula introduction} to those which appear in the classical degeneration formula. We speculate that the virtual push-forward methods developed below can be adapted for this purpose, and plan to return to this in future work. We thank the anonymous referee for suggesting this additional direction.

Theorem~\ref{thm: degeneration formula introduction} produces a virtual class on the central fibre, but it does not provide a method to compute it. In the rest of this paper, we develop new tools to solve this problem. These tools are specific to the setting of hypersurface degenerations, though we anticipate that variants may be applied in other contexts.

\subsection{Hypersurface degenerations (\S \ref{sec: preliminaries})} Consider a smooth projective variety $X$ and a family of hypersurfaces with smooth total space
\begin{equation*} Z \subseteq X \times \Aaff^{\! 1} \end{equation*}
whose general fibre $Z_{t \neq 0}$ is smooth and whose central fibre $D=Z_0$ is singular. Our running example is a smooth cubic curve $E \subseteq \PP^2$ degenerating to the toric boundary $\Delta \subseteq \PP^2$. Let $\Xcal$ denote the divisorial logarithmic scheme:
\begin{equation*} \Xcal = (X\times \Aaff^{\! 1},Z). \end{equation*}
The projection $\Xcal \to \Aaff^{\! 1}$ defines a logarithmically flat family satisfying the hypotheses of Theorem~\ref{thm: degeneration formula introduction}. For $t \neq 0$, $\Xcal_t$ is the logarithmic scheme associated to the smooth pair $(X,Z_t)$. However, $\Xcal_0$ is typically \emph{not} the logarithmic scheme associated to the pair $(X,Z_0)$. In fact, $\Xcal_0$ is not even logarithmically smooth. Nevertheless, Theorem~\ref{thm: degeneration formula introduction} produces a virtual class
\begin{equation*} [\Kup^{\log}(\Xcal_0)]^{\virt}\end{equation*}
integrals over which coincide with the logarithmic Gromov--Witten invariants of the general fibre. This produces new information: the moduli space $\Kup^{\log}(\Xcal_0)$ typically has a decomposition into clopen substacks, indexed by appropriate combinatorial data. The individual contributions of these substacks provide refinements of the logarithmic Gromov--Witten invariants of the general fibre, and contain information concerning the degeneration behaviour of tangent curves. See \S \ref{sec: new numbers introduction} below for a detailed discussion of this in the case $(\PP^2,E) \rightsquigarrow (\PP^2,\Delta)$.

\subsection{Virtual push-forward formula (\S \ref{section computing virtual class})} Our goal is to compute the aforementioned refined invariants on $\Xcal_0$. To this end, we establish a powerful virtual push-forward formula. The logarithmic scheme $\Xcal_0$ has underlying variety $X$ and so there is a finite and representable morphism of moduli spaces:
\begin{equation*} \iota \colon \Kup^{\log}(\Xcal_0) \to \Kup(X).\end{equation*}
In general it is not known, or even expected, that there is a simple way to relate the two virtual classes via this map. The basic problem is that the perfect obstruction theory for the logarithmic moduli space is defined over a moduli space of logarithmically smooth curves, which has a different deformation theory to the usual moduli space of prestable curves (accounting for deformations of the logarithmic structure on the base).

We show that for certain hypersurface degenerations, the difference in deformation theories can in fact be controlled and described explicitly, resulting in a virtual push-forward formula:
\begin{customthm}{B}[Theorem~\ref{thm: virtual pushforward}] \label{thm: virtual pushforward introduction} Consider a logarithmically flat family $\Xcal \to \Aaff^{\! 1}$ arising from a hypersurface degeneration as above, and satisfying Assumptions~\ref{assumption: convexity of D} and \ref{assumption factoring through D}. Consider a moduli space of genus zero logarithmic stable maps to $\Xcal$ with maximal tangency at a single marked point. Then, there exists a vector bundle $F$ on $\Kup(X)$ which satisfies
\begin{equation}\label{eqn: virtual pushforward introduction} \iota_\star [\Kup^{\log}(\Xcal_0)]^{\virt} = \e(F) \cap [\Kup(X)]^{\virt}\end{equation}
and 
\begin{equation*} \e(F) \cdot \e(\LogOb) = \e(\pi_\star f^\star \OO_{X}(D)).\end{equation*}
\end{customthm}
Theorem~\ref{thm: virtual pushforward introduction} is obtained from the following result, which isolates the ``logarithmic part'' of the obstruction theory on the central fibre. This part is packaged in the $\LogOb$ term above. 
\begin{customthm}{C}[Theorem~\ref{thm: isolating LogOb}] \label{thm: isolating LogOb introduction} There is a perfect obstruction theory for the morphism $\psi$ in the diagram
\begin{center}
\begin{tikzcd}
\Kup^{\log}(\Xcal_0) \ar[r,"\varphi"] \ar[rr,bend right=20pt, swap, "\psi"] & \Log\Mfrak_{0,1} \ar[r] & \Mfrak_{0,1}
\end{tikzcd}
\end{center}
and an equality of virtual fundamental classes:
\[ \psi^![\Mfrak_{0,1}] = \varphi^![\Log\Mfrak_{0,1}].\]

\end{customthm}
 As far as we are aware, this is the first result giving a direct comparison between logarithmic and non-logarithmic obstruction theories. The difference is controlled by the Artin fan and encoded in the $\LogOb$ term (\S \ref{sec: isolating LogOb} and Definition \ref{def: LogOb}). We show that $\LogOb$ can be calculated explicitly in terms of line bundles associated to piecewise-linear functions on the tropicalisation (\S \ref{sec: computing LogOb}). We explain this in some detail, as we believe similar methods will be applicable in other contexts.

 We remark that there is no analogue of Theorem~\ref{thm: isolating LogOb introduction} on the general fibre. The central fibre moduli space has tightly constrained tropical geometry (\S \ref{sec: tropical moduli and minimal monoid}). This allows us to control the deformation theory of the morphism $\Log \Mfrak_{0,1} \to \Mfrak_{0,1}$ (Proposition~\ref{prop: Lchi in -1,1}), which we use to compare the obstruction theories. We expect our methods to be applicable whenever the tropical geometry is similarly constrained.

\subsection{New numbers, new conjectures (\S \ref{sec: component calculations})} \label{sec: contribution 3}\label{sec: new numbers introduction} The virtual push-forward formula \eqref{eqn: virtual pushforward introduction} reduces the calculation of the (refined) logarithmic Gromov--Witten invariants of $\Xcal_0$ to tautological integrals on a moduli space of ordinary stable maps. In the final section, we employ torus localisation to calculate these integrals in our main example, conjecture new hypergeometric formulae, and deduce consequences for classical enumerative geometry.

Recall that we consider a smooth cubic $E$ degenerating to the toric boundary $\Delta$. The moduli space of logarithmic stable maps to $\Xcal_0$ coincides with the moduli space of ordinary stable maps to $\Delta$ (Lemma \ref{lem: factor through D} and Proposition~\ref{prop: identification of moduli spaces}):
\begin{equation*} \Kup^{\log}(\Xcal_0) = \Kup(\Delta).\end{equation*}
This space decomposes into clopen substacks by fixing the degree of the stable map over each of the three components of $\Delta$. The virtual class $[\Kup^{\log}(\Xcal_0)]^{\virt}$ similarly decomposes. We thus obtain refinements of the maximal contact logarithmic Gromov--Witten invariants of $(\PP^2,E)$, indexed by length-$3$ partitions of the degree.

These refinements can be computed using Theorem~\ref{thm: virtual pushforward introduction}. In the equivariant setting, the class $\e(\LogOb)$ is invertible, and so \eqref{eqn: virtual pushforward introduction} can be rewritten as:
\begin{equation*} \iota_\star [\Kup^{\log}(\Xcal_0)]^{\virt} = \left( \dfrac{\e(\pi_\star f^\star \OO_{\PP^2}(3))}{\e(\LogOb)}\right) \cap [\Kup(\PP^2)].\end{equation*}
This is the formulation which we use to carry out our calculations. The localisation procedure is outlined in \S \ref{sec: localisation scheme}. A novel and crucial aspect is the computation of $\LogOb$ (see \S \ref{sec: computing LogOb}), as well as a recursive algorithm for cutting localisation graphs into simpler pieces (see \S \ref{sec: graph splitting}). These give methods for explicitly computing the logarithmic part of the obstruction theory in terms of evaluation and cotangent line classes, which we expect to be applicable more broadly.

The functoriality of virtual localisation allows us to separate out the individual component contributions and thus compute the desired refinements. We implement the localisation algorithm in accompanying Sage code, which we use to generate the refined invariants up to degree $8$. Organising by unordered multi-degree, the first few cases are:

\begin{center}
\begin{minipage}[t]{0.4\textwidth}
\begin{center}
\noindent \textbf{Degree $1$}\\
\begin{tabular}{| c | c |}
	\hline
	\textbf{Multi-degree} & \textbf{Contribution} \\
	\hline \hline
	$(1,0,0)$ & $9$ \\
	\hline \hline
	\textbf{Total:} & $9$ \\
	\hline
\end{tabular}\bigskip

\noindent \textbf{Degree $2$}\\
\begin{tabular}{| c | c |}
	\hline
	\textbf{Multi-degree} & \textbf{Contribution} \\
	\hline \hline
	$(2,0,0)$ & $63/4$ \\
	\hline
	$(1,1,0)$ & $18$ \\
	\hline \hline
	\textbf{Total:} & $135/4$ \\
	\hline
\end{tabular}\bigskip

\noindent \textbf{Degree $3$}\\
\begin{tabular}{| c | c |}
	\hline
	\textbf{Multi-degree} & \textbf{Contribution} \\
	\hline \hline
	$(3,0,0)$ & $55$ \\
	\hline
	$(2,1,0)$ & $162$ \\
	\hline 
	$(1,1,1)$ & $27$ \\
	\hline\hline
	\textbf{Total:} & $244$ \\
	\hline
\end{tabular}
\end{center}
\end{minipage}
\begin{minipage}[t]{0.4\textwidth}
\begin{center}
\noindent \textbf{Degree $4$}\\
\begin{tabular}{| c | c |}
	\hline
	\textbf{Multi-degree} & \textbf{Contribution} \\
	\hline \hline
	$(4,0,0)$ & $4,\!095/16$ \\
	\hline
	$(3,1,0)$ &  $936$ \\
	\hline 
	$(2,2,0)$ &  $1,\!089/2$ \\
	\hline
	$(2,1,1)$ & $576$ \\
	\hline\hline
	\textbf{Total:} & $36,\!999/16$ \\
	\hline
\end{tabular}\bigskip

\noindent \textbf{Degree $5$}\\
\begin{tabular}{| c | c |}
	\hline
	\textbf{Multi-degree} & \textbf{Contribution} \\
	\hline \hline
	$(5,0,0)$ & $34,\!884/25$ \\
	\hline
	$(4,1,0)$ &  $6,\!120$ \\
	\hline 
	$(3,2,0)$ &  $8,\!190$ \\
	\hline
	$(3,1,1)$ &  $4,\!680$ \\
	\hline
	$(2,2,1)$ & $5,\!040$ \\
	\hline\hline
	\textbf{Total:} & $635,\!634/25$ \\
	\hline
\end{tabular}
\end{center}
\end{minipage}
\end{center}
\noindent  Complete tables are given in \S \ref{sec: tables}. Based on these low-degree calculations, we conjecture general hypergeometric formulae for some of the ordered multi-degree contributions (Conjecture \ref{conj: hypergeometric conjecture}):
\begin{align} \label{eqn: introduction multi-degree d conjecture} \operatorname{C}_{\operatorname{ord}}(d,0,0) = \dfrac{1}{d^2\ } {4d-1 \choose d} \qquad & (d \geq 1), \\[8pt]
\operatorname{C}_{\operatorname{ord}}(d_1,d_2,0) = \dfrac{6}{d_1 d_2} {4d_1+2d_2-1 \choose d_1 - 1} {4d_2 + 2d_1 - 1 \choose d_2 - 1} \qquad & (d_1,d_2\geq 1).
\end{align}
Although we are unable to prove these conjectures, we provide strong theoretical evidence for them: we show in Proposition~\ref{prop: hypergeometric conjecture equivalent to combinatorial conjecture} that \eqref{eqn: introduction multi-degree d conjecture} is equivalent to the purely combinatorial Conjecture~\ref{conj: combinatorial conjecture}, which has been verified for large $d$.

Finally (\S \ref{sec: degenerations of curves}) we show how these component contributions can be leveraged to uncover the classical degeneration behaviour of embedded tangent curves to $E$, uncovering complete information for $d=2$ and $d=3$. The resulting degeneration pictures are presented at the start of this paper. 

\subsection{Future directions} There are a number of directions in which the techniques developed in this paper can be applied. All share a common theme: the calculation of geometrically meaningful refinements of logarithmic Gromov--Witten invariants, arising in the central fibre of a degenerating family.

A relatively straightforward extension would be to repeat the calculations of \S \ref{sec: component calculations} for the other toric del Pezzo surfaces. More difficult, but perhaps more interesting, would be to consider higher-dimensional targets (here some difficulties might arise from the fact that a general degeneration of a smooth divisor will not have smooth total space). Another interesting direction would be to consider degenerations to non-reduced divisors.

Even within the scope of our main example, it remains to unravel the degeneration pictures of  \S \ref{sec: degenerations of curves} for $d \geq 4$. As we discuss, this requires a refinement of our construction which separates the contributions of different torsion points. It is possible that such a refinement can be found by synthesising our techniques with the scattering diagram approach of \cite{Graefnitz}. We plan to investigate this jointly with T. Gr\"afnitz.

The logarithmically singular degeneration formula also applies to degenerations of a target $X \rightsquigarrow X_0$ with trivial logarithmic structure. This provides an alternative to the usual degeneration formula, since the central fibre moduli space is a space of \emph{ordinary} stable maps to $X_0$. It would be worthwhile to explore the geometric meaning of the central fibre refinements, and how they compare to the contributions in the classical degeneration formula.

\subsection{Relation to work of Gr\"afnitz.} \label{rmk: Graefnitz}
While this work was in progress, we became aware of the beautiful \cite{Graefnitz}, which also treats a degeneration of $(\PP^2,E)$. Although both the degeneration and the techniques involved are completely different to ours, there appears to be some concordance in the resulting numerical calculations.

Gr\"afnitz considers a simultaneous degeneration of the total space and the divisor (depending on two parameters $s$ and $t$) whose central fibre is a union of three copies of $\PP(1,1,3)$ glued along toric strata (see \cite[Example 1.5]{Graefnitz}). He then passes to an an affine slice by fixing $s = s_0 \neq 0$, and (after passing to a suitable resolution to ensure logarithmic smoothness) applies the degeneration formula, which expresses the invariant as a sum over rigid tropical curves. Our family, on the other hand, in which the ambient space does not change, is obtained by taking the orthogonal affine slice $t = t_0 \neq 0$:
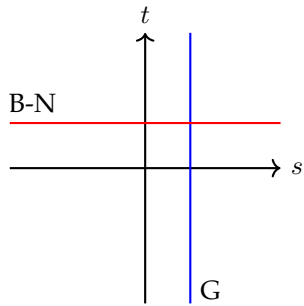
\begin{center}
\begin{tikzpicture}[scale=0.6]
	\draw [thick,->] (-3,0) -- (3,0);
	\draw (3,0) node[right]{\small$s$};
	
	\draw[thick,->] (0,-3) -- (0,3);
	\draw (0,3) node[above]{\small$t$};
	
	\draw [blue,thick] (1,-3) -- (1,3);
	\draw (1,-2.7) node[right]{\small{G}};
	
	\draw [red,thick] (-3,1) -- (3,1);
	\draw (-2.5,1) node[above]{\small{B-N}};
	
\end{tikzpicture}
\end{center}
It is natural to try to marry the two approaches, by taking a further degeneration of each towards the origin. Heuristic arguments suggest that when we do this, our component contributions should correspond to Gr\"afnitz's tropical curve contributions, with the multi-degree given by ``projection to infinity'' (see \cite[Figure 3.4 and Remark 4.16]{Graefnitz}). For $d \leq 4$ one can enumerate the tropical curves by hand, and thus verify this correspondence (see \cite[\S 7.1.2]{Graefnitz}).

If this could be made precise, it would explain the simplicity of the formulae appearing in Conjecture~\ref{conj: hypergeometric conjecture}, as they would be the first order terms of a scattering process. In the other direction, however, it is not immediately clear how our conjectured formula \eqref{eqn: degree d1 d2 conjecture} could have been arrived at via scattering diagram considerations.

\subsection{User's guide} This paper is organised into four sections. Each section has a fixed set of hypotheses, stated clearly at the start. The sections are ordered from most general to most specific:
\begin{itemize}
\item \S \ref{sec: degeneration formula}: We consider logarithmically flat families $\Xcal \to \Aaff^{\! 1}$ with logarithmically smooth total space. The main result is the logarithmically singular degeneration formula (Theorem~\ref{thm: degeneration formula}).\medskip
\item \S \ref{sec: hypersurface degenerations}: We consider logarithmic families arising from hypersurface degenerations, fixing notation and verifying logarithmic flatness (Lemma~\ref{lemma p logarithmic flat}).\medskip
\item \S \ref{section computing virtual class}: We restrict to maximal contact invariants of hypersurface degenerations satisfying the additional Assumptions~ \ref{assumption: convexity of D} and \ref{assumption factoring through D}. The main result is the virtual push-forward formula (Theorem~\ref{thm: virtual pushforward}).\medskip
\item \S \ref{sec: component calculations}: We restrict to the main example $(\PP^2,E)\rightsquigarrow(\PP^2,\Delta)$. We calculate the refinements via localisation (\S \ref{sec: localisation scheme}), producing tables up to $d=8$ (\S \ref{sec: tables}), and conjecturing general hypergeometric formulae (\S \ref{sec: hypergeometric formulae}). Finally, we use the low-degree calculations to uncover the behaviour of embedded tangent curves (\S \ref{sec: embedded curves}).
\end{itemize}

\subsection{Logarithmic background} In this paper we assume familiarity with the basics of logarithmic geometry. We now provide a high-level overview of the subject. Details may be found in any modern reference, see e.g. \cite{AbramovichLog, AbramovichEtAlSkeletons,OgusBook}.

Conceptually, a logarithmic structure on a scheme $X$ is a collection of functions which are declared to be monomials. Formally, a logarithmic structure consists of a constructible sheaf of monoids $\overline{M}_X$ (which plays the role of an indexing sheaf for the monomials) and a line bundle and section $(\OO_X(\alpha),s_\alpha)$ associated to every section $\alpha$ of $\overline{M}_X$. This data is equivalently encoded in a morphism
\[ X \to \Acal_X\]
where $\Acal_X$ is the \emph{Artin fan} of $X$ \cite{OlssonLogStacks,BorneVistoli,AbramovichWiseBirational}. This is an irreducible zero-dimensional Artin stack, locally modelled on the quotient of a toric variety by its dense torus. The pairs $(\OO_X(\alpha),s_\alpha)$ on $X$ arise as pullbacks of certain universal pairs on $\Acal_X$.

The data of $\Acal_X$ is equivalent to the data of the tropicalisation $\Sigma_X$, which is an abstract cone complex generalising the fan of a toric variety \cite{CavalieriChanUlirschWise}. Piecewise-linear functions on $\Sigma_X$ are equivalent to global sections of $\overline{M}_X$, and so the association $\alpha \mapsto (\OO_X(\alpha),s_\alpha)$ generalises the correspondence between piecewise-linear functions and toric Cartier divisors. Much of logarithmic geometry may profitably be interpreted as a far-reaching generalisation of toric geometry.

\subsection{Conventions} We work over an algebraically closed field of characteristic zero, denoted $\kfield$. Given a morphism $X \to Y$ of stacks we will denote the derived dual of the relative cotangent complex by
\begin{equation*} \TT_{X|Y} := (\LL_{X|Y})^\vee \end{equation*}
and refer to it as the relative tangent complex. Similar conventions will apply in the logarithmic setting. The following table collects the most commonly-used symbols:
\begin{center}
\begin{longtable}{r c c p{0.8\textwidth}}
$X$ & & & a smooth projective variety\\
$Z$ & & & a smooth divisor in $X \times \Aaff^{\! 1}$ with smooth general fibre $Z_{t \neq 0}$ and singular central fibre $Z_0 \subseteq X$ \\
$D$ & & & the central fibre $Z_0$\\
$\Xcal$ & & & the variety $X \times \Aaff^{\! 1}$ equipped with the divisorial logarithmic structure corresponding to $Z$\\
$\Xcal_t$ & & & the variety $X$ with the logarithmic structured obtained by pulling back $M_{\Xcal}$ from $X \times \Aaff^{\! 1}$ to $X \times \{t\}$\\
$\Delta$ & & & the toric boundary in $\PP^2$\\
$Q$ & & & the toric monoid associated to a tropical stable map\\
$P$ & & & the free monoid associated to a nodal curve \\
$U_R$ & & & the affine toric variety $\Spec \kfield[R]$ associated to a toric monoid $R$ \\
$T_R$ & & & the dense torus in $U_R$ \\
$\Acal_R$ & & & the Artin cone $[U_R/T_R]$\\
$\Kup$ & & & a moduli space of ordinary stable maps \\
$\Kup^{\log}$ & & & a moduli space of (minimal) logarithmic stable maps \\
$\Mfrak_{0,1}$ & & & the moduli space of prestable curves \\
$\Log\Mfrak_{0,1}$ & & & the moduli space of (not-necessarily minimal) logarithmically smooth curves.
\end{longtable}
\end{center}

\subsection{Acknowledgements} It is a pleasure to thank Pierrick Bousseau and Tim Gr\"afnitz for many inspiring discussions. We also thank Dan~Abramovich, Michel~van~Garrel, Tom~Graber, Mark~Gross, Sanghyeon~Lee, Dhruv~Ranganathan and Helge~Ruddat for helpful conversations. We are grateful to the anonymous referee for numerous helpful suggestions and expositional advice, as well as the observation that our logarithmically singular degeneration formula might have applications beyond hypersurface degenerations.

Parts of this work were carried out during research visits at the National Centre for Theoretical Sciences Taipei, the University of Glasgow, the Mathematisches Forschungsinstitut Oberwolfach and Boston College, and it is a pleasure to thank these institutions for hospitality and financial support. L.J.B. was supported by the National Centre for Theoretical Sciences Taipei and Boston College. N.N. was supported by EPSRC grant EP/R009325/1.

\section{Degeneration formula for logarithmically singular families} \label{section constructing obstruction theory}\label{sec: degeneration formula}

\noindent Let $\Xcal$ be a logarithmic scheme and let $p \colon \Xcal \to \Aaff^{\! 1}$ be a projective surjective morphism, where the base is equipped with the \emph{trivial} logarithmic structure. We assume:
\begin{enumerate}
\item $\Xcal$ is logarithmically smooth over the trivial logarithmic point.
\item $p \colon \Xcal \to \Aaff^{\! 1}$ is logarithmically flat.	
\end{enumerate}
Since $\Aaff^{\! 1}$ has the trivial logarithmic structure, $p$ is logarithmically flat if and only if the morphism $\Xcal \to \Acal_\Xcal \times \Aaff^{\! 1}$ is flat in the usual sense \cite{GillamLogFlat}.

Under these assumptions we establish a degeneration formula, in genus zero, for the family of logarithmic schemes given by $p$. Note that $p$ may be logarithmically singular, and there are many interesting examples for which this is the case.

Choose arbitrary discrete data for a moduli space of genus zero logarithmic stable maps to $\Xcal$ and denote the resulting moduli space by $\Kup^{\log}(\Xcal)$. Since every stable map to $\Xcal$ must factor through a fibre of $p$, there is a proper morphism
\begin{equation*} q \colon \Kup^{\log}(\Xcal) \to \Aaff^{\! 1}\end{equation*}
whose fibre over a point $t \in \Aaff^{\! 1}$ is the moduli space of stable maps to the corresponding fibre of $p$:
\begin{equation*} \Kup^{\log}(\Xcal)_t = \Kup^{\log}(\Xcal_t).\end{equation*}
We begin by constructing a perfect obstruction theory for the family of moduli spaces $\Kup^{\log}(\Xcal)$ relative to the base $\Log\Mfrak_{0,n} \times \Aaff^{\! 1}$ (Theorem~\ref{thm: POT over A1}).

This produces a bivariant class for the morphism $q$, giving a family of virtual classes on the fibres satisfying the conservation of number principle. If a given fibre $\Xcal_t$ is logarithmically smooth, then the induced class on this fibre coincides with the usual virtual fundamental class for the moduli space of logarithmic stable maps to $\Xcal_t$. Otherwise, the class we construct is new (Theorem~\ref{thm: degeneration formula}).

Logarithmic families to which this result applies include hypersurface degenerations (see \S \ref{sec: hypersurface degenerations}), as well as target degenerations with trivial logarithmic structure.

\subsection{Perfect obstruction theory}
Since $\Xcal$ is logarithmically smooth, the space $\Kup^{\log}(\Xcal)$ admits a perfect obstruction theory over the moduli stack of not-necessarily-minimal logarithmically smooth curves. The latter stack is not smooth, but is nonetheless logarithmically smooth and irreducible of the expected dimension (in this case, $n-3$). It can be described \cite[Appendix A]{GrossSiebertLog} as
\begin{equation*} \Log\Mfrak_{0,n} \end{equation*}
where $\Mfrak_{0,n}$ is viewed as a logarithmic stack with divisorial logarithmic structure corresponding to the locus of singular curves, and $\Log$ denotes Olsson's moduli stack of logarithmic structures \cite{OlssonLogStacks}.

\begin{theorem}\label{prop: POT over A1} \label{thm: POT over A1} There exists a compatible triple of perfect obstruction theories for the diagram
\begin{center}
\begin{tikzcd}
\Kup^{\log}(\Xcal) \ar[r,"\rho"] \ar[rr,bend right=20pt] & \Log\Mfrak_{0,n}\times\Aaff^{\! 1} \ar[r] & \Log\Mfrak_{0,n}
\end{tikzcd}
\end{center}
in which $\EE_{\Kup^{\log}(\Xcal)|\Log\Mfrak_{0,n}}$ is the usual obstruction theory for the moduli space of logarithmic stable maps. The obstruction theory for $\rho$ is given explicitly by:
\[(\Rder\pi_\star \Lder f^\star \TTlog_{\Xcal/\Aaff^{\! 1}} ) ^\vee .\]
\end{theorem}
\begin{proof}

  Consider the following commutative diagram involving the universal logarithmic stable map:
\begin{center}
\bcd
\Ccal \ar[r,"f"] \ar[d,"\pi" left] & \Xcal \ar[d,"p"] \\
\Kup^{\log}(\Xcal) \ar[r,"q"] & \Aaff^{\! 1}.
\ecd
\end{center}
There is the following exact triangle on $\Xcal$:
\begin{equation}\label{eqn: pushforward log tangent} \TTlog_{\Xcal/\Aaff^{\! 1}} \to \Tlog_{\Xcal} \to p^\st \T_{\Aaff^{\! 1}} \xrightarrow{{[1]}}. \end{equation}
We have $\Rder\pi_\star f^\star p^\star \T_{\Aaff^{\! 1}}=q^\star \T_{\Aaff^{\! 1}}$ because $\Rder \pi_\star \OO_{\Ccal} = \OO_{\Kup^{\log}(\Xcal)}$ (this is the only place where the genus zero assumption is used). Applying $\Rder\pi_\st \Lder f^\st$ to \eqref{eqn: pushforward log tangent} gives an exact triangle
\begin{equation*} \Rder \pi_\star \Lder f^\star \TTlog_{\Xcal/\Aaff^{\! 1}} \to \Rder\pi_\star f^\star \Tlog_{\Xcal} \to q^\star \T_{\Aaff^{\! 1}} \xrightarrow{[1]} \end{equation*}
which we dualise to obtain:
\begin{equation*} q^\star \Omega_{\Aaff^{\! 1}} \to \EE_{\Kup^{\log}(\Xcal)|\Log\Mfrak_{0,n}} \to \EE_{\Kup^{\log}(\Xcal)|\Log \Mfrak_{0,n} \times \Aaff^{\! 1}} \xrightarrow{{[1]}}.\end{equation*}
The arguments of \cite[Construction 3.~13]{ManolachePull} then apply to show that $\EE_{\Kup^{\log}(\Xcal)|\Log \Mfrak_{0,n} \times \Aaff^{\! 1}}$ forms a perfect obstruction theory. Commutativity with the morphisms to the cotangent complexes is automatic, since $\Log \Mfrak_{0,n} \times \Aaff^{\! 1} \to \Log \Mfrak_{0,n}$ is smooth.
\end{proof}

\subsection{Conservation of number} With this at hand, we may consider for any $t \in \Aaff^{\! 1}$ the cartesian square:
\begin{center}
\begin{tikzcd}
\Kup^{\log}(\Xcal_t) \ar[r, hook,"j_t"] \ar[d,"\rho_t"] \ar[rd,phantom,"\square"] & \Kup^{\log}(\Xcal) \ar[d,"\rho"] \\
\Log\Mfrak_{0,n} \ar[r,hook,"i_t"] & \Log\Mfrak_{0,n}\times\Aaff^{\! 1}
\end{tikzcd}
\end{center}
The obstruction theory $\EE_\rho=\EE_{\Kup^{\log}(\Xcal)|\Log\Mfrak_{0,n}\times\Aaff^{\! 1}}$ defines \cite{ManolachePull} a refined virtual pullback morphism $\rho^!$ from which we obtain a virtual class
\begin{equation}\label{eq: virtual class} [\Kup^{\log}(\Xcal_t)]^{\virt} := \rho^! [ \Log\Mfrak_{0,n}] \end{equation}
(recall that $\Log\Mfrak_{0,n}$ is an irreducible stack of dimension $n-3$). Equivalently we have 
\begin{equation*}  [\Kup^{\log}(\Xcal_t)]^{\virt} = \rho_t^! [ \Log\Mfrak_{0,n}] \end{equation*}
where $\EE_{\rho_t}=\EE_{\Kup^{\log}(\Xcal_t)|\Log\Mfrak_{0,n}} = \Lder j_t^\st\EE_{\Kup^{\log}(\Xcal)|\Log\Mfrak_{0,n}\times\Aaff^{\! 1}}$ is the induced perfect obstruction theory (see \cite[Proposition 7.2]{BehrendFantechi}). The usual pull-push yoga shows that the family of classes \eqref{eq: virtual class} satisfies the conservation of number principle \cite[Proposition 3.9]{ManolachePush}.

\subsection{Degeneration formula} We now assume that the general fibre $\Xcal_{t \neq 0}$ is logarithmically smooth. In this setting, the conservation of number principle will ensure that the invariants of the central fibre $\Xcal_0$ coincide with those of the logarithmically smooth general fibre $\Xcal_{t \neq 0}$.

\begin{lemma} \label{lem: obstruction theory on fibre} For $t \in \Aaff^{\! 1}$ the obstruction theory $\EE_{\rho_t}$ is given by:
\begin{equation*} \EE_{\rho_t} = (\Rder \pi_\star \Lder f^\star \TT^{\log}_{\Xcal_t})^\vee. \end{equation*}\end{lemma}
\begin{proof} The only slightly delicate point here is to observe that $p\colon \Xcal \to \Aaff^{\! 1}$ is logarithmically flat by assumption, and so by \cite[(1.1 (iv))]{OlssonCotangent} we have a natural isomorphism:
\begin{equation*} \Lder i_t^\st \TTlog_p = \TTlog_{\Xcal_t}. \end{equation*}
The result then follows immediately by pull-push yoga. \end{proof}

Note in particular that $\EE_{\rho_t}$ is necessarily of perfect amplitude contained in $[-1,0]$. For arbitrary targets this does not hold: we have used the fact that $\Xcal_t$ sits in a logarithmically flat family $\Xcal \to \Aaff^{\! 1}$ with logarithmically smooth total space.

\begin{theorem}[Degeneration formula] \label{thm: degeneration formula}
Suppose that the general fibre $\Xcal_{t \neq 0}$ is logarithmically smooth. For all $t$ we have a virtual fundamental class
\begin{equation*} [\Kup^{\log}(\Xcal_t)]^{\virt} = \rho_t^! [ \Log\Mfrak_{0,n}] \end{equation*}
satisfying the conservation of number principle. For $t \neq 0$ this coincides with the usual virtual fundamental class for the logarithmically smooth target $\Xcal_t$.\end{theorem}
\begin{proof}This follows from Lemma~\ref{lem: obstruction theory on fibre}, which shows that the two obstruction theories on $\Kup^{\log}(\Xcal_t)$ coincide.\end{proof}
 
\begin{remark} In certain situations it may be possible to resolve the logarithmic singularities of the morphism $\Xcal \to \Aaff^{\! 1}$ and then apply the usual degeneration formula. However, this is not always possible: for instance, in the setting of hypersurface degenerations, $\Xcal$ admits no nontrivial logarithmic modifications. In any event, we prefer to work with the original family, where the connection to the classical enumerative geometry is more direct. \end{remark}

\subsection{Refined invariants} The newly-constructed virtual class for the central fibre
\begin{equation*} [\Kup^{\log}(\Xcal_0)]^{\virt} \end{equation*}
naturally decomposes as a sum over the connected components of the moduli space. More generally, any decomposition of $\Kup^{\log}(\Xcal_0)$ into a union of clopen substacks produces a corresponding decomposition of the virtual class. This produces refinements of the logarithmic Gromov--Witten invariants of the general fibre.

In most cases, this provides more information than is available on the general fibre. This additional information is geometrically meaningful: it describes how logarithmic stable maps degenerate in the family. We explore this circle of ideas, in the case of $(\PP^2,E) \rightsquigarrow (\PP^2,\Delta)$, in \S \ref{sec: component calculations}, and deduce previously-unknown results in classical enumerative geometry.

What we still require is a method for computing these component contributions. This is the original motivation for the development of the virtual push-forward technique, which we establish over the next two sections.

\section{Hypersurface degenerations}\label{sec: preliminaries}\label{sec: hypersurface degenerations}
\noindent We will apply the degeneration formula of the previous section to hypersurface degenerations. In this section, we describe the target geometry and establish basic facts about the moduli space. In the next section, we prove the virtual push-forward result which will allow us to compute the central fibre refinements.

\subsection{Geometric setup} Our target geometry consists of a static ambient variety $X$ --- assumed to be smooth and projective --- together with a one-parameter family of divisors in $X$
\begin{equation*} Z \subseteq X \times \Aaff^{\! 1} \end{equation*}
with smooth total space $Z$. We are interested in the situation where the general fibre $Z_{t \neq 0}$ is smooth but the central fibre $D=Z_0$ is singular.

\runningexample We adopt the following running example, which serves as our primary motivation (indeed, the calculations and conjectures of \S\ref{sec: component calculations} will focus entirely on this case). Take $X=\PP^2$ and consider a family $Z \subseteq \PP^2 \times \Aaff^{\! 1}$ of plane cubics whose general fibres $Z_{t\neq 0}$ are smooth genus one curves and whose central fibre is the toric boundary
\begin{equation*} D = Z_0 = \Delta = D_0 \cup D_1 \cup D_2 \subseteq \PP^2_{z_0 z_1 z_2}. \end{equation*}
Such families with smooth total space $Z$ are easy to construct, by choosing a sufficiently generic homogeneous cubic.

\subsection{Logarithmic structures}\label{sec: logarithmic structures} Given the geometric setup above, we introduce several associated logarithmic schemes. We will consistently use calligraphic letters to denote logarithmic schemes, and ordinary letters to denote ordinary schemes.

\begin{definition} Let $\Xcal$ denote the variety $X \times \Aaff^{\! 1}$ equipped with the divisorial logarithmic structure corresponding to $Z$. Note that since $Z$ is smooth, $\Xcal$ is logarithmically smooth over the trivial logarithmic point. \end{definition}

\begin{definition} For $t \in \Aaff^{\! 1}$ let $\Xcal_t$ denote the variety $X$ equipped with the pullback of the logarithmic structure $M_{\Xcal}$ along the inclusion $i_t \colon X = X \times \{ t \} \hookrightarrow X \times \Aaff^{\! 1}$.\end{definition}
 The ghost sheaf of $\Xcal_t$ is given by
\begin{equation*}\ol{M}_{\Xcal_t} = i_t^{-1} \ol{M}_{\Xcal} = \ul{\N}_{Z_t},\end{equation*}
that is, the constant sheaf with stalk $\N$ supported along the divisor $Z_t$. The line bundle and section associated to the generator of this monoid are:
\begin{equation*} (i_t^\st \OO_{X \times \Aaff^{\! 1}}(Z), i_t^\st s_{Z}) = (\OO_{X}(Z_t),s_{Z_t}).\end{equation*}
From this we have, using the Olsson/Borne--Vistoli characterisation of logarithmic structures \cite{OlssonLogStacks,BorneVistoli}:
\begin{lemma} For $t \neq 0$, $M_{\Xcal_t}$ is the divisorial logarithmic structure associated to the smooth divisor $Z_t \subseteq X$.\end{lemma}
 For $t=0$, however, this is not the case. It is easy to see why: the ghost sheaf $\ol{M}_{\Xcal_0}$ is the constant sheaf $\ul{\N}_D$ supported on $D=Z_0$, whereas the ghost sheaf for the divisorial logarithmic structure will have higher rank along the singular locus of $D$. In the language of \cite{ChenLog}, $M_{\Xcal_0}$ is the rank one Deligne--Faltings pair associated to $(\OO_{X}(D),s_{D})$. 

The following result emphasises the mildly pathological nature of the logarithmic scheme $\Xcal_0$.

\begin{lemma}\label{lem: X0 not logarithmically smooth} The logarithmic scheme $\Xcal_0$ is not logarithmically smooth.\end{lemma}
\begin{proof} It is enough to show that $\Xcal_0$ is not logarithmically regular \cite[Theorem~3.5.1]{OgusBook}. Let $x \in D$ be a singular point of the divisor and let $g \in \OO_{X,x}$ be the germ of a local equation for $D$. A local chart for the logarithmic structure is given by $\N \to \OO_X, 1 \mapsto g$. Thus, using the notation of \cite[Definition 2.2]{Niziol}, the ideal generated by the non-invertible part of the logarithmic structure is
\begin{equation*} I_{X,x}\OO_{X,x} = (g) \vartriangleleft \OO_{X,x} \end{equation*}
and so $\OO_{X,x}/I_{X,x}\OO_{X,x} = \OO_{D,x}$, which is not regular since $x$ is a singular point of $D$. We thus conclude that $\Xcal_0$ is not logarithmically smooth at $x$.	
\end{proof}

\begin{remark} The structure morphism from $\Xcal_0$ to the trivial logarithmic point satisfies the conditions to be a \emph{generically logarithmically smooth family}, as defined in \cite[Definition 2.1]{FFRSmoothingToroidal}.
\end{remark}

\runningexample Recall that $Z \subseteq \PP^2 \times \Aaff^{\! 1}$ is a family of plane cubics with smooth general fibre and central fibre equal to the toric boundary. For $t \neq 0$ we get the divisorial logarithmic structure associated to the smooth plane curve $Z_t \subseteq \PP^2$, while for $t=0$ the logarithmic structure has local charts given by
\begin{equation*} \N \to \OO_{\PP^2}, 1 \mapsto s_\Delta \end{equation*}
where $s_\Delta$ is a local equation for the toric boundary $\Delta \subseteq \PP^2$. Away from the co-ordinate points $[1,0,0],[0,1,0],[0,0,1]$ the boundary is smooth and the logarithmic structure agrees with the divisorial logarithmic structure with respect to $\Delta$, but near such a co-ordinate point it differs since the stalk of the ghost sheaf is $\N$. Note that the tropicalisation of $\Xcal_0$ is simply a ray (see \cite[Appendix B]{GrossSiebertLog} for an introduction to tropicalisations of logarithmic schemes).\medskip

\subsection{Logarithmic flatness} Returning to the general setup, let $p$ denote the second projection:
\begin{equation*} p \colon X \times \Aaff^{\! 1} \to \Aaff^{\! 1}. \end{equation*}
We equip $\Aaff^{\! 1}$ with the trivial logarithmic structure. With the logarithmic structures on source and target fixed, there is a unique enhancement of $p$ to a logarithmic morphism $p \colon \Xcal \to \Aaff^{\! 1}$.

\begin{remark} Note that there is no way of enhancing $p$ to a logarithmic morphism if we equip $\Aaff^{\! 1}$ with its toric logarithmic structure; since the central fibre contains points where the logarithmic structure is trivial, such an enhancement would restrict to give a map from the trivial logarithmic point to the standard logarithmic point, which does not exist.\end{remark}
With this definition, the logarithmic schemes $\Xcal_t$ arise as fibre products
\bcd
\Xcal_t \ar[r] \ar[d] \ar[rd,phantom,"\square"] & \Xcal \ar[d,"p"] \\
\{ t \} \ar[r,hook] & \Aaff^{\! 1}
\ecd
(note that since the base morphism is strict, the underlying scheme of the fibre product is the fibre product of the underlying schemes). The following result will allow us to apply the degeneration formula established in \S \ref{section constructing obstruction theory}.

\begin{lemma} \label{lemma p logarithmic flat} The morphism $p$ is logarithmically flat, but not logarithmically smooth.\end{lemma}
\begin{proof} If $p$ was logarithmically smooth then by base change we would have that $\Xcal_0$ is logarithmically smooth, contradicting Lemma \ref{lem: X0 not logarithmically smooth}.

In order to show that $p$ is logarithmically flat, we will use the local chart criterion due to Gillam \cite[\S 1.1]{GillamLogFlat}. Choosing an open subset of $\Xcal$ of the form $\mathcal{U}=U \times \Aaff^{\! 1}$ for $U\subseteq X$ affine open, we have
\begin{equation*} \mathcal{U} = \Spec B \times \Aaff^{\! 1} = \Spec B[s] \end{equation*}
where $B$ is a regular ring. Taking $g(s)\in B[s]$ a local equation for $Z$, a local chart for the logarithmic structure is given by:
\begin{align*} P = \N & \to B[s] \\
1 & \mapsto g(s).\end{align*}
On the other hand, the base of $p$ is $\Aaff^{\! 1} = \Spec \kfield[s]$ with the trivial logarithmic structure, so a chart is given by the trivial map:
\begin{equation*} Q = 0 \to \kfield[s]. \end{equation*}
Then Gillam's local chart criterion says that $p$ is logarithmically flat if and only if the map
\begin{align*} \kfield[s][\N] = \kfield[s,w] & \to B[s,w^{\pm 1}] = B[\Z] \\
s & \mapsto s \\
w & \mapsto g(s) \cdot w \end{align*}
is flat. The corresponding map on schemes is
\begin{align*} \Spec B \times \Aaff^{\! 1}_s \times \mathbb{G}_{\op{m},w} \to \Aaff^{\! 1}_s \times \Aaff^{\! 1}_w \end{align*}
and the fibre over a point $(t,r) \in \Aaff^{\! 1}_s \times \Aaff^{\! 1}_w$ is
\begin{equation*} \mathrm{V}( g(t)\cdot w - r )\subseteq \Spec B \times \mathbb{G}_{\op{m},w} \end{equation*}
which always gives a divisor in $\Spec B \times \mathbb{G}_{\op{m},w}$. This is a morphism with equidimensional fibres between two smooth schemes, and hence is flat by miracle flatness.\end{proof}

\subsection{Logarithmic stable maps}\label{sec: logarithmic stable maps to Xcal} This paper is concerned with logarithmic Gromov--Witten theory in the genus zero maximal contacts setting. Consider therefore the following combinatorial data
\begin{equation*} \Gamma=(g,\beta,n,\alpha) = (0,\beta,1,(Z \cdot \beta)) \end{equation*}
where $\beta \in \mathsf{H}_2^+(X \times \Aaff^{\! 1}) = \mathsf{H}_2^+(X)$ is any curve class; $n=1$ is the number of marked points; and $\alpha=(Z \cdot \beta) = (Z_t \cdot \beta)$ (which does not depend on $t$) is the tangency order at the single marked point $x$. We may then consider the associated moduli space of logarithmic stable maps \cite{ChenLog,AbramovichChenLog,GrossSiebertLog}:
\begin{equation*} \Kup^{\log}(\Xcal).\end{equation*}
We suppress the fixed combinatorial data $\Gamma$ from the notation. Since $\Xcal$ is logarithmically smooth and $p \colon \Xcal \to \Aaff^{\! 1}$ is logarithmically flat, we may apply the degeneration formula (Theorems~\ref{prop: POT over A1} and \ref{thm: degeneration formula}) to conclude the existence of a perfect obstruction theory for:
\begin{equation*} \Kup^{\log}(\Xcal) \to \Log\Mfrak_{0,n}\times\Aaff^{\! 1}.\end{equation*}
This induces a family of virtual classes on the fibres of $q \colon \Kup^{\log}(\Xcal) \to \Aaff^{\! 1}$ satisfying the conservation of number principle.

\runningexample Given $\beta=d \in \mathsf{H}_2^+(\PP^2)$ the tangency order is $\alpha=(Z_t \cdot \beta) = (3d)$. For $t \neq 0$ the moduli space $\Kup^{\log}(\Xcal_t)$ consists of logarithmic stable maps to $(\PP^2,Z_t)$ with tangency order $3d$ at a single marking. It has virtual dimension zero, but in general is obstructed. The exception is the case $d=1$, when the moduli space consists of $9$ isolated points, corresponding to the $9$ flex lines of the smooth cubic $Z_t$. For $d \geq 2$ it contains higher-dimensional components corresponding to multiple covers and reducible curves.

\section{Obstruction bundle and virtual push-forward} \label{section computing virtual class}
\noindent From now on we make the following two assumptions concerning our hypersurface degeneration. Both are trivially satisfied in our main example.
\begin{assumption}\label{assumption: convexity of D} $\OO_X(D)$ is convex.\end{assumption}
\begin{assumption} \label{assumption factoring through the divisor}\label{assumption factoring through D} Consider any morphism $f \colon \PP^1 \to X$ of class $\beta^\prime$ with $0 < \beta^\prime \leq \beta$. If $f$ does not factor through $D$, then $f^{-1}(D)$ consists of at least two distinct points.\end{assumption}
\noindent The second assumption holds in particular whenever the components of $D$ are sufficiently positive and have empty total intersection.

Consider as above the moduli space of logarithmic stable maps to $\Xcal_0$. There is a forgetful morphism to the moduli space of ordinary stable maps to the underlying variety:
\begin{equation*} \iota \colon \Kup^{\log}(\Xcal_0) \to \Kup(X).\end{equation*}
In this section, we obtain a formula relating the virtual classes of these spaces:
\begin{theorem} \label{thm: virtual pushforward} There exists a vector bundle $F$ on $\Kup(X)$ which satisfies
\begin{equation}\label{eqn: virtual pushforward} \iota_\star [\Kup^{\log}(\Xcal_0)]^{\virt} = \e(F) \cap [\Kup(X)]^{\virt} \end{equation}
and
\begin{equation*} \e(F) \cdot \e(\LogOb) = \e(\pi_\star f^\star \OO_X(D)).\end{equation*}\end{theorem}
\noindent Such a push-forward result necessitates the comparison of logarithmic and non-logarithmic obstruction theories; one of the main contributions of this work is to explain how such a comparison can be made, and made rather explicitly, via computations on the Artin fans. The difference between the obstruction theories is captured in the $\LogOb$ term (see Definition \ref{def: LogOb}). In \S \ref{sec: computing LogOb} we show how to express $\LogOb$ in terms of tautological line bundles on $\Kup(X)$.

Theorem \ref{thm: virtual pushforward} reduces the logarithmic Gromov--Witten theory of $\Xcal_0$ to tautological integrals on the moduli space of ordinary stable maps to $X$. In \S \ref{sec: component calculations} we use functorial virtual localisation to calculate these integrals in our main example, determining the component contributions which refine the logarithmic Gromov--Witten invariants.

\begin{remark} \label{rmk: push forward abuse} Once we localise, the equivariant class $\e(\LogOb)$ will be invertible (see Theorem \ref{thm: LogOb pure weight}). This allows us to rewrite \eqref{eqn: virtual pushforward} as:
\begin{equation} \label{eqn: virtual pushforward fraction} \iota_\star [\Kup^{\log}(\Xcal_0)]^{\virt} = \left( \dfrac{ \e(\pi_\star f^\star \OO_X(D))}{\e(\LogOb)} \right) \cap [\Kup(X)]^{\virt}.\end{equation}
It is this formulation which we will use to carry out our calculations.
\end{remark}

\subsection{Central fibre moduli: factorisation through $D$}\label{sec: central fibre moduli space} We now investigate the central fibre of our family of moduli spaces:
\begin{equation*} \Kup^{\log}(\Xcal_0). \end{equation*}
The following result is a direct consequence of Assumption \ref{assumption factoring through D} and the maximal contacts setup: 
\begin{lemma} \label{lem: factor through D} The underlying morphism of any logarithmic stable map to $\Xcal_0$ factors through the divisor $D \subseteq X$, and hence the morphism forgetting the logarithmic structures
\begin{equation*}\Kup^{\log}(\Xcal_0) \to \Kup(X)\end{equation*}
factors through $\Kup(D) \hookrightarrow \Kup(X)$.\end{lemma}

\begin{remark} Throughout we understand $\Kup(X)$ to mean the moduli space of ordinary stable maps with associated combinatorial data $\underline{\Gamma} = (g,\beta,n) = (0,\beta,1)$, induced by the combinatorial data $\Gamma$ for the logarithmic moduli space (see \S \ref{sec: logarithmic stable maps to Xcal}). \end{remark}

\begin{proof}[Proof of Lemma \ref{lem: factor through D}] We first claim that it is sufficient to prove the assertion on the level of closed points. For consider a general family $f \colon \Ccal \to \Xcal_0$ of logarithmic stable maps over a base logarithmic scheme $\mathcal{S}$. By passing to an open cover, we may assume that $\mathcal{S}$ is atomic \cite[\S 2.2]{AbramovichWiseBirational}, and we let $Q = \Gamma(\mathcal{S},\ol{M}_{\mathcal{S}})$. Tropicalising, we obtain a family of tropical stable maps
\begin{equation} \label{diag: tropical family}
\begin{tikzcd}
\sqC \ar[r,"\mathsf{f}"] \ar[d,"\mathsf{p}" left] & \RR_{\geq 0} \\
\sigma & 
\end{tikzcd}
\end{equation}
over the base cone $\sigma = Q^\vee_{\RR} = \Hom_{\N}(Q,\RR_{\geq 0})$. The fibre of $\mathsf{p}$ over an interior point of the base cone gives the combinatorial type of the ``most degenerate'' fibre in the family $\Ccal \to \mathcal{S}$; that is, the fibre over the unique deepest stratum of the atomic logarithmic scheme $\mathcal{S}$.

The map $f \colon C \to X$ factors through $D$ if and only if $f^\star s_D=0$, where $s_D$ is a section of $\OO_X(D)$ cutting out $D$. This section corresponds to the identity piecewise-linear function on the tropical target $\RR_{\geq 0}$, and so $f^\star s_D=0$ if and only if $\mathsf{f} \colon \sqC \to \RR_{\geq 0}$ factors through $\RR_{>0} \subseteq \RR_{\geq 0}$ (after removing the unique vertex of the cone complex $\sqC$). Moreover this property is preserved under edge contractions, so it suffices to check it for the most degenerate fibre. Fixing a closed point $s$ in the deepest stratum of $\mathcal{S}$ we see that the tropicalisation of $\Ccal_s \to \Xcal_0$ is the same as the tropicalisation \eqref{diag: tropical family} of the family. Therefore it is sufficient to show that $\Ccal_s \to \Xcal_0$ factors through $D$.

Consider therefore a logarithmic stable map over a closed logarithmic point, with underlying morphism $f \colon (C,x) \to X$. Suppose for a contradiction that there is some irreducible component $C^\prime \subseteq C$ not mapped inside $D$. Since we must have $x \mapsto D$ it follows that there must be an irreducible component of $C$ not mapped into $D$ and not contracted to a point, so we may assume that $f$ is not constant on $C^\prime$. But then by Assumption \ref{assumption factoring through D} it follows that there are at least two points of $C^\prime$ which are mapped into $D$ by $f$. It is easy to see from this that the resulting tropical map cannot satisfy the balancing condition \cite[\S 1.4]{GrossSiebertLog}, since there is only one marked point and the genus of the source curve is zero.\end{proof}

\subsection{Central fibre moduli: tropical moduli and minimal monoid} \label{sec: tropical moduli and minimal monoid} The previous lemma implies a number of surprising and extremely useful facts concerning the minimal monoids appearing in the logarithmic structure on $\Kup^{\log}(\Xcal_0)$. Fix a logarithmic stable map
\begin{equation*} f \colon \Ccal \to \Xcal_0 \end{equation*}
over a logarithmic point $(\Speck, Q)$, where $Q$ is the corresponding minimal monoid \cite{GrossSiebertLog}. Tropicalising, we obtain a tropical stable map
\begin{equation}\label{tropical map} \mathsf{f} \colon \sqC \to \mathbb{R} _{\geq 0}\end{equation}
over $\sigma = Q^\vee_{\RR}$. Since the image of $f$ is contained inside $D$ it follows that the image of $\mathsf{f}$ is contained inside $\RR_{>0} \subseteq \RR_{\geq 0}$. An example of such a tropical stable map is illustrated below:
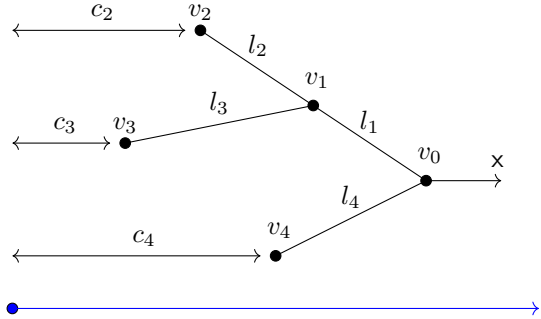
\begin{center}
\begin{tikzpicture}

\draw[fill=blue] (-2,0.3) circle[radius=2pt];
\draw[->,color=blue] (-2,0.3) to (5,0.3);

\draw[fill] (3.5,2) circle[radius=2pt];
\draw (3.55,2.1) node[above]{\small$v_0$};

\draw[->] (3.5,2) -- (4.5,2);
\draw (4.45,2.05) node[above]{\small$\mathsf{x}$};

\draw (2,3) -- (3.5,2);
\draw (2.75,2.5) node[above]{\small$l_1$};

\draw[fill] (2,3) circle[radius=2pt];
\draw (2.05,3.1) node[above]{\small$v_1$};

\draw (2,3) -- (0.5,4);
\draw (1,3.8) node[right]{\small$l_2$};

\draw[fill] (0.5,4) circle[radius=2pt];
\draw (0.5,4) node[above]{\small$v_2$};

\draw[<->] (0.3,4) -- (-2,4);
\draw (-0.8,4) node[above]{\small$c_2$};

\draw (2,3) -- (-0.5,2.5);
\draw (0.75,2.75) node[above]{\small$l_3$};

\draw[fill] (-0.5,2.5) circle[radius=2pt];
\draw (-0.5,2.5) node[above]{\small$v_3$};

\draw[<->] (-0.7,2.5) -- (-2,2.5);
\draw (-1.3,2.5) node[above]{\small$c_3$};

\draw (1.5,1) -- (3.5,2);
\draw (2.5,1.5) node[above]{\small$l_4$};

\draw[fill] (1.5,1) circle[radius=2pt];
\draw (1.55,1.1) node[above]{\small$v_4$};

\draw[<->] (-2,1) -- (1.3,1);
\draw (-0.25,1) node[above]{\small$c_4$};

\end{tikzpicture}	
\end{center}
The tropical parameters are indicated in the diagram; they consist of the source edge lengths $l_1,l_2,l_3,l_4$ and the target offsets $c_2,c_3,c_4$.

In general, let $r$ denote the number of edges and $m$ the number of leaves of $\sqC$, the latter of which are in bijective correspondence with the target offset parameters.
\begin{proposition}\label{prop: minimal monoid} There is a natural quotient morphism $\N^{r+m} \to Q$ defining the minimal monoid $Q$. The relations are linearly independent, so the associated closed embedding of toric varieties
\begin{equation*} U_Q = \Spec \kfield[Q] \hookrightarrow \Spec \kfield[\N^{r+m}] = \Aaff^{r+m} \end{equation*}
is a complete intersection. Moreover, $\dim U_Q = r+1$.\end{proposition}

\begin{remark} The fact that $U_Q$ is regularly embedded in a smooth toric variety will turn out to be crucial when we come to analyse the logarithmic deformation theory in \S \ref{sec: isolating LogOb}. \end{remark}

\begin{proof} In \cite[\S 1.5]{GrossSiebertLog}, the minimal monoid $Q$ is obtained by taking the free monoid generated by the tropical parameters and quotienting by the tropical continuity relations, with the caveat that one is required to saturate both the relations and the resulting quotient monoid. We will show that in our setting these additional saturation steps are unnecessary.

Let $v_0$ denote the vertex of $\sqC$ containing the marking $\mathsf{x}$. Suppose as above that $\sqC$ has $r$ edges $e_1,\ldots,e_r$, with associated lengths $l_{e_1},\ldots,l_{e_r}$ and expansion factors $\alpha_{e_1},\ldots,\alpha_{e_r}$, and $m$ leaves $v_1,\ldots,v_m$ with associated target offsets $c_1,\ldots,c_m$. For each leaf $v_i$ we have
\begin{equation*} c_i + \sum_{e} \alpha_e  l_e = \mathsf{f}(v_0) \end{equation*}
where the sum runs over the edges $e$ connecting $v_i$ to $v_0$. A complete set of tropical continuity relations is thus given by equations of the form
\begin{equation} \label{eqn: tropical continuity relations} c_i + \sum_{e} \alpha_e l_e = c_j + \sum_{e^\prime} \alpha_{e^\prime} l_{e^\prime} \end{equation}
for $i,j \in \{1,\ldots,m\}$ distinct. This gives $m-1$ independent relations, and since each equation involves two distinct target offsets appearing without multiplicity, it follows that this set of equations is saturated, linearly independent, and that the quotient monoid is saturated.

Since $Q$ is obtained as a quotient of a free monoid of rank $r+m$ by $m-1$ linearly independent relations, it has rank $r+1$ and so $\dim U_Q = r+1$ as claimed.
\end{proof}

\begin{corollary}\label{cor: pre-saturated monoid} The pre-saturated monoid $Q^{\mathrm{pre}}$ (also called the coarse monoid: see \cite[\S 3.7]{ChenLog}) is automatically saturated, so $Q^{\mathrm{pre}}=Q$.
	\end{corollary}
\begin{proof} Follows directly from the quotient description of $Q$ given above.\end{proof}

\begin{corollary} \label{cor: map Nr to Q injective} The natural composite $\N^r \to \N^{r+m} \to Q$ is injective.	
\end{corollary}
\begin{proof}Since both $\N^r$ and $Q$ are integral, it suffices to show that the map on groupifications is injective. The equations \eqref{eqn: tropical continuity relations} above imply that $Q^{\operatorname{gp}} = \Z^r \times \Z$ with final generator	given by a single offset parameter. The map on groupifications is then	 $\Z^r \hookrightarrow \Z^r \times \Z$ which is clearly injective.
\end{proof}

\subsection{Central fibre moduli: identification with maps to $D$.} The following result, an improvement upon Lemma~\ref{lem: factor through D}, will allow us to precisely describe the geometry of the central fibre moduli space. The basic idea is that the factorisation of the stable map through $D$ trivialises the problem of comparing the line bundles and sections encoded in the logarithmic structure, ensuring existence and uniqueness of logarithmic lifts.
\begin{proposition} \label{prop: identification of moduli spaces} The morphism
\begin{equation*} \iota \colon \Kup^{\log}(\Xcal_0) \to \Kup(D) \end{equation*}
forgetting the logarithmic structures is an isomorphism of stacks over schemes. \end{proposition}

\begin{proof} We will show that $\iota$ is \'etale and bijective on geometric points. Since it is also representable \cite[Theorem 1.2.1]{ChenLog}, this is enough to conclude that it is an isomorphism \cite[Tag 02LC]{Stacks}.

It follows immediately from Corollary \ref{cor: pre-saturated monoid} that $\iota$ is injective on geometric points (see \cite[\S 3.7]{ChenLog}). To show surjectivity, consider an ordinary stable map
\begin{equation*} f \colon C \to D \end{equation*}
over $\Speck$. Since we are in genus zero, there is a unique assignment of expansion factors to the nodes of $C$ such that the balancing condition is satisfied. This discrete data produces a minimal monoid $Q$ together with a morphism $\N^r \to Q$ where $r$ is the number of nodes of $C$.

Choose a logarithmic morphism $(\Speck,Q) \to (\Speck,\N^r)$ extending the morphism $\N^r \to Q$, and consider the logarithmically smooth curve
\begin{equation*} \Ccal \to (\Speck,Q) \end{equation*}
obtained by pulling back the minimal logarithmically smooth curve over $(\Speck,\N^r)$. The morphism $(\Speck,Q) \to (\Speck,\N^r)$ involves choices, but crucially since $\N^r \to Q$ is injective (Corollary~\ref{cor: map Nr to Q injective}) these choices  all differ by automorphisms of $(\Speck,Q)$. Hence we obtain a unique logarithmic curve $\Ccal$ up to isomorphism.

It remains to enhance the morphism $C \to D \to X$ to a logarithmic morphism $\Ccal \to \Xcal_0$. The morphism
\begin{equation*} f^\flat \colon f^{-1} \ol{M}_{\Xcal_0} \to \ol{M}_{\Ccal} \end{equation*}
on the level of ghost sheaves is uniquely determined by the tropical combinatorics. Letting $1\in\ol{M}_{\Xcal_0}$ denote the unique generator, we may interpret $f^\flat 1$ as a piecewise-linear function on $\sqC$ with values in $Q$ \cite[Remark 7.3]{CavalieriChanUlirschWise}. The associated line bundle $\OO_{\Ccal}(f^\flat 1)$ restricted to each component $C^\prime \subseteq C$ is a sum over adjacent nodes and markings, weighted by expansion factors. By the balancing condition this has the same degree, and hence is isomorphic to, $f^\star \OO_X(D)|_{C^\prime}$. Since $C$ is genus zero, these component-wise isomorphisms patch to give a global isomorphism:
\begin{equation*} f^\star \OO_X(D) \cong \OO_{\Ccal}(f^\flat 1).\end{equation*}
Moreover, the associated sections vanish on both sides. On the left-hand side, this is because the curve is mapped entirely inside the divisor. On the right-hand side, this is because $\sqC$ maps entirely inside $\RR_{>0}$ and so the piecewise-linear function is nonzero on every vertex of the tropical curve, which implies that the associated section vanishes, see e.g. \cite[Propositon~2.4.1]{RanganathanSantosParkerWise1}. This identification of line bundles and sections gives a logarithmic enhancement $f \colon \Ccal \to \Xcal_0$. This shows that $\iota$ is surjective on geometric points.

Finally, we need to show that $\iota$ is \'etale. Consider therefore a square-zero lifting problem for schemes:
\bcd
S \ar[r] \ar[d,hook] & \Kup^{\log}(\Xcal_0) \ar[d] \\
S^\prime \ar[r] \ar[ru,dashed] & \Kup(D).
\ecd
This is equivalent to the data of a logarithmic stable map to $\Xcal_0$ over $S$ and an ordinary stable map to $D$ over $S^\prime$:
\bcd
(C,M_C) \ar[r] \ar[d] & (\Xcal_0,M_{\Xcal_0}) & \qquad & C^\prime \ar[r] \ar[d] & D \\
(S,M_S), & \ & \ & S^\prime. &
\ecd
We begin by constructing the logarithmic structure on $S^\prime$. Note that the underlying topological spaces for $S$ and $S^\prime$ are identical. We will produce local charts for $S^\prime$ using the local charts for $S$. Replace $S$ by an open subset which admits a chart $Q \to \OO_S$, where $Q$ is the minimal monoid described in \S \ref{sec: tropical moduli and minimal monoid}. Our aim is to produce a natural lift:
\bcd
\, & \OO_{S^\prime} \ar[d] \\
Q \ar[r] \ar[ru,dashed] & \OO_S.
\ecd
Recall from Proposition~\ref{prop: minimal monoid} that $Q$ arises as a quotient $\N^{r+m} \to Q$ where $r$ is the number of nodes of the source curve and $m$ is the number of target offset parameters. The prestable curve $C^\prime \to S^\prime$ has an associated minimal logarithmic structure pulled back from $\Mfrak_{0,1}$, which defines a morphism:
\begin{equation*} \N^r \to \OO_{S^\prime}. \end{equation*}
On the other hand, the target offsets $c_1,\ldots,c_m$ must map to zero in $\OO_{S^\prime}$ since the curve is always mapped inside the divisor. We therefore obtain a unique map $\N^{r+m} \to \OO_{S^\prime}$ given simply as the composite:
\begin{equation*} \N^{r+m} \to \N^r \to \OO_{S^\prime}. \end{equation*}
The quotient $\N^{r+m} \to Q$ is defined by equations of the form \eqref{eqn: tropical continuity relations}. Since each $c_i$ is mapped to zero in $\OO_{S^\prime}$ it follows that both sides of these equations are sent to zero under the map $\N^{r+m} \to \OO_{S^\prime}$. Hence this map descends to the quotient, giving a chart
\begin{equation*} Q \to \OO_{S^\prime} \end{equation*}
as required. These local constructions are consistent with respect to generisation of the monoid, and therefore glue to produce a strict square-zero extension:
\begin{equation*} (S,M_S) \to (S^\prime,M_{S^\prime}).\end{equation*}
We now choose a logarithmic enhancement $(S^\prime,M_{S^\prime}) \to (\Mfrak_{0,1},M_{\Mfrak_{0,1}})$ of the morphism $S^\prime \to \Mfrak_{0,1}$, and pull back the universal curve to obtain a logarithmic enhancement of the source curve:
\begin{equation*} (C^\prime,M_{C^\prime}) \to (S^\prime,M_{S^\prime}),\end{equation*}
Again the logarithmic enhancement on the base involves choices, but since $\N^r \to Q$ is injective (Corollary~\ref{cor: map Nr to Q injective}) all such choices differ by the automorphisms of $M_{S^\prime}$. Hence we obtain a unique logarithmic curve up to isomorphism.

Finally, to obtain a logarithmic map $(C^\prime,M_{C^\prime}) \to \Xcal_0$ we first observe that $C^\prime$ and $C$ have the same underlying topological space, and therefore that there is an equality of ghost sheaves $\ol{M}_{C^\prime} = \ol{M}_C$. This gives us the logarithmic enhancement on the level of ghost sheaves, and completing this to a full logarithmic enhancement proceeds via similar arguments to those used above to prove that $\iota$ is surjective on geometric points. We have thus constructed a unique lift $S^\prime \to \Kup^{\log}(\Xcal_0)$, which completes the proof.\end{proof}

\begin{corollary}\label{cor: iota closed embedding} The morphism $\iota \colon \Kup^{\log}(\Xcal_0) \to \Kup(X)$ is a closed embedding.\end{corollary}

\begin{remark} The inclusion $\Kup(D) \hookrightarrow \Kup(X)$ furnishes the domain with a virtual fundamental class whose pushforward to $\Kup(X)$ is:
\[ \e(\pi_\star f^\star \OO_X(D)) \cap [\Kup(X)]^{\virt}.\]
Although the previous lemma establishes an isomorphism $\Kup^{\log}(\Xcal_0)=\Kup(D)$, the virtual classes on the two spaces do not coincide and even have different dimensions. Theorem~\ref{thm: virtual pushforward} relates the two classes, identifying the difference as $\e(\LogOb)$. \end{remark}

\subsection{Target geometry} We derive a fundamental exact triangle on the target space $\Xcal_0$, which will be useful later.
\begin{proposition}\label{prop: exact triangle on X0} We have the following exact triangle on $\Xcal_0$:	
\begin{equation}\label{eqn: important triangle on X0} \TTlog_{\Xcal_0} \to \T_X \to  \OO_D(D) \xrightarrow{[1]}. \end{equation}
\end{proposition}

\begin{proof}
Consider the logarithmically smooth scheme $\Xcal=(X\times\Aaff^{\! 1},Z)$. We have a short exact sequence of sheaves:
\begin{equation*} 0 \to \Tlog_{\Xcal} \to \T_{X \times \Aaff^{\! 1}} \to \OO_Z(Z) \to 0.\end{equation*}
Pulling back along the inclusion $i_0 \colon \Xcal_0 \hookrightarrow \Xcal$ we obtain an exact triangle:
\begin{equation*} i_0^\star \Tlog_{\Xcal} \to i_0^\star \T_{X\times \Aaff^{\! 1}} \to \Lder i_0^\star \OO_Z(Z) \xrightarrow{[1]}.\end{equation*}
By taking the standard resolution of $\OO_Z(Z)$ it is easy to see that $\Lder i_0^\star \OO_Z(Z)=\OO_D(D)$, so in fact we have the exact triangle:
\begin{equation}\label{eqn: exact triangle 1} i_0^\star \Tlog_{\Xcal} \to i_0^\star \T_{X\times \Aaff^{\! 1}} \to \OO_D(D) \xrightarrow{[1]}.\end{equation}
Consider on the other hand the morphism $p \colon \Xcal \to \Aaff^{\! 1}$. Since $p$ is logarithmically flat, we have by base change \cite[1.1(iv)]{OlssonCotangent}:
\begin{equation*} \TTlog_{\Xcal_0} = \Lder i_0^\star \TTlog_{p}. \end{equation*}
The morphism $p$ is integral and therefore satisfies Olsson's condition $(T)$ \cite[1.3]{OlssonCotangent}. Hence $\TTlog_{p}$ fits into an exact triangle \cite[1.1(v)]{OlssonCotangent}:
\begin{equation*} \TTlog_p \to \Tlog_{\Xcal} \to p^\star \T_{\Aaff^{\! 1}} \xrightarrow{[1]}.\end{equation*}
Applying $\Lder i_0^\star$ to this, we obtain
\begin{equation*} \TTlog_{\Xcal_0} \to i_0^\star \Tlog_{\Xcal} \to i_0^\star \ p^\star \T_{\Aaff^{\! 1}} \xrightarrow{[1]}\end{equation*}
which we rotate to give:
\begin{equation} \label{eqn: exact triangle 2}
i_0^\star \Tlog_{\Xcal} \to i_0^\star \ p^\star \T_{\Aaff^{\! 1}} \to \TTlog_{\Xcal_0}[1] \xrightarrow{[1]}.\end{equation}
Finally, we have the following exact sequence, again better thought of as an exact triangle:
\begin{equation} \label{eqn: exact triangle 3} i_0^\star \T_{X \times \Aaff^{\! 1}} \to i_0^\star \ p^\star \T_{\Aaff^{\! 1}} \to \T_{X}[1] \xrightarrow{[1]}.\end{equation}
Since the logarithmic morphism $p \colon \Xcal \to \Aaff^{\! 1}$ factors through $X \times \Aaff^{\! 1}$ with the trivial logarithmic structure, the composite of the first arrow in \eqref{eqn: exact triangle 1} with the first arrow in \eqref{eqn: exact triangle 3} yields the first arrow in \eqref{eqn: exact triangle 2}. Applying the octahedral axiom to \eqref{eqn: exact triangle 1}, \eqref{eqn: exact triangle 2}, \eqref{eqn: exact triangle 3}, we obtain \eqref{eqn: important triangle on X0} as required.
\end{proof}

\begin{remark} The connecting homomorphism $i_0^\star \ p^\star \T_{\Aaff^{\! 1}} \to \T_X[1]$ in \eqref{eqn: exact triangle 3} is zero since the triangle is split. However, there does not appear to be any way of using this fact to infer a splitting of \eqref{eqn: important triangle on X0}.\end{remark}

\subsection{Isolating $\LogOb$}\label{sec: isolating LogOb} Having established the basic properties of the logarithmic target and the moduli space, we start working towards the virtual push-forward Theorem~\ref{thm: virtual pushforward}.

The basic difficulty in obtaining a virtual push-forward result for $\iota$ is that the obstruction theories for the source and the target are defined with respect to different bases: $\Log\Mfrak_{0,1}$ and $\Mfrak_{0,1}$, respectively. In this section, we address with this issue by producing a perfect obstruction theory for the morphism
\begin{equation*} \psi \colon \Kup^{\log}(\Xcal_0) \to \Mfrak_{0,1} \end{equation*}
thereby extracting the ``logarithmic part'' of the obstruction theory. 

\begin{theorem}\label{thm: isolating LogOb} There is a perfect obstruction theory for the morphism $\psi$ in the following diagram:
\begin{equation}\label{eqn: diagram Klog LogM M}
\begin{tikzcd}
\Kup^{\log}(\Xcal_0) \ar[r,"\varphi=\rho_0"] \ar[rr,bend right=20pt,"\psi" below] & \Log\Mfrak_{0,1} \ar[r] & \Mfrak_{0,1}.
\end{tikzcd}
\end{equation}
This obstruction theory fits into an exact triangle
\[ \varphi^\star \LL_{\Log \Mfrak_{0,1}|\Mfrak_{0,1}} \to \EE_{\Kup^{\log}(\Xcal_0)|\Mfrak_{0,1}} \to \EE_{\Kup^{\log}(\Xcal_0)|\Log\Mfrak_{0,1}} \xrightarrow{{[1]}} \]
and produces a virtual fundamental class on $\Kup^{\log}(\Xcal_0)$ which agrees with the virtual fundamental class constructed in \S \ref{sec: degeneration formula}:
\[ \psi^! [ \Mfrak_{0,1}] = \varphi^! [\Log \Mfrak_{0,1}].\]
\end{theorem}

\begin{remark} As is commonplace when dealing with Artin stacks of infinite type, the above result holds only after replacing $\Log\Mfrak_{0,1}$ by a suitable union of smooth charts, which we now describe. Recall that $\Log\Mfrak_{0,1}$ has a smooth cover by stacks of the form
\begin{equation}\label{eqn: local model for LogM01} V \times_{\Acal_P} \Acal_Q \end{equation}
where $V \to \Mfrak_{0,1}$ is a smooth chart, $P$ is a monoid giving a local chart $P \to \OO_V$ for the logarithmic structure on $\Mfrak_{0,1}$, and $P \to Q$ is any morphism of monoids \cite[Corollary~5.25 and Remark~5.26]{OlssonLogStacks}. 

The morphism $\varphi$ factors through the smooth cover consisting of those stacks of the form \eqref{eqn: local model for LogM01} in which $Q$ is the minimal monoid associated to a logarithmic stable map to $\Xcal_0$. To be more precise: given a closed point $\xi \in \Kup^{\log}(\Xcal_0)$ there is an associated minimal monoid $Q$ giving a local chart for the logarithmic structure around $\xi$, and another minimal monoid $P$ giving a local chart for the logarithmic structure around $\psi(\xi) \in \Mfrak_{0,1}$, together with a natural morphism $P \to Q$. After choosing a suitable atomic neighbourhood $V$ for $\psi(\xi)$, we have a local factorisation of $\varphi$ through the associated chart:
\begin{equation*} \psi^{-1}(V) \to V \times_{\Acal_P} \Acal_Q \to \Log\Mfrak_{0,1}.\end{equation*}
From now on, therefore, we replace $\Log\Mfrak_{0,1}$ by the image of the charts described above. \end{remark}

We begin with a fundamental result on the geometry of the morphism $\Log\Mfrak_{0,1} \to \Mfrak_{0,1}$. The proof makes crucial use of the properties of the minimal monoid $Q$ established in \S \ref{sec: tropical moduli and minimal monoid}.

\begin{proposition} \label{prop: Lchi in -1,1} $\LL_{\Log\Mfrak_{0,1}|\Mfrak_{0,1}}$ is of perfect amplitude contained in $[-1,1]$.\end{proposition}
\begin{proof} The assertion is local, so we may replace $\Log\Mfrak_{0,1}$ and $\Mfrak_{0,1}$ by suitable smooth charts, as described above, to obtain a diagram:
\begin{equation}
\begin{tikzcd}
\Log\Mfrak_{0,1} \ar[r] \ar[d] \ar[rd,phantom,"\square"] & \Mfrak_{0,1} \ar[d] \\
\Acal_Q \ar[r] & \Acal_P.
\end{tikzcd}	
\end{equation}
The map $\Mfrak_{0,1} \to \Acal_P$ to the Artin fan is smooth because $\Mfrak_{0,1}$ is logarithmically smooth. Therefore by flat base change we have
$$\LL_{\Log\Mfrak_{0,1}|\Mfrak_{0,1}} = \LL_{\Acal_Q|\Acal_P}$$
(we have suppressed the pullback from the notation). Hence, we obtain an exact triangle
\begin{equation}\label{eqn: exact triangle for Lchi} \LL_{\Acal_P} \to \LL_{\Acal_Q} \to \LL_{\Log\Mfrak_{0,1}|\Mfrak_{0,1}} \xrightarrow{[1]}\end{equation}
(again suppressing pullbacks). Since the Artin cones are global smooth quotients, their cotangent complexes may be described easily in terms of the prequotients using the equivariance of the exact triangles associated to the quotient morphism, see e.g. \cite[p.4]{BehrendDeRham}. Starting with $P$, we have $P=\N^r$ where $r$ is the number of nodes of the source curve. Consequently the prequotient $U_P = \Aaff^r$ is smooth, and $\LL_{\Acal_P}$ has an explicit two-term resolution, which may be expressed as a $T_P$-equivariant complex of vector bundles on $U_P$:
\begin{equation*} \LL_{\Acal_P} = [ \Omega_{U_P} \to (\mathrm{Lie}\, T_P)^\vee \otimes \OO_{U_P} ].\end{equation*}
Thus, $\LL_{\Acal_P}$ has perfect amplitude contained in $[0,1]$. The monoid $Q$, on the other hand, is not always free, and as such $U_Q$ is not always smooth. However, we have seen in Proposition~\ref{prop: minimal monoid} that $U_Q$ is always regularly embedded inside an affine space $\Aaff^{r+m}$, and using this we obtain a natural three-term resolution of $\LL_{\Acal_Q}$ by vector bundles
\begin{equation*} \LL_{\Acal_Q} = [\mathrm{N}^\vee_{U_Q|\Aaff^{r+m}} \to \Omega_{\Aaff^{r+m}}|_{U_Q} \to (\mathrm{Lie}\, T_Q)^\vee \otimes \OO_{U_Q} ]\end{equation*}
demonstrating that $\LL_{\Acal_Q}$ has perfect amplitude contained in $[-1,1]$. From these two facts and the exact sequence \eqref{eqn: exact triangle for Lchi}, it is easy to show (by the same argument as in the proof of Theorem~\ref{prop: POT over A1}) that $\LL_{\Log\Mfrak_{0,1}|\Mfrak_{0,1}}$ has perfect amplitude contained in $[-1,1]$. \end{proof}

\begin{remark} The exact triangle \eqref{eqn: exact triangle for Lchi} obtained in the above proof allows us to express the cohomology sheaves of $\LL_{\Log\Mfrak_{0,1}|\Mfrak_{0,1}}$ in terms of pullbacks of toric line bundles from the Artin fan. By definition, these pullbacks are the line bundles associated to piecewise-linear functions on the tropicalisation of the moduli space. This will prove to be a crucial computational tool.\end{remark}

\begin{proof}[Proof of Theorem \ref{thm: isolating LogOb}] Consider the composite
\begin{equation*} \EE_{\Kup^{\log}(\Xcal_0)|\Log\Mfrak_{0,1}}[-1] \to \LL_{\Kup^{\log}(\Xcal_0)|\Log\Mfrak_{0,1}}[-1] \to \varphi^\star \LL_{\Log\Mfrak_{0,1}|\Mfrak_{0,1}}\end{equation*}
and denote the cone of this morphism by $\EE_{\Kup^{\log}(\Xcal_0)|\Mfrak_{0,1}}$. By the axioms of a triangulated category, we obtain a morphism of exact triangles:
\begin{equation}\label{eqn: morphism of exact triangles comparing obstruction theories}
\begin{tikzcd}
\varphi^\star \LL_{\Log\Mfrak_{0,1}|\Mfrak_{0,1}} \ar[r] \ar[d,equals] & \EE_{\Kup^{\log}(\Xcal_0)|\Mfrak_{0,1}} \ar[r] \ar[d] & \EE_{\Kup^{\log}(\Xcal_0)|\Log\Mfrak_{0,1}} \ar[d] \ar[r,"{[1]}"] & \, \\
\varphi^\star \LL_{\Log\Mfrak_{0,1}|\Mfrak_{0,1}} \ar[r] & \LL_{\Kup^{\log}(\Xcal_0)|\Mfrak_{0,1}} \ar[r] & \LL_{\Kup^{\log}(\Xcal_0)|\Log\Mfrak_{0,1}} \ar[r,"{[1]}"] & .
\end{tikzcd}
\end{equation}
We know that $\EE_{\Kup^{\log}(\Xcal_0)|\Log\Mfrak_{0,1}}$ is of perfect amplitude contained in $[-1,0]$, and Proposition~\ref{prop: Lchi in -1,1} shows that $\varphi^\star \LL_{\Log\Mfrak_{0,1}|\Mfrak_{0,1}}$ is of perfect amplitude contained in $[-1,1]$. From this it is easy to show (by the same argument as in the proof of Proposition \ref{prop: POT over A1}) that
\[ \EE_{\Kup^{\log}(\Xcal_0)|\Mfrak_{0,1}}\] 
is of perfect amplitude contained in $[-1,1]$. Several applications of the Four Lemma imply that the morphism
\begin{equation*} \EE_{\Kup^{\log}(\Xcal_0)|\Mfrak_{0,1}} \to \LL_{\Kup^{\log}(\Xcal_0)|\Mfrak_{0,1}}\end{equation*}
is surjective on $\HH^{-1}$ and an isomorphism on both $\HH^0$ and $\HH^1$. But
\[ \HH^1 \big(\LL_{\Kup^{\log}(\Xcal_0)|\Mfrak_{0,1}}\big) = 0\]
since the map is representable, so we conclude that
\[ \HH^1\big(\EE_{\Kup^{\log}(\Xcal_0)|\Mfrak_{0,1}}\big)=0.\]
Thus, $\EE_{\Kup^{\log}(\Xcal_0)|\Mfrak_{0,1}}$ is of perfect amplitude contained in $[-1,0]$, and defines a perfect obstruction theory which fits into an exact triangle:
\begin{equation*} \varphi^\star \LL_{\Log\Mfrak_{0,1}|\Mfrak_{0,1}} \to \EE_{\Kup^{\log}(\Xcal_0)|\Mfrak_{0,1}} \to \EE_{\Kup^{\log}(\Xcal_0)|\Log\Mfrak_{0,1}}.\end{equation*}
It remains to show that this induces the same virtual class as the obstruction theory over $\Log\Mfrak_{0,1}$. There is a morphism of vector bundle stacks over $\Kup^{\log}(\Xcal_0)$
\[ \mathfrak{E}_{\Kup^{\log}(\Xcal_0)|\Log\Mfrak_{0,1}} \xrightarrow{\kappa} \mathfrak{E}_{\Kup^{\log}(\Xcal_0)|\Mfrak_{0,1}}\]
and \eqref{eqn: morphism of exact triangles comparing obstruction theories} gives an identification:
\[ \mathfrak{C}_{\Kup^{\log}(\Xcal_0)|\Log\Mfrak_{0,1}} = \mathfrak{C}_{\Kup^{\log}(\Xcal_0)|\Mfrak_{0,1}} \times_{\mathfrak{E}_{\Kup^{\log}(\Xcal_0)|\Mfrak_{0,1}}} \mathfrak{E}_{\Kup^{\log}(\Xcal_0)|\Log\Mfrak_{0,1}}.\]
From this we obtain the equality:
\begin{align*}
	\varphi^![\Log\Mfrak_{0,1}] & = 0^!_{\mathfrak{E}_{\Kup^{\log}(\Xcal_0)|\Log\Mfrak_{0,1}}}[\mathfrak{C}_{\Kup^{\log}(\Xcal_0)|\Log\Mfrak_{0,1}}] \\
	& = 0^!_{\mathfrak{E}_{\Kup^{\log}(\Xcal_0)|\Log\Mfrak_{0,1}}} \kappa^! [\mathfrak{C}_{\Kup^{\log}(\Xcal_0)|\Mfrak_{0,1}}] \\
	& = 0^!_{\mathfrak{E}_{\Kup^{\log}(\Xcal_0)|\Mfrak_{0,1}}}[\mathfrak{C}_{\Kup^{\log}(\Xcal_0)|\Mfrak_{0,1}}] \\
	& = \psi^![\Mfrak_{0,1}].\qedhere
\end{align*}
\end{proof}

In anticipation of the computations to follow, we take this opportunity to introduce some useful notation.
\begin{definition}Given a bounded complex $\FF$, we write $\chi(\FF)$ to denote the alternating sum of its cohomology sheaves
\begin{equation*} \chi(\FF) = \sum_{i \in \mathbb{Z}} (-1)^i \cdot \HH^i(\FF) \end{equation*}	
viewed as a class in K-theory. This mild abuse of notation should not lead to any confusion.
\end{definition}

\begin{definition}\label{def: LogOb} We let $\LogOb$ denote the following K-theory class on $\Kup^{\log}(\Xcal_0)$:
\begin{equation*} \LogOb = \chi(\varphi^\star \TT_{\Log\Mfrak_{0,1}|\Mfrak_{0,1}}) + \OO.\end{equation*}
Locally, therefore, $\LogOb$ may be expressed as:
\begin{equation}\label{eqn: basic formula for LogOb} \LogOb = \chi(\TT_{\Acal_Q}) - \chi(\TT_{\Acal_P}) + \OO.\end{equation}
The extra trivial bundle term is to account for the difference in rank between the monoids $Q$ and $P$ (see \S \ref{sec: computing LogOb}). We note that $\LogOb$ has constant rank $1$.\end{definition}
The computations of the following section will only depend on the Euler class of $\LogOb$. This is why we focus on the K-theory class.

\subsection{Building the obstruction bundle} Having constructed an obstruction theory for $\Kup^{\log}(\Xcal_0)$ over $\Mfrak_{0,1}$, we now compare it to the obstruction theory for the space of ordinary stable maps.
\begin{proposition}\label{prop: pushforward compatible triple} There exists a compatible triple of perfect obstruction theories for the diagram
\begin{equation} \label{diag: final triangle of spaces}
\begin{tikzcd}
\Kup^{\log}(\Xcal_0) \ar[r,"\iota"] \ar[rr,bend right=20pt,"\psi" below] & \Kup(X) \ar[r] & \Mfrak_{0,1}
\end{tikzcd}
\end{equation}
in which $\EE_{\Kup^{\log}(\Xcal_0)|\Kup(X)}$ is given by a vector bundle supported in degree $-1$.\end{proposition}
\begin{proof} We will first build the obstruction theory for the morphism $\iota$, and after show that it is given by a vector bundle supported in degree $-1$. Recall from Proposition~\ref{prop: exact triangle on X0} that there is an exact triangle on $\Xcal_0$:
\begin{equation*} \TTlog_{\Xcal_0} \to \T_X \to  \OO_D(D) \xrightarrow{[1]}. \end{equation*}
Applying $\Rder \pi_\star \Lder f^\star$ and dualising and shifting the result, we obtain:
\begin{equation} \label{eqn: second exact triangle 1} \EE_{\Kup^{\log}(\Xcal_0)|\Log\Mfrak_{0,1}}[-1] \to \left( \Rder \pi_\star \Lder f^\star \OO_D(D) \right)^\vee \to \iota^\star \EE_{\Kup(X)|\Mfrak_{0,1}} \xrightarrow{[1]}.\end{equation}
Here we have used the explicit expression for the first term given by Lemma~\ref{lem: obstruction theory on fibre}. On the other hand, we have from Theorem \ref{thm: isolating LogOb} an exact triangle:
\begin{equation} \label{eqn: second exact triangle 2} \EE_{\Kup^{\log}(\Xcal_0)|\Log\Mfrak_{0,1}}[-1] \to \varphi^\star \LL_{\Log\Mfrak_{0,1}|\Mfrak_{0,1}} \to \EE_{\Kup^{\log}(\Xcal_0)|\Mfrak_{0,1}} \xrightarrow{[1]}.\end{equation}
Our goal is to construct a complex $\EE_{\Kup^{\log}(\Xcal_0)|\Kup(X)}$ fitting into an exact triangle:
\begin{equation} \label{eqn: second exact triangle 3} (\Rder\pi_\star \Lder f^\star \OO_D(D))^\vee \to \varphi^\star \LL_{\Log\Mfrak_{0,1}|\Mfrak_{0,1}} \to \EE_{\Kup^{\log}(\Xcal_0)|\Kup(X)} \xrightarrow{[1]}. \end{equation}
Once this is achieved, we will obtain the desired compatible triple by applying the octahedral axiom to \eqref{eqn: second exact triangle 1}, \eqref{eqn: second exact triangle 2}, \eqref{eqn: second exact triangle 3}. To construct \eqref{eqn: second exact triangle 3}, we will construct a  morphism
\begin{equation}\label{eqn: map from OD(D) to LogOb} (\Rder\pi_\star \Lder f^\star \OO_D(D))^\vee \to \varphi^\star \LL_{\Log\Mfrak_{0,1}|\Mfrak_{0,1}}\end{equation}
and then take the mapping cone. By the axioms of a triangulated category, it is enough to construct a morphism:
\begin{equation*}\iota^\star \EE_{\Kup(X)|\Mfrak_{0,1}} \to \EE_{\Kup^{\log}(\Xcal_0)|\Mfrak_{0,1}}. \end{equation*}
From $\TTlog_{\Xcal_0} \to \T_X$ we obtain the following morphism, which has already appeared in the exact triangle \eqref{eqn: second exact triangle 1}:
\begin{equation*} \iota^\star \EE_{\Kup(X)|\Mfrak_{0,1}} \to \EE_{\Kup^{\log}(\Xcal_0)|\Log\Mfrak_{0,1}}.\end{equation*}
From the exactness of \eqref{eqn: second exact triangle 2}, it follows that this morphism factors through $\EE_{\Kup^{\log}(\Xcal_0)|\Mfrak_{0,1}}$ if and only if the following composition is zero:
\begin{equation*} \iota^\star \EE_{\Kup(X)|\Mfrak_{0,1}} \to \EE_{\Kup^{\log}(\Xcal_0)|\Log\Mfrak_{0,1}} \to \varphi^\star \LL_{\Log\Mfrak_{0,1}|\Mfrak_{0,1}}[1]. \end{equation*}
Recall from the proof of Theorem~\ref{thm: isolating LogOb} that this factors as:
\begin{equation} \label{eqn: factorisation of map 2} \iota^\star \EE_{\Kup(X)|\Mfrak_{0,1}} \to \EE_{\Kup^{\log}(\Xcal_0)|\Log\Mfrak_{0,1}} \to \LL_{\Kup^{\log}(\Xcal_0)|\Log\Mfrak_{0,1}} \to \varphi^\star \LL_{\Log\Mfrak_{0,1}|\Mfrak_{0,1}}[1]. \end{equation}
Now consider the following commuting diagram of moduli spaces:
\begin{equation} \label{eqn: moduli square in proof of obstruction bundle}
\begin{tikzcd}
\Kup^{\log}(\Xcal_0) \ar[r,"\iota"] \ar[d,"\varphi"] & \Kup(X) \ar[d] \\
\Log\Mfrak_{0,1} \ar[r] & \Mfrak_{0,1}.
\end{tikzcd}
\end{equation}
Associated to this diagram is the following square: 
\begin{equation} \label{eqn: obstruction theory square in proof of obstruction bundle}
\begin{tikzcd}
\iota^\star \EE_{\Kup(X)|\Mfrak_{0,1}} \ar[r] \ar[d] & \EE_{\Kup^{\log}(\Xcal_0)|\Log\Mfrak_{0,1}} \ar[d] \\
\iota^\star \LL_{\Kup(X)|\Mfrak_{0,1}} \ar[r] & \LL_{\Kup^{\log}(\Xcal_0)|\Log\Mfrak_{0,1}}.
\end{tikzcd}
\end{equation}
This square is commutative. This follows directly from the construction of the perfect obstruction theory for logarithmic stable maps \cite[\S 5]{GrossSiebertLog} applied to the targets $X$ (with trivial logarithmic structure) and $\Xcal$ (with divisorial logarithmic structure), restricting the latter to the central fibre. Using this, we may refactor \eqref{eqn: factorisation of map 2} as:
\begin{equation} \label{eqn: factorisation of map 1} \iota^\star \EE_{\Kup(X)|\Mfrak_{0,1}} \to \iota^\star \LL_{\Kup(X)|\Mfrak_{0,1}} \to \LL_{\Kup^{\log}(\Xcal_0)|\Log\Mfrak_{0,1}} \to \varphi^\star \LL_{\Log\Mfrak_{0,1}|\Mfrak_{0,1}}[1].\end{equation}
Finally, from \eqref{eqn: moduli square in proof of obstruction bundle} we obtain a pair of interlocking exact triangles
\[
\begin{tikzcd}
\iota^\star \LL_{\Kup(X)|\Mfrak_{0,1}} \ar[r] & \LL_{\Kup^{\log}(\Xcal_0)|\Mfrak_{0,1}} \ar[r] \ar[d,"="] & \LL_{\Kup^{\log}(\Xcal_0)|\Kup(X)} \ar[r,"{[1]}"] & \, \\
\varphi^\star \LL_{\Log\Mfrak_{0,1}|\Mfrak_{0,1}} \ar[r] & \LL_{\Kup^{\log}(\Xcal_0)|\Mfrak_{0,1}} \ar[r] & \LL_{\Kup^{\log}(\Xcal_0)|\Log\Mfrak_{0,1}} \ar[r,"{[1]}"] &  \,
\end{tikzcd}
\]
which allow us to factor \eqref{eqn: factorisation of map 1} as:
\[ \iota^\star \EE_{\Kup(X)|\Mfrak_{0,1}} \to \iota^\star \LL_{\Kup(X)|\Mfrak_{0,1}} \to \LL_{\Kup^{\log}(\Xcal_0)|\Mfrak_{0,1}} \to \LL_{\Kup^{\log}(\Xcal_0)|\Log\Mfrak_{0,1}} \to \varphi^\star \LL_{\Log\Mfrak_{0,1}|\Mfrak_{0,1}}[1].\]
We conclude that the composition is zero, because the final three terms form an exact triangle. Thus, we obtain the morphism \eqref{eqn: map from OD(D) to LogOb} giving rise to the exact triangle \eqref{eqn: second exact triangle 3}. Applying the octahedral axiom to \eqref{eqn: second exact triangle 1}, \eqref{eqn: second exact triangle 2}, \eqref{eqn: second exact triangle 3} we obtain the fundamental exact triangle:
\begin{equation}\label{eqn: compatible triple for pushforward} \iota^\star \EE_{\Kup(X)|\Mfrak_{0,1}} \to \EE_{\Kup^{\log}(\Xcal_0)|\Mfrak_{0,1}} \to \EE_{\Kup^{\log}(\Xcal_0)|\Kup(X)}\xrightarrow{[1]}.\end{equation}
We now show that $\EE_{\Kup^{\log}(\Xcal_0)|\Kup(X)}$ forms a perfect obstruction theory. Consider the exact triangle \eqref{eqn: second exact triangle 3}. Since $\OO_X(D)$ is convex (Assumption \ref{assumption: convexity of D}), we have a $2$-term resolution
\begin{equation*}  (\Rder\pi_\star \Lder f^\star \OO_D(D))^\vee = [ \OO \to \pi_\star f^\star \OO_X(D)]^\vee
\end{equation*}
and therefore this term has perfect amplitude contained in $[0,1]$. The second term of \eqref{eqn: second exact triangle 3} has perfect amplitude contained in $[-1,1]$, and thus (by the same argument as in the proof of Theorem~\ref{prop: POT over A1}) we conclude that $\EE_{\Kup^{\log}(\Xcal_0)|\Kup(X)}$ has perfect amplitude contained in $[-1,1]$.

On the other hand, the first two terms of \eqref{eqn: compatible triple for pushforward} are both of perfect amplitude contained in $[-1,0]$, from which it follows that the final term is of perfect amplitude contained in $[-2,0]$. Combining these two observations, we conclude that $\smash{\EE_{\Kup^{\log}(\Xcal_0)|\Kup(X)}}$ is of perfect amplitude contained in $[-1,0]$. The axioms of a triangulated category produce a morphism of exact triangles
\begin{equation*}
\begin{tikzcd}
	 \iota^\star \EE_{\Kup(X)|\Mfrak_{0,1}} \ar[r] \ar[d] & \EE_{\Kup^{\log}(\Xcal_0)|\Mfrak_{0,1}} \ar[r] \ar[d] & \EE_{\Kup^{\log}(\Xcal_0)|\Kup(X)} \ar[d] \ar[r,"{[1]}"] & \, \\
	 \iota^\star \LL_{\Kup(X)|\Mfrak_{0,1}} \ar[r] & \LL_{\Kup^{\log}(\Xcal_0)|\Mfrak_{0,1}} \ar[r] & \LL_{\Kup^{\log}(\Xcal_0)|\Kup(X)} \ar[r,"{[1]}"] & \,\end{tikzcd}
\end{equation*}
and two applications of the Four Lemma show that the right-hand vertical morphism
\begin{equation*} \EE_{\Kup^{\log}(\Xcal_0)|\Kup(X)} \to \LL_{\Kup^{\log}(\Xcal_0)|\Kup(X)} \end{equation*}
is surjective on $\HH^{-1}$ and an isomorphism on $\HH^0$. We thus obtain a perfect obstruction theory for $\iota$, fitting into the compatible triple \eqref{eqn: compatible triple for pushforward}.

It remains to show that this obstruction theory is given by a vector bundle in degree $-1$. The crucial observation is that $\iota$ is a closed embedding (Corollary~\ref{cor: iota closed embedding}) and so:
\begin{equation*} \HH^0 \big(\LL_{\Kup^{\log}(\Xcal_0)|\Kup(X)} \big) = 0. \end{equation*}
This fact is specific to our setting. It follows ultimately from the observation that all logarithmic stable maps to $\Xcal_0$ must factor through the divisor $D$ (Lemma~\ref{lem: factor through D}). We thus conclude that $\EE_{\Kup^{\log}(\Xcal_0)|\Kup(X)}$ is of perfect amplitude concentrated in degree $-1$, so there is a vector bundle $E$ such that
 \begin{equation}\label{eqn: vector bundle relative POT} \EE_{\Kup^{\log}(\Xcal_0)|\Kup(X)} = E^\vee[1]\end{equation}
 which completes the proof.
 \end{proof}

\begin{proof}[Proof of Theorem \ref{thm: virtual pushforward}] Examining \eqref{eqn: second exact triangle 3}, we note that $(\Rder \pi_\star \Lder f^\star \OO_D(D))^\vee$ may be expressed as the pullback of a complex on $\Kup(X)$. As will be demonstrated in the computations of \S \ref{sec: component calculations}, the same is true for $\varphi^\star \LL_{\Log\Mfrak_{0,1}|\Mfrak_{0,1}}$ since it admits an explicit resolution in terms of tautological bundles. Thus, we see that the vector bundle $E$ obtained in \eqref{eqn: vector bundle relative POT} is the pullback of a bundle from $\Kup(X)$:
\begin{equation*} E = \iota^\star F.\end{equation*}
From this, it follows that
\begin{equation*} \iota_\star [\Kup^{\log}(\Xcal_0)]^{\virt} = \iota_\star \iota^! [\Kup(X)]^{\virt} = \e(F) \cap [\Kup(X)]^{\virt}. \end{equation*}
The cohomological term $\e(F)$ may easily be computed using the exact triangle \eqref{eqn: second exact triangle 3} (or, to be more precise, its analogue on $\Kup(X)$). We have the following relation in K-theory:
\begin{align*} F & = -\chi \big(\EE_{\Kup^{\log}(\Xcal_0)|\Kup(X)}^\vee \big) \\
& = -\chi \big(\varphi^\star \TT_{\Log\Mfrak_{0,1}|\Mfrak_{0,1}}\big) + \chi \big( \Rder \pi_\star \Lder f^\star \OO_D(D) \big).
\end{align*}
From the proof of Proposition~\ref{prop: pushforward compatible triple} we have
\begin{equation*}\chi ( \Rder \pi_\star \Lder f^\star \OO_D(D) ) = \pi_\star f^\star \OO_X(D) - \OO\end{equation*}
while on the other hand by Proposition~\ref{prop: Lchi in -1,1} we have (suppressing pullbacks as before):
\begin{equation*} -\chi(\varphi^\star \TT_{\Log\Mfrak_{0,1}|\Mfrak_{0,1}}) = -\chi(\TT_{\Acal_Q}) + \chi(\TT_{\Acal_P}). \end{equation*}
Putting everything together, we arrive at the formula
\begin{equation*} F = \pi_\star f^\star \OO_X(D) - \LogOb \end{equation*}
where $\LogOb$ is given by Definition \ref{def: LogOb}. This completes the proof.\end{proof}

\section{Component contributions and geometric applications}\label{sec: component calculations}

\noindent While the preceding results may seem rather formal, we will now show that they are amenable to calculation, with extremely concrete geometric consequences. We focus on the main example of $(\PP^2,E)$ degenerating to $(\PP^2,\Delta)$. Although it is clear that our methods apply to any example with a sufficiently strong torus action, we leave the investigation of these to future work.

\subsection{Outline} We begin (\S \ref{sec: decomp by multidegree}) by specifying the decomposition of the moduli space whose contributions we are interested in calculating. We then (\S \ref{sec: localisation scheme}) describe a localisation procedure for calculating the integral of the zero-dimensional logarithmic virtual class:
\begin{equation*} \iota_\star [\Kup^{\log}(\Xcal_0)]^{\virt} \in \mathrm{A}_0(\Kup_{0,1}(\PP^2,d)).\end{equation*}
By the functoriality of virtual localisation, this provides a method for isolating the contribution of each component of the central fibre to the logarithmic Gromov--Witten invariant. This scheme is fully implemented in accompanying Sage code, which we use to generate tables of component contributions up to degree 8 (\S \ref{sec: tables}).

Based on these low-degree numerics, we conjecture general formulae for the various component contributions, and provide some evidence for these, including a striking combinatorial conjecture which arises from our localisation calculations (\S \ref{sec: conjectures}).

Finally we show how, working back from these component contributions, combined with the Gross--Pandharipande--Siebert multiple cover formula, one is able to completely describe the degeneration behaviour of embedded tangent curves up to degree $3$, and obtain partial information in higher degrees (\S \ref{sec: degenerations of curves}). The resulting theorems are purely classical, but we are not aware of a proof which does not pass through logarithmic Gromov--Witten theory.

\subsection{Decomposition by unordered multi-degree}\label{sec: decomp by multidegree} \label{sec: decomposition of moduli space} As discussed, from now on we focus on the main example. Consider the moduli space $\Kup^{\log}(\Xcal_0)=\Kup(\Delta)=\Kup_{0,1}(\Delta,d)$. This space has many connected and irreducible components, and is not pure-dimensional, not even locally. Enumerating all of its components, for general $d$, is a somewhat non-trivial task.

Instead we focus on a more granular decomposition of the moduli space than that given by connected or irreducible components. Note that since $\Delta \subseteq \PP^2$ is a curve, any stable map $f \colon C \to \Delta$ of degree $d$ gives a well-defined (ordered, non-negative) partition
\begin{equation*} d = d_0+d_1+d_2 \end{equation*}
which records the total degree of $f$ over each component $D_i$ of $\Delta$. Forgetting the ordering, we obtain a locally-constant multi-degree function
\begin{equation*} \mathrm{m} \colon \Kup(\Delta) \to \N^3/\mathrm{S}_3 \end{equation*}
which produces a decomposition into clopen substacks
\begin{equation*} \Kup(\Delta) = \coprod_{\mathbf{d} \in \N^3/\mathrm{S}_3} \mathrm{m}^{-1}(\mathbf{d}).\end{equation*}
indexed by (unordered, non-negative) partitions $\mathbf{d} \vdash d$ of length $3$. We note that each $\mathrm{m}^{-1}(\mathbf{d})$ has at least a three-fold symmetry, arising from the three-fold rotational symmetry of $\Delta$. We are interested in calculating the contributions of these substacks to the logarithmic Gromov--Witten invariant:
\begin{equation*} \int_{\mathrm{m}^{-1}(\mathbf{d})} [\Kup^{\log}(\Xcal_0)]^{\virt}. \end{equation*}
These provide new, geometrically meaningful refinements of the much-studied logarithmic Gromov--Witten invariants of $(\PP^2,E)$. It is these which we wish to calculate.

\begin{remark} Our methods allow us to deal with much finer decompositions than the one given by the unordered multi-degree. We have decided to organise information at this level for the purposes of exposition. \end{remark}

\subsection{Localisation scheme and computation of $\LogOb$}\label{sec: localisation scheme} The strategy is to apply functorial virtual localisation to the virtual push-forward formula (see Theorem \ref{thm: virtual pushforward} and Remark \ref{rmk: push forward abuse}):
\begin{equation*} \iota_\star [\Kup^{\log}(\Xcal_0)]^{\virt} = \left( \dfrac{\e(\pi_\star f^\star \OO_{\PP^2}(\Delta))}{\e(\LogOb)} \right) \cap [\Kup_{0,1}(\PP^2,d)].\end{equation*}
Since virtual localisation is a well-established technique in enumerative geometry, we will not spell out every detail in what follows, opting instead to focus on those aspects of our calculation which are novel. 

\subsubsection{Localisation setup}\label{subsection localisation setup} We quickly run through the standard localisation setup for $\PP^2$. Take $T=(\Cstar)^2$ and denote the standard weights by $\mu_1,\mu_2$. Choose an injective group homomorphism $T \to (\Cstar)^3$. This induces a linear action $T \acts \C^3$, whose weights we denote by $-\lambda_0,-\lambda_1,-\lambda_2$. Here each $\lambda_i$ is a linear form in $\mu_1,\mu_2$, and we assume that:
\begin{equation}\label{eqn: condition on weights} \lambda_0+\lambda_1+\lambda_2=0.\end{equation}
The action $T \acts \C^3$ descends to an action $\T \acts \PP^2$ whose fixed points are the standard co-ordinate points $p_0,p_1,p_2$ and whose one-dimensional orbit closures are the toric divisors $D_0,D_1,D_2$. Since the toric boundary $\Delta$ is preserved by this action, we obtain an action on the logarithmic scheme $T \acts \Xcal_0$.

The action $T \acts \C^3$ induces a linearisation of the tautological bundle $\OO_{\PP^2}(-1)$ and consequently we obtain an action $T \acts \OO_{\PP^2}(k)$ for any $k \in \Z$, whose weights over $p_0,p_1,p_2$ are $k\lambda_0,k\lambda_1,k\lambda_2$, respectively. Whenever we write $\OO_{\PP^2}(k)$ we will mean the $T$-equivariant line bundle equipped with this action.

\subsubsection{Functoriality, fixed loci and normal bundles}\label{subsection fixed loci}
The actions on $\Xcal_0$ and $\PP^2$ induce actions on the corresponding moduli spaces, such that the morphism
\begin{equation*} \iota \colon \Kup^{\log}(\Xcal_0) \to \Kup_{0,1}(\PP^2,d) \end{equation*}
is equivariant. Since a $T$-fixed stable map to $\PP^2$ must factor through $\Delta$, we see that the $T$-fixed loci in the source and the target are identical, and that $\iota$ restricts to an isomorphism between the two. These fixed loci are well-understood and are indexed by so-called localisation graphs $\Theta$ \cite{KontsevichEnumeration, GraberPandharipande}. Up to a finite cover, each fixed locus $\Fup_\Theta$ is a product of Deligne--Mumford spaces (parametrising curve components contracted to the torus-fixed points) and finite cyclic gerbes (parametrising curve components covering the torus-invariant lines).

Given such a fixed locus, its contribution to the integral of $\iota_\star[\Kup^{\log}(\Xcal_0)]^{\virt}$ is given by:
\begin{equation}\label{eqn: localisation term}
\int_{\Fup_{\Theta}} \left( \dfrac{\Euler(\pi_\star f^\star \OO_{\PP^2}(\Delta)|_{\Fup_{\Theta}})}{\Euler(\LogOb|_{\Fup_{\Theta}})\cdot \Euler\big(\Norm_{\Fup_{\Theta}|\Kup_{0,1}(\PP^2,d)}\big)} \right).
\end{equation}
Formulae for the normal bundle term in the denominator are well-known, see \cite[\S 4]{GraberPandharipande} or \cite[Theorem~9.2.1]{CoxKatz}.

By functorial virtual localisation \cite[Lemma 2.1]{LLY3}, for every choice of multi-degree $\mathbf{d}$ the contribution of the open and closed substack (see \S \ref{sec: decomposition of moduli space})
\begin{equation*} \mathrm{m}^{-1}(\mathbf{d}) \subseteq \Kup^{\log}(\Xcal_0) \end{equation*}
is given by the sum of those terms \eqref{eqn: localisation term} for which $\Fup_\Theta \subseteq \mathrm{m}^{-1}(\mathbf{d})$. Determining all such fixed loci is an easy combinatorial exercise. Thus, the localisation calculation will allow us to separate out the individual component contributions.

\subsubsection{Computing $\pi_\star f^\star \OO_{\PP^2}(\Delta)$}\label{subsection computing numerator}
Since we are in genus zero, the bundle $\pi_\star f^\star \OO_{\PP^2}(\Delta)|_{\Fup_{\Theta}}$ is (non-equivariantly) trivial, meaning that the term $\Euler(\pi_\star f^\star \OO_{\PP^2}(\Delta)|_{\Fup_{\Theta}})$ is pure weight. The assumption $\lambda_0+\lambda_1+\lambda_2=0$ ensures we have an identification of \emph{equivariant} line bundles
\begin{equation*} \mathrm{N}_{\Delta|\PP^2}=\OO_{\PP^2}(3)|_{\Delta}.\end{equation*}
The weights at the torus-fixed points are therefore $3\lambda_0,3\lambda_1,3\lambda_2$. From this the weights on $\pi_\star f^\star \OO_{\PP^2}(\Delta)|_{\Fup_{\Theta}}$ can easily be calculated, see e.g. \cite[\S 4]{GraberPandharipande}.

\subsubsection{Local computation of $\mathrm{LogOb}$} \label{subsection computing LogOb}\label{sec: computing LogOb} It remains to describe the denominator term $\Euler(\LogOb|_{\Fup_{\Theta}})$. This is the most novel part of the argument, relying crucially on the deformation theory of the Artin fan and its relation to line bundles encoded in the logarithmic structure, together with a tropical-geometric method for computing such bundles.

We begin with an explicit local description of $\LogOb$. Consider an atomic open neighbourhood $\mathcal{V} \subseteq \Kup^{\log}(\Xcal_0)$ (see \cite[\S 2.2]{AbramovichWiseBirational}). The unique closed stratum of $\mathcal{V}$ is indexed by a combinatorial type of tropical stable map to $\RR_{\geq 0}$. As before, we let $\sqC$ denote the source curve of this combinatorial type; this has $r$ edges and $m$ leaves, corresponding to the edge lengths $l_1,\ldots,l_r$ and target offsets $c_1,\ldots,c_m$ in the tropical moduli.

\begin{lemma}\label{lem: LogOb description 1} We have
\begin{equation*}\LogOb|_{\mathcal{V}} = \sum_{i=1}^m \OO(c_i) - \sum_{j=1}^{m-1} \OO(r_j)\end{equation*}
in which the $r_j$ are certain relation parameters, defined in the proof.\end{lemma}

\begin{remark} The quantities $c_i$ and $r_j$ give piecewise-linear functions on the tropicalisation of $\mathcal{V}$ or, equivalently, global sections of the ghost sheaf. These give rise to associated line bundles $\OO(c_i)$ and $\OO(r_j)$ equipped with preferred sections. These are the pullbacks of the corresponding toric Cartier divisors on the Artin fan.\end{remark}

\begin{proof}
Recall that we have
\begin{equation*} \LogOb|_{\mathcal{V}} = \chi(\TT_{\Acal_Q}) - \chi(\TT_{\Acal_P}) + \OO \end{equation*}
where $Q$ and $P$ are the minimal monoids corresponding to the tropical stable map and the underlying tropical curve, respectively. There are explicit presentations (see \S \ref{sec: tropical moduli and minimal monoid} and \S \ref{sec: isolating LogOb}; as before, we suppress pullbacks from the notation):
\begin{align*} \TT_{\Acal_Q} &  = \big[\OO^{\oplus r+1} \to \T_{\Aaff^{r+m}} \to \mathrm{N}_{U_Q|\Aaff^{r+m}} \big], \\[3pt]
\TT_{\Acal_P} & = \big[ \OO^{\oplus r} \to  \T_{\Aaff^r} \big]. \end{align*}
From these, we obtain:
\begin{align*}\LogOb |_{\mathcal{V}} & = \big[ -(r+1)\OO + \mathrm{T}_{\Aaff^{r+m}} - \mathrm{N}_{U_Q|\Aaff^{r+m}} \big] - \big[ - r \OO + \mathrm{T}_{\Aaff^r} \big] + \OO	\\[3pt]
& = \mathrm{T}_{\Aaff^{r+m}} - \mathrm{T}_{\Aaff^r} - \mathrm{N}_{U_Q|\Aaff^{r+m}}.\addtocounter{equation}{1}\tag{\theequation} \label{eqn: LogOb description tangent normal bundles} \end{align*}

The tangent bundle terms decompose into toric line bundles associated to the co-ordinate hyperplanes:
\begin{align*} \T_{\Aaff^{r+m}} = \sum_{i=1}^r \OO(l_i) + \sum_{i=1}^m \OO(c_i), \qquad
\T_{\Aaff^r} = \sum_{i=1}^r \OO(l_i).\end{align*}
The normal bundle term may also be expressed in terms of such bundles. Recall that $Q$ arises as a quotient
\begin{equation*} \N^{r+m} \rightarrow Q \end{equation*}
given by $m-1$ independent continuity relations. For each such relation, let $r_j$ denote the sum of tropical parameters appearing on one side of the equation (notice that we always have $r_j = f(v_j)$ for some vertex $v_j \in \sqC$). We then have:
\begin{equation*} \mathrm{N}_{U_Q|\Aaff^{r+m}} = \sum_{j=1}^{m-1} \OO(r_j).\end{equation*}
Putting everything together, we arrive at the desired formula. 
\end{proof}

The preceding lemma gives a local description for $\LogOb$ in terms of line bundles associated to piecewise-linear functions on the tropical moduli space. We now give a method for calculating such bundles, in terms of evaluation and cotangent line classes. A variant of this technique was first employed in \cite[\S 3]{BNRGenus1}.

\begin{construction} Consider the restriction to $\mathcal{V}$ of the universal logarithmic stable map
\bcd
\Ccal \ar[r,"f"] \ar[d,"\pi"] & \Xcal_0 \\
\mathcal{V} \ar[u,bend left,"x" pos=0.45] & \,
\ecd
which we tropicalise to obtain a tropical stable map
\bcd
\sqC \ar[d,"\mathsf{p}"] \ar[r,"\mathsf{f}"] & \RR_{\geq 0} \\
\sigma \ar[u, bend left, "\mathsf{x}"] & \,
\ecd
over the base cone $\sigma = Q^\vee_{\RR} = \operatorname{Trop}(\mathcal{V})$. Given a piecewise-linear function $\varphi$ on $\sigma$ we wish to describe the associated line bundle. First note that we have:
\begin{equation*} \OO_{\mathcal{V}}(\varphi) = x^\star \pi^\star \OO_{\mathcal{V}}(\varphi) = x^\star \OO_{\Ccal}(\mathsf{p}^\star \varphi).\end{equation*}
The basic idea is to compare the piecewise-linear functions $\mathsf{p}^\star \varphi$ and $\mathsf{f}^\star 1$ on $\sqC$. This will give a relation between the bundles $\OO_{\Ccal}(\mathsf{p}^\star \varphi)$ and $\OO_{\Ccal}(\mathsf{f}^\star 1) = f^\star \OO_X(D)$ on $\Ccal$. Pulling back along the section $x$, we will then obtain an expression for $\OO_{\mathcal{V}}(\varphi)$.

Let $v_0 \in \sqC$ be the vertex containing the marking leg, and denote the adjacent edges by $e_1,\ldots,e_k$, with associated expansion factors $\alpha_{1},\ldots,\alpha_k$. Let $\Ccal_0$ be the corresponding curve component and $q_1,\ldots,q_k$ the corresponding nodes. The piecewise-linear function $\mathsf{f}^\star 1$ has slope $3d$ along the marking leg, and $-\alpha_{i}$ along each edge $e_i$. On the other hand, the function $\mathsf{p}^\star \varphi$ has slope zero along every edge and leg. We thus obtain \cite[Proposition 2.4.1]{RanganathanSantosParkerWise1}:
\begin{equation}\label{eqn: tropical bundles on C} \OO_{\Ccal}( \mathsf{f}^\star 1 - \mathsf{p}^\star \varphi)|_{\Ccal_0} = \OO_{\Ccal_0}(3d x - \Sigma_{i=1}^k \alpha_i q_i) \otimes \pi^\star \OO_{\mathcal{V}}(\mathsf{f}(v_0) - \varphi).\end{equation}
Pulling back along $x$, using the fact that $\OO_{\Ccal}(\mathsf{f}^\star 1) = f^\star \OO_{X}(D) = f^\star \OO_{\PP^2}(3)$, we obtain
\begin{align*} \OO_{\mathcal{V}}(\varphi) & = \ev_x^\star \OO_{\PP^2}(3) \otimes x^\star \OO_{\Ccal_0}(-3dx) \otimes \OO_{\mathcal{V}}(\varphi - \mathsf{f}(v_0))\\
& = \ev_x^\star \OO_{\PP^2}(3) \otimes L_x^{3d} \otimes \OO_{\mathcal{V}}(\varphi - \mathsf{f}(v_0)) \addtocounter{equation}{1}\tag{\theequation} \label{eqn: formula for O(phi)} \end{align*}
where $L_x$ is the cotangent line bundle. For every $\varphi$ we will consider, the piecewise-linear function on $\mathcal{V}$
\begin{equation*} \varphi-\mathsf{f}(v_0) \end{equation*}
will be expressible as a linear combination of edge lengths (this will typically not be the case for $\varphi$ itself). Since the line bundle associated to an edge length is given by the pullback of the corresponding boundary divisor in $\Mfrak_{0,1}$, the identity \eqref{eqn: formula for O(phi)} gives a closed formula for $\OO_{\mathcal{V}}(\varphi)$ in terms of tautological bundles. The expression for $\e (\OO_{\mathcal{V}}(\varphi))$ in terms of tautological classes immediately follows.\end{construction}

\begin{example} Here we illustrate this process. Consider the following combinatorial type of tropical stable map
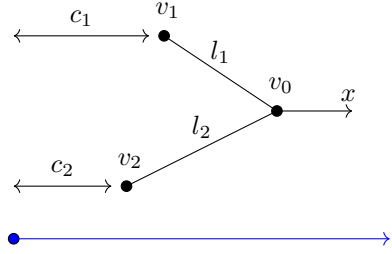
\begin{center}
\begin{tikzpicture}

\draw[fill=blue] (0,0.3) circle[radius=2pt];
\draw[->,color=blue] (0,0.3) to (5,0.3);

\draw[fill] (2,3) circle[radius=2pt];
\draw (2.05,3.1) node[above]{\small$v_1$};

\draw (2,3) -- (3.5,2);
\draw (2.75,2.5) node[above]{\small$l_1$};

\draw[<->] (0,3) -- (1.8,3);
\draw (0.9,3) node[above]{\small$c_1$};

\draw[fill] (1.5,1) circle[radius=2pt];
\draw (1.55,1.1) node[above]{\small$v_2$};

\draw (1.5,1) -- (3.5,2);
\draw (2.5,1.5) node[above]{\small$l_2$};

\draw[<->] (0,1) -- (1.3,1);
\draw (0.65,1) node[above]{\small$c_2$};

\draw[fill] (3.5,2) circle[radius=2pt];
\draw (3.55,2.1) node[above]{\small$v_0$};

\draw[->] (3.5,2) -- (4.5,2);
\draw (4.45,2) node[above]{\small$x$};

\end{tikzpicture}	
\end{center}
where $v_1$ and $v_2$ have degree $1$ and $v_0$ has degree $0$. By Lemma \ref{lem: LogOb description 1} we have the local description $\LogOb = \OO(c_1) + \OO(c_2) - \OO(r_1)$, where $r_1 = \mathsf{f}(v_0)= c_1 + 3l_1 = c_2 + 3l_2$. We begin by calculating $\OO(c_1)$. The piecewise-linear function $\mathsf{f}^\star 1 - \mathsf{p}^\star c_1$ on $\sqC$ is given by
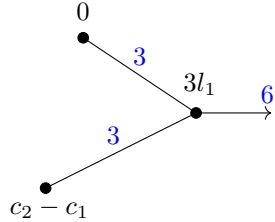
\begin{center}
\begin{tikzpicture}

\draw[fill] (2,3) circle[radius=2pt];
\draw (2,3.1) node[above]{\small$0$};

\draw (2,3) -- (3.5,2);
\draw[blue] (2.75,2.5) node[above]{\small$3$};

\draw[fill] (1.5,1) circle[radius=2pt];
\draw (1.55,1) node[below]{\small$c_2-c_1$};

\draw (1.5,1) -- (3.5,2);
\draw[blue] (2.4,1.45) node[above]{\small$3$};

\draw[fill] (3.5,2) circle[radius=2pt];
\draw (3.55,2.1) node[above]{\small$3l_1$};

\draw[->] (3.5,2) -- (4.5,2);
\draw[blue] (4.45,2) node[above]{\small$6$};

\end{tikzpicture}	
\end{center}
where the slopes are indicated in blue. The identity \eqref{eqn: tropical bundles on C} then reads:
\begin{equation*} \OO_{\Ccal}(\mathsf{f}^\star 1 - \mathsf{p}^\star c_1)|_{\Ccal_0} = \OO_{\Ccal_0}(6x-3q_1 -3q_2) \otimes \pi^\star \OO(3l_1). \end{equation*}
Pulling back along $x$ and rearranging, we obtain
\begin{align*} \OO_{\mathcal{V}}(c_1) & = x^\star \OO_{\Ccal}(\mathsf{f}^\star 1) \otimes x^\star \OO_{\Ccal_0}(-6x) \otimes \OO_{\mathcal{V}}(-3l_1) \\
& = \ev_x^\star \OO_{\PP^2}(3) \otimes L_x^6 \otimes \OO_{\mathcal{V}}(-3l_1).\end{align*}
We therefore conclude that
\begin{equation*} \e (\OO_{\mathcal{V}}(c_1)) = \ev_x^\star (3H) + \psi_x^6 - 3D_1\end{equation*}
where $D_1=\e(\OO_{\mathcal{V}}(l_1))$ is the pullback of the corresponding boundary divisor from $\Mfrak_{0,1}$ (this may be identified with a stratum of the logarithmic moduli space or, if the node persists on an open neighbourhood, with a sum of tangent line classes). By the same arguments, it is easy to see that:
\begin{align*} \e(\OO_{\mathcal{V}}(c_2)) & = \ev_x^\star(3H) + \psi_x^6 - 3D_2, \\
\e(\OO_{\mathcal{V}}(r_1)) & = \ev_x^\star (3H) + \psi_x^6.\end{align*}\end{example}

\subsubsection{Global computation of $\LogOb$} The above computations are local to an atomic neighbourhood of the moduli space $\Kup^{\log}(\Xcal_0)$. We now show how to obtain a \emph{global} description of $\e(\LogOb|_{\Fup_{\Theta}})$, over each fixed locus $\Fup_{\Theta}$. We begin with the following key observation, which  drastically simplifies the calculations:
\begin{theorem} \label{thm: LogOb pure weight} $\e(\LogOb|_{\Fup_{\Theta}})$ is pure weight.\end{theorem}
\begin{proof}
The fixed locus $\Fup_{\Theta}$ determines a unique ``least degenerate'' combinatorial type of tropical curve, with only those nodes forced by the graph $\Theta$. Since we are in genus zero, this in turn defines a unique combinatorial type of tropical stable map to $\RR_{\geq 0}$. As before, let us suppose that $\sqC$ has $r$ edges and $m$ leaves, corresponding to the edge lengths  $l_1,\ldots,l_r$ and target offsets $c_1,\ldots,c_m$ in the tropical moduli.

This combinatorial type may degenerate as we move towards the boundary of the fixed locus. Note, however, that since the fixed locus only contains degenerations of contracted components, and every leaf is non-contracted by stability, the number of leaves $m$ remains constant on the entire fixed locus (whereas the number of edges may exceed $r$ over the boundary).

We assume for simplicity that the generic combinatorial type takes the following form:
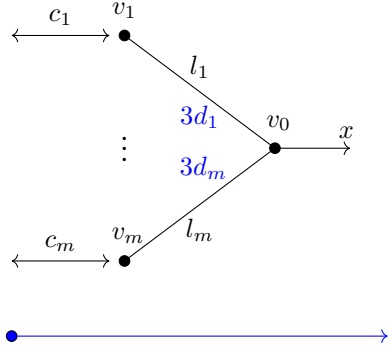
\begin{center}
\begin{tikzpicture}

\draw[fill=blue] (0,-0.5) circle[radius=2pt];
\draw[->,color=blue] (0,-0.5) to (5,-0.5);

\draw[fill] (1.5,3.5) circle[radius=2pt];
\draw (1.5,3.6) node[above]{\small$v_1$};

\draw (1.5,2.1) node{\vdots};

\draw (1.5,3.5) -- (3.5,2);
\draw (2.5,2.8) node[above]{\small$l_1$};
\draw[blue] (2.5,2.7) node[below]{\small$3d_1$};

\draw[<->] (0,3.5) -- (1.3,3.5);
\draw (0.65,3.5) node[above]{\small$c_1$};

\draw (1.5,0.5) -- (3.5,2);
\draw (2.5,1.2) node[below]{\small$l_m$};
\draw[blue] (2.55,1.45) node[above]{\small$3d_m$};

\draw[fill] (1.5,0.5) circle[radius=2pt];
\draw (1.55,0.6) node[above]{\small$v_m$};

\draw[<->] (0,0.5) -- (1.3,0.5);
\draw (0.65,0.5) node[above]{\small$c_m$};

\draw[fill] (3.5,2) circle[radius=2pt];
\draw (3.55,2.1) node[above]{\small$v_0$};

\draw[->] (3.5,2) -- (4.5,2);
\draw (4.45,2) node[above]{\small$x$};

\end{tikzpicture}	
\end{center}
Indeed, this is always the local structure around each vertex. To prove that $\e(\LogOb|_{\Fup_\Theta})$ is pure weight, it is equivalent to prove that it is pure weight when pulled back to each factor of the fixed locus $\Fup_\Theta$. We may thus consider each vertex individually. Up to a finite cover, we have $\Fup_{\Theta}=\ol{\Mcal}_{0,m+1}$. We begin by considering the target offset bundles $\OO(c_i)$. Away from the boundary of $\Fup_{\Theta}$, we compute
\begin{equation*} \e (\OO(c_i)) = 3\lambda_{j(i)} + 3d \psi_x + 3d_i \psi_{q_i}\end{equation*}
where $p_{j(i)}$ is the fixed point mapped to by the non-contracted leaf component $C_{v_i}$ away from its intersection with $C_{v_0}$. As the combinatorial type degenerates, this formula must be modified with appropriate boundary corrections, which we now describe. Given a partition $A \sqcup B = \{1,\ldots,m,x\}$ with $x \in B$, we let 
\begin{equation*} D_{(A,B)} \subseteq \overline{\Mcal}_{0,m+1} \end{equation*}
denote the corresponding boundary divisor. A direct calculation local to each boundary divisor then gives the following global formula for $\e (\OO(c_i))$ on $\Fup_{\Theta}=\ol\Mcal_{0,m+1}$:
\begin{equation}\label{eqn: formula for O(ci)} \e (\OO(c_i)) = 3\lambda_{j(i)} + 3d \psi_x + 3d_i \psi_{q_i} - \sum_{\substack{(A,B) \\ \text{with } i \in A}} \left( \sum_{j \in A} 3d_j \cdot D_{(A,B)} \right).\end{equation}
Now, for $j \neq i$ we have the following boundary relation, obtained by pullback from $\ol\Mcal_{0,3}$:
\begin{equation*} \psi_x = ( i \ j\ |\ x) = \sum_{\substack{(A,B)\\ \text{with }i,j \in A}} D_{(A,B)}.\end{equation*}
Using $3d=3d_1+\ldots+3d_m$ we thus obtain:
\begin{equation*} (3d -3d_i) \psi_x = \sum_{j \neq i} 3d_j \psi_x = \sum_{\substack{(A,B) \\ \text{with } i \in A}}\bigg( \sum_{\substack{j \in A \\ j \neq i}} 3d_j\cdot D_{(A,B)} \bigg).\end{equation*}
On the other hand, we have the following relation involving the remaining cotangent line classes in \eqref{eqn: formula for O(ci)}, again easily obtained by pullback from $\ol\Mcal_{0,3}$ (see also \cite{LeePandharipande}):
\begin{equation*} 3d_i \psi_x + 3d_i \psi_{q_i} = \sum_{\substack{(A,B) \\ \text{with } i \in A}} 3d_i \cdot D_{(A,B)}.\end{equation*}
Combining these two expressions, we see that the non-weight terms in \eqref{eqn: formula for O(ci)} cancel precisely, leaving us with
\begin{equation*} \e(\OO(c_i)) = 3\lambda_{j(i)} \end{equation*}
which is pure weight. It remains to show the same for the $\OO(r_j)$ terms. Recall (Lemma \ref{lem: LogOb description 1}) that these arose from the normal bundle of the local toric model for the moduli space of tropical stable maps. We will show that this bundle is non-equivariantly trivial on the fixed locus, which immediately implies that its Euler class is pure weight.

We begin by clarifying notation. The monoid $Q$ will be used to denote the minimal monoid corresponding to the least degenerate combinatorial type on the fixed locus. As already noted, this combinatorial type can degenerate over $\Fup_{\Theta}$, producing additional edge lengths but no additional target offsets. We denote the minimal monoid corresponding to such a degeneration by $Q^\prime$, so that we have a regular embedding
\begin{equation*} U_{Q^\prime} \subseteq \Aaff^{r^\prime + m} \end{equation*}
where $r^\prime \geq r$ is the number of edge lengths. The relevant piece of $\LogOb$ is then given, local to such a stratum, by the bundle:
\begin{equation*} \mathrm{N}_{U_{Q^\prime}|\Aaff^{r^\prime+m}}.\end{equation*}
On the other hand it is easy to see, by examining the defining equations, that the following square is cartesian
\bcd
U_Q \ar[r,hook] \ar[d,hook] \ar[rd,phantom,"\square"] & \Aaff^{r+m} \ar[d,hook] \\
U_{Q^\prime} \ar[r,hook] & \Aaff^{r^\prime+m}
\ecd
where the morphism $U_Q \hookrightarrow U_{Q^\prime}$ is induced by the generisation map $Q^\prime \to Q$. Since this intersection is transverse, the square is Tor-independent, and so we have an identification
\begin{equation*} \mathrm{N}_{U_{Q^\prime}|\Aaff^{r^\prime + m}} = \mathrm{N}_{U_Q|\Aaff^{r+m}}\end{equation*}
where as usual we suppress pullbacks. Along the stratum under consideration, the morphism $\Fup_{\Theta} \to U_{Q^\prime}$ factors through $U_Q$, and we may therefore identify the relevant piece of $\LogOb$ with the bundle
\begin{equation*} \mathrm{N}_{U_Q|\Aaff^{r+m}} \end{equation*}
corresponding to the least degenerate combinatorial type. There is such an identification along every stratum of $\Fup_{\Theta}$, and since these are compatible along generisations we obtain a global identification. But the morphism $\Fup_{\Theta} \to U_Q$ factors through the origin (since all the tropical parameters persist on the fixed locus), and so this bundle is pulled back from a point, hence trivial. \end{proof}

The upshot of the previous result is that in our computations we may discard all classes which are not pure weight. This includes in particular all boundary correction terms. Consequently, the class $\e(\LogOb|_{\Fup_{\Theta}})$ may be computed from the generic combinatorial type of the fixed locus, since all corrections arising from further degenerations will be discarded. Given the techniques described in \S \ref{sec: computing LogOb} for calculating $\e(\OO(c_i))$ and $\e(\OO(r_j))$, this is now an easy process.

\subsubsection{Graph splitting formalism} \label{sec: graph splitting} The discussion thus far shows how to compute all of the equivariant classes appearing in \eqref{eqn: localisation term}, and hence gives a complete in-principle method for carrying out the localisation computation.

However, for the purposes of proving general formulae, as well as efficient computer calculations, it is necessary to be more explicit. In this section, we uncover a recursive structure governing the localisation contributions, which we leverage to carry out our computations.

The basic idea is to recursively split each localisation graph at the root vertex supporting the marking leg $x$. There are two possible situations, depending on the valency of this vertex:
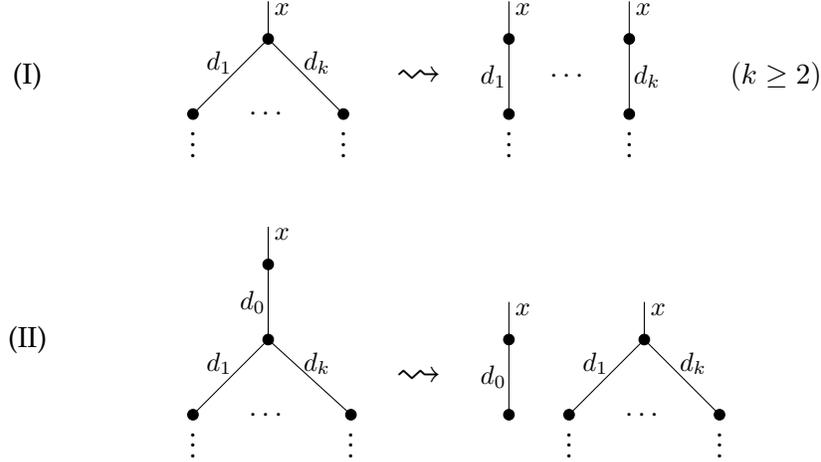
\begin{center}
\begin{tikzpicture}
	\draw (-3.2,-0.5) node{(I)};
	
	\draw[fill=black] (0,0) circle[radius=2pt];
	\draw (0,0) -- (0,0.5);
	\draw (-0.05,0.4) node[right]{\small$x$};
	
	\draw (0,0) -- (-1,-1);
	\draw (-0.65,-0.6) node[above]{\small$d_1$};
	\draw[fill=black] (-1,-1) circle[radius=2pt];
	\draw (-1,-1.3) node{$\vdots$};
	
	\draw (0,-1) node{$\cdots$};
	
	\draw (0,0) -- (1,-1);
	\draw (0.65,-0.6) node[above]{\small$d_k$};
	\draw[fill=black] (1,-1) circle[radius=2pt];
	\draw (1,-1.3) node{$\vdots$};
	
	\draw (2,-0.5) node{\LARGE$\rightsquigarrow$};
	
	\draw[fill=black] (3.2,0) circle[radius=2pt];
	\draw (3.2,0) -- (3.2,0.5);
	\draw (3.15,0.4) node[right]{\small$x$};
	\draw (3.2,0) -- (3.2,-1);
	\draw (3.3,-0.5) node[left]{\small$d_1$};
	\draw[fill=black] (3.2,-1) circle[radius=2pt];
	\draw (3.2,-1.3) node{$\vdots$};
	
	\draw (4,-0.5) node{$\cdots$};
	
	\draw[fill=black] (4.8,0) circle[radius=2pt];
	\draw (4.8,0) -- (4.8,0.5);
	\draw (4.75,0.4) node[right]{\small$x$};
	\draw (4.8,0) -- (4.8,-1);
	\draw (4.725,-0.5) node[right]{\small$d_k$};
	\draw[fill=black] (4.8,-1) circle[radius=2pt];
	\draw (4.8,-1.3) node{$\vdots$};
	
	\draw (6.75,-0.5) node{$(k \geq 2)$};
	\draw (-3.2,-4) node{(II)};
	
	\draw[fill=black] (0,-4) circle[radius=2pt];
	\draw (0,-4) -- (0,-3);
	\draw (0.1,-3.5) node[left]{\small$d_0$};
	\draw[fill=black] (0,-3) circle[radius=2pt];
	\draw (0,-3) -- (0,-2.5);
	\draw (-0.05,-2.6) node[right]{\small$x$};
	
	\draw (0,-4) -- (-1,-5);
	\draw (-0.65,-4.6) node[above]{\small$d_1$};
	\draw[fill=black] (-1,-5) circle[radius=2pt];
	\draw (-1,-5.3) node{$\vdots$};
	
	\draw (0,-5) node{$\cdots$};
	
	\draw (0,-4) -- (1.1,-5);
	\draw (0.65,-4.6) node[above]{\small$d_k$};
	\draw[fill=black] (1.1,-5) circle[radius=2pt];
	\draw (1.1,-5.3) node{$\vdots$};
	
	\draw (2,-4.5) node{\LARGE$\rightsquigarrow$};
	
	\draw[fill=black] (3.2,-4) circle[radius=2pt];
	\draw (3.2,-4) -- (3.2,-3.5);
	\draw (3.15,-3.6) node[right]{\small$x$};
	\draw (3.2,-4) -- (3.2,-5);
	\draw (3.3,-4.5) node[left]{\small$d_0$};
	\draw[fill=black] (3.2,-5) circle[radius=2pt];
	
	\draw[fill=black] (5,-4) circle[radius=2pt];
	\draw (5,-4) -- (5,-3.5);
	\draw (4.95,-3.6) node[right]{\small$x$};
	
	\draw (5,-4) -- (4,-5);
	\draw (4.35,-4.6) node[above]{\small$d_1$};
	\draw[fill=black] (4,-5) circle[radius=2pt];
	\draw (4,-5.3) node{$\vdots$};
	
	\draw (5,-5) node{$\cdots$};
	
	\draw (5,-4) -- (6,-5);
	\draw (5.65,-4.6) node[above]{\small$d_k$};
	\draw[fill=black] (6,-5) circle[radius=2pt];
	\draw (6,-5.3) node{$\vdots$};
\end{tikzpicture}
\end{center}

\begin{remark} It is important not to misinterpret the localisation graphs as tropical curves.\end{remark}

Each time we split, we compare the contribution of the input graph to the product of contributions of the output graphs. This ratio is referred to as the \emph{defect}. We have the following useful lemma:

\begin{lemma} The defect in $\pi_\star f^\star \OO_{\PP^2}(3)$ cancels with the defect in $\LogOb$.\end{lemma}

\begin{proof} Note that all the classes involved are pure weight. We will deal with the two splitting pictures separately.

Case (I): Let $p_i$ be the torus-fixed point mapped to by the root vertex. The defect in $\pi_\star f^\star \OO_{\PP^2}(3)$ may be calculated from the normalisation sequence. There are $k$ factors of $(3\lambda_i)^{-1}$ coming from the nodes of the input graph which disappear after splitting, and one factor of $3\lambda_i$ coming from the contracted component associated to the root vertex of the input graph; the defect is thus $(3\lambda_i)^{1-k}$. On the other hand, the defect in $\LogOb$ is given by the $k-1$ relation parameters at the root vertex (which disappear after splitting). Since $\e(\OO(r_j)) = \ev_x^\star (3H)$ after discarding non-equivariant terms, each of these contributes a factor of $3\lambda_i$, and hence the overall defect is also $(3\lambda_i)^{1-k}$.

Case (II): Let $p_j$ be the torus-fixed point mapped to by the central $(k+1)$-valent vertex of the input graph. The defect in $\pi_\star f^\star \OO_{\PP^2}(3)$ is given by $(3\lambda_j)^{-1}$, coming from the single node which disappears after splitting. On the other hand, the defect in $\LogOb$ is given by (the inverse of) the target offset $c_0$ in the first output graph (the other parameters and relations are unchanged). We compute $\e(\OO(c_0))^{-1} = (3\lambda_j)^{-1}$, and so once again the defects cancel.\end{proof}

Therefore, at each step it is only necessary to calculate the defect arising from the normal bundle term, which amounts to a simple calculation on Deligne--Mumford space. Recursively, this expresses the contribution of each localisation graph in terms of the contributions of the so-called \emph{atomic graphs}
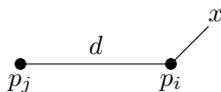
\begin{center}
\begin{tikzpicture}
	
	\draw[fill=black] (0,0) circle[radius=2pt];
	\draw (0,0) node[below]{\small$p_j$};
	
	\draw (0,0) -- (2,0);
	\draw (1,0) node[above]{\small$d$};
	
	\draw[fill=black] (2,0) circle[radius=2pt];
	\draw (2,0) node[below]{\small$p_i$};
	
	\draw (2,0) -- (2.5,0.5);
	\draw (2.6,0.65) node{\small$x$};

\end{tikzpicture}	
\end{center}
which, using the techniques of \S \ref{sec: localisation scheme}, we easily calculate to be:
\begin{equation}\label{eqn: atomic contribution} \left( \dfrac{(-1)^d}{d \cdot (d!)^2\cdot (\lambda_j-\lambda_i)^{2d-1}} \right) \cdot \prod_{a=1}^{d-1}\left( a \lambda_i + (3d-a)\lambda_j \right) \cdot \prod_{b=0}^{d-1} \left( (3d-b)\lambda_i + b \lambda_j \right).\end{equation}

\subsection{Tables of contributions}\label{sec: tables} The graph splitting algorithm described above is implemented in accompanying Sage code. The code is effective on an average laptop computer up to degree $8$. We use it to generate tables of component contributions, which we organise according to unordered multi-degree (see \S \ref{sec: decomposition of moduli space}). Note that, as we should expect, the sum of the contributions for each degree gives the corresponding maximal contact logarithmic Gromov--Witten invariant (as calculated for instance in \cite[Example 2.2]{GathmannMirror}).
\begin{center}
\begin{minipage}[t]{0.4\textwidth}
\noindent \center\textbf{Degree $1$}\\
\begin{tabular}{| c | c |}
	\hline
	\textbf{Multi-degree} & \textbf{Contribution} \\
	\hline \hline
	$(1,0,0)$ & $9$ \\
	\hline \hline
	\textbf{Total:} & $9$ \\
	\hline
\end{tabular}\medskip
\noindent \center\textbf{Degree $2$}\\
\begin{tabular}{| c | c |}
	\hline
	\textbf{Multi-degree} & \textbf{Contribution} \\
	\hline \hline
	$(2,0,0)$ & $63/4$ \\
	\hline
	$(1,1,0)$ & $18$ \\
	\hline \hline
	\textbf{Total:} & $135/4$ \\
	\hline
\end{tabular}\medskip
\noindent \center\textbf{Degree $3$}\\
\begin{tabular}{| c | c |}
	\hline
	\textbf{Multi-degree} & \textbf{Contribution} \\
	\hline \hline
	$(3,0,0)$ & $55$ \\
	\hline
	$(2,1,0)$ & $162$ \\
	\hline 
	$(1,1,1)$ & $27$ \\
	\hline\hline
	\textbf{Total:} & $244$ \\
	\hline
\end{tabular}\medskip
\noindent \center\textbf{Degree $4$}\\
\begin{tabular}{| c | c |}
	\hline
	\textbf{Multi-degree} & \textbf{Contribution} \\
	\hline \hline
	$(4,0,0)$ & $4,\!095/16$ \\
	\hline
	$(3,1,0)$ &  $936$ \\
	\hline 
	$(2,2,0)$ &  $1,\!089/2$ \\
	\hline
	$(2,1,1)$ & $576$ \\
	\hline\hline
	\textbf{Total:} & $36,\!999/16$ \\
	\hline
\end{tabular}
\end{minipage}
\begin{minipage}[t]{0.4\textwidth}
\noindent \center\textbf{Degree $5$}\\
\begin{tabular}{| c | c |}
	\hline
	\textbf{Multi-degree} & \textbf{Contribution} \\
	\hline \hline
	$(5,0,0)$ & $34,\!884/25$ \\
	\hline
	$(4,1,0)$ &  $6,\!120$ \\
	\hline 
	$(3,2,0)$ &  $8,\!190$ \\
	\hline
	$(3,1,1)$ &  $4,\!680$ \\
	\hline
	$(2,2,1)$ & $5,\!040$ \\
	\hline\hline
	\textbf{Total:} & $635,\!634/25$ \\
	\hline
\end{tabular}
\noindent \center\textbf{Degree $6$}\\
\begin{tabular}{| c | c |}
	\hline
	\textbf{Multi-degree} & \textbf{Contribution} \\
	\hline \hline
	$(6,0,0)$ & $33,\!649/4$ \\
	\hline
	$(5,1,0)$ & $43,\!092$ \\
	\hline 
	$(4,2,0)$ & $130,\!815/2$ \\
	\hline
	$(4,1,1)$ & $40,\!014$  \\
	\hline
	$(3,3,0)$ & $36,\!992$ \\
	\hline
	$(3,2,1)$ & $96,\!228$ \\
	\hline
	$(2,2,2)$ & $67,\!797/4$ \\
	\hline\hline
	\textbf{Total:} & $307,\!095$ \\
	\hline
\end{tabular}
\end{minipage} \newpage

\begin{minipage}[t]{0.4\textwidth}
\noindent \center\textbf{Degree $7$}\\
\begin{tabular}{| c | c |}
	\hline
	\textbf{Multi-degree} & \textbf{Contribution} \\
	\hline \hline
	$(7,0,0)$ & $2,\!664,\!090/49$ \\
	\hline
	$(6,1,0)$ & $318,\!780$ \\
	\hline 
	$(5,2,0)$ & $541,\!926$ \\
	\hline
	$(5,1,1)$ & $350,\!658$  \\
	\hline
	$(4,3,0)$ & $682,\!290$ \\
	\hline
	$(4,2,1)$ & $948,\!528$ \\
	\hline
	$(3,3,1)$ & $513,\!639$ \\
	\hline
	$(3,2,2)$ & $547,\!344$ \\
	\hline\hline
	\textbf{Total:} &  $193,\!919,\!175/49$ \\
	\hline
\end{tabular}
\end{minipage}
\begin{minipage}[t]{0.4\textwidth}
\noindent \center\textbf{Degree $8$}\\
\begin{tabular}{| c | c |}
	\hline
	\textbf{Multi-degree} & \textbf{Contribution} \\
	\hline \hline
	$(8,0,0)$ & $23,\!666,\!175/64$ \\
	\hline
	$(7,1,0)$ & $2,\!442,\!960$ \\
	\hline 
	$(6,2,0)$ & $4,\!601,\!610$ \\
	\hline
	$(6,1,1)$ & $3,\!116,\!880$ \\
	\hline
	$(5,3,0)$ & $6,\!375,\!600$ \\
	\hline
	$(5,2,1)$ & $9,\!448,\!560$ \\
	\hline
	$(4,4,0)$ & $28,\!227,\!969/8$ \\
	\hline
	$(4,3,1)$ & $11,\!139,\!552$ \\
	\hline
	$(4,2,2)$ & $6,\!045,\!264$ \\
	\hline
	$(3,3,2)$ & $6,\!407,\!712$ \\
	\hline\hline
	\textbf{Total:} &  $3,\!442,\!490,\!759/64$ \\
	\hline
\end{tabular}
\end{minipage}
\end{center}

\subsection{Conjectures}\label{sec: conjectures} \label{sec: hypergeometric formulae} Based on the low-degree calculations presented above, we conjecture general formulae for some of the component contributions. We then provide some theoretical evidence for these in Proposition~\ref{prop: hypergeometric conjecture equivalent to combinatorial conjecture}.

The conjectures are most conveniently stated by organising the component contributions according to the \emph{ordered} multi-degree, as opposed to the unordered multi-degree employed thus far. This is a fairly trivial refinement, amounting to simply dividing each unordered multi-degree contribution by its obvious symmetries. Given an unordered multi-degree $\mathbf{d}=(d_0,d_1,d_2)$, we let $\mathrm{A}(\mathbf{d})$ denote the number of ordered multi-degrees which induce $\mathbf{d}$ upon forgetting the ordering. For instance, we have $\mathrm{A}(d,0,0) = 3$ and:
\begin{align*}
& \mathrm{A}(d_1,d_2,0) = \begin{cases} 3 \qquad \text{if } d_1=d_2, \\ 6 \qquad \text{if } d_1 \neq d_2.\end{cases}
\end{align*}
The ordered multi-degree contribution is then obtained by dividing the unordered multi-degree contribution by $\mathrm{A}(\mathbf{d})$:
\begin{equation*} \operatorname{C}_{\operatorname{ord}}(\mathbf{d}) =  \operatorname{C}_{\operatorname{unord}}(\mathbf{d})/\mathrm{A}(\mathbf{d}).\end{equation*}

\begin{conjecture} \label{conj: hypergeometric conjecture} We have the following hypergeometric expressions for the ordered multi-degree contributions:
\begin{align} \label{eqn: multi-degree d hypergeometric conjecture}\operatorname{C}_{\operatorname{ord}}(d,0,0) = \dfrac{1}{d^2\ } {4d-1 \choose d} \qquad & (d \geq 1), \\[8pt]
\label{eqn: degree d1 d2 conjecture} \operatorname{C}_{\operatorname{ord}}(d_1,d_2,0) = \dfrac{6}{d_1 d_2} {4d_1+2d_2-1 \choose d_1 - 1} {4d_2 + 2d_1 - 1 \choose d_2 - 1} \qquad & (d_1,d_2\geq 1).
\end{align}
\end{conjecture}

\begin{remark} After taking the logarithm, \eqref{eqn: multi-degree d hypergeometric conjecture} is the coefficient of the diagonal term of the $3$-Kronecker quiver (according to a conjecture of Gross, proved by Reineke \cite[Theorem 6.4]{Reineke}). This supports the proposed correspondence with the scattering diagram calculations of~\cite{Graefnitz} (see Remark \ref{rmk: Graefnitz}). On the other hand, it is not immediately clear how to obtain \eqref{eqn: degree d1 d2 conjecture} via this same heuristic.\end{remark}

\begin{conjecture} The ordered multi-degree contributions enjoy the following integrality property:
\begin{equation*} \operatorname{gcd}(d_0,d_1,d_2)^2 \cdot \operatorname{C}_{\operatorname{ord}}(d_0,d_1,d_2) \in \Z_{\geq 0}.\end{equation*}\end{conjecture}

We conclude by providing some theoretical evidence for the conjectures. To be more precise, we will show that Conjecture \ref{conj: hypergeometric conjecture} \eqref{eqn: multi-degree d hypergeometric conjecture} is equivalent to the following purely combinatorial formula:
\begin{conjecture} \label{conj: combinatorial conjecture} Fix an integer $d \geq 1$. Then we have
\begin{equation}\label{eqn: combinatorial conjecture}
	\sum_{(d_1,\ldots,d_r) \vdash d} \dfrac{2^{r-1} \cdot d^{r-2}}{\#\! \operatorname{Aut}(d_1,\ldots,d_r)} \prod_{i=1}^r \dfrac{(-1)^{d_i-1}}{d_i} {3d_i \choose d_i} = \dfrac{1}{d^2} {4d-1 \choose d} \end{equation}
where the sum is over strictly positive unordered partitions of $d$ (of any length).\end{conjecture}
Unfortunately, we are unable to prove this conjecture itself (though we have verified it up to $d=50$). We note that, indexing the conjugacy classes of $\mathrm{S}_d$ by partitions and using the formula $d!/(\#\!\operatorname{Aut}(d_1,\ldots,d_r) \cdot \Pi_{i=1}^r d_i )$ for the size of each such class, the conjecture may be recast as a formula for the total sum of a certain class function on $\mathrm{S}_d$. Alternatively, one can encode the right-hand side and the product factors on the left-hand side into hypergeometric generating functions, and the conjecture then asserts a non-trivial relationship between these functions.

\begin{proposition} \label{prop: hypergeometric conjecture equivalent to combinatorial conjecture} Conjecture \ref{conj: hypergeometric conjecture} \eqref{eqn: multi-degree d hypergeometric conjecture} is equivalent to Conjecture \ref{conj: combinatorial conjecture}.\end{proposition}

\begin{proof} The connected component corresponding to the ordered multi-degree $(d,0,0)$ is simply:
\begin{equation*} \Kup_{0,1}(D_0,d) = \Kup_{0,1}(\PP^1,d) \subseteq \Kup(\Delta).\end{equation*}
We will show that the integral of the logarithmic virtual class over this component is equal to the left-hand side of \eqref{eqn: combinatorial conjecture}. Proceeding with the localisation procedure outlined in \S \ref{sec: localisation scheme}, we make the following specialisation:
\begin{equation*} \lambda_2 = 0.\end{equation*}
This is well-defined since $\lambda_2$ never appears as a factor in the denominator of a localisation contribution. At the end of the graph-splitting algorithm (see \S \ref{sec: graph splitting}), we are left with atomic graphs of the form:
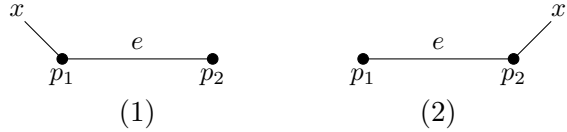
\begin{center}
\begin{tikzpicture}
	\draw (-3,-0.4) node[below]{$(1)$};
	
	\draw[fill=black] (-4,0) circle[radius=2pt];
	\draw (-4,0) node[below]{\small$p_1$};
	
	\draw (-4,0) -- (-4.5,0.5);
	\draw (-4.6,0.65) node{\small$x$};
	
	\draw (-4,0) -- (-2,0);
	\draw (-3,0) node[above]{\small$e$};
	
	\draw[fill=black] (-2,0) circle[radius=2pt];
	\draw (-2,0) node[below]{\small$p_2$};

	\draw (1,-0.4) node[below]{$(2)$};
	
	\draw[fill=black] (0,0) circle[radius=2pt];
	\draw (0,0) node[below]{\small$p_1$};
	
	\draw (0,0) -- (2,0);
	\draw (1,0) node[above]{\small$e$};
	
	\draw[fill=black] (2,0) circle[radius=2pt];
	\draw (2,0) node[below]{\small$p_2$};
	
	\draw (2,0) -- (2.5,0.5);
	\draw (2.6,0.65) node{\small$x$};
\end{tikzpicture}	
\end{center}
However, we see from \eqref{eqn: atomic contribution} that the atomic contributions of graphs of type $(2)$ contain a factor of $\lambda_2$, and therefore vanish. As such, we only need to consider localisation graphs whose atomic pieces are all of type $(1)$. It is easy to see that these must take the following simple form:
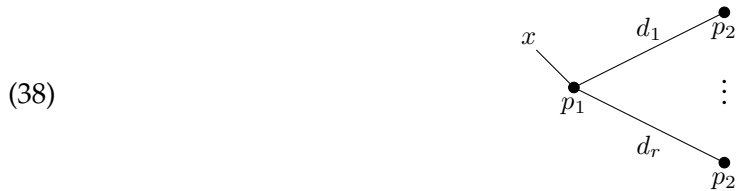
\begin{equation} \label{eqn: localisation graph}
\begin{tikzpicture}[baseline=(current  bounding  box.center)]

	\draw[fill] (0,0) circle[radius=2pt];
	\draw (0,0) node[below]{\small$p_1$};
	
	\draw (0,0) -- (-0.5,0.5);
	\draw (-0.6,0.65) node{\small$x$};
	
	\draw (0,0) -- (2,1);
	\draw (1,0.5) node[above]{\small$d_1$};
	\draw[fill] (2,1) circle[radius=2pt];
	\draw (2,1) node[below]{\small$p_2$};
	
	\draw (0,0) -- (2,-1);
	\draw (1,-0.5) node[below]{\small$d_r$};
	\draw[fill] (2,-1) circle[radius=2pt];
	\draw (2,-1) node[below]{\small$p_2$};
	
	\draw (2,0.05) node{$\vdots$};
	
\end{tikzpicture}	
\end{equation}
Recall that $\lambda_0+\lambda_1+\lambda_2=0$, and thus the specialisation $\lambda_2=0$ implies $\lambda_0=-\lambda_1$. Consequently, every localisation contribution will collapse to a number. Using \eqref{eqn: atomic contribution}, we calculate the product of the atomic contributions of \eqref{eqn: localisation graph} to be:
\begin{equation*} \prod_{i=1}^r \dfrac{(-1)^{d_i - 1}}{d_i^2} {3d_i \choose d_i}.\end{equation*}
On the other hand, the defect in the normal bundle is given by:
\begin{equation*} \int_{\ol{\Mcal}_{0,r+1}} \dfrac{2^{r-1} \cdot \lambda_1^{2r-2}}{\prod_{i=1}^r ( \lambda_1/d_i - \psi_i)}.\end{equation*}
Expanding the denominator as a power series and examining monomials of degree $r-2$ in the $\psi_i$, we obtain a sum over terms of the form
\begin{align*} \prod_{i=1}^r (d_i/\lambda_1)^{a_i+1} \cdot \int_{\ol\Mcal_{0,r+1}} \prod_{i=1}^r \psi_i^{a_i} \end{align*}
where $a_1+\ldots+a_r = r-2$ with $a_i \geq 0$. We may now calculate these integrals:
\begin{align*} \prod_{i=1}^r (d_i/\lambda_1)^{a_i+1} \int_{\ol\Mcal_{0,r+1}} \prod_{i=1}^r \psi_i^{a_i} & = \prod_{i=1}^r (d_i/\lambda_1)^{a_i+1}  \cdot {r-2 \choose a_1, \ldots, a_r} \\
& = (1/\lambda_1^{2r-2})  \prod_{i=1}^r d_i^{a_i+1} \cdot {r-2 \choose a_1, \ldots, a_r}. \end{align*}
Finally, by the multinomial theorem, the sum of these terms over $(a_1,\ldots,a_r)$ is equal to
\begin{equation*} (1/\lambda_1^{2r-2}) \left( \Pi_{i=1}^r d_i \right) (d_1+\ldots+d_r)^{r-2} = (1/\lambda_1^{2r-2}) \left( \Pi_{i=1}^r d_i \right) d^{r-2}\end{equation*}
and we conclude that the normal bundle defect is given by:
\begin{equation*} 2^{r-1} d^{r-2} \prod_{i=1}^r d_i. \end{equation*}
Multiplying this defect by the product of the atomic contributions, we obtain the contribution of the localisation graph \eqref{eqn: localisation graph}:
\begin{equation*} \dfrac{2^{r-1} \cdot d^{r-2}}{\#\!\operatorname{Aut}(d_1,\ldots,d_r)} \prod_{i=1}^r \dfrac{(-1)^{d_i-1}}{d_i} {3d_i \choose d_i}.\end{equation*}
Since such graphs are indexed by strictly positive unordered partitions of $d$, the claim follows.\end{proof}

\subsection{Degenerations of embedded curves}\label{sec: degenerations of curves}\label{sec: embedded curves} The Gromov--Witten theory of $(\PP^2,E)$ incorporates contributions from multiple covers and reducible curves, making a direct geometric interpretation difficult. On the other hand, in low degrees it is possible to directly count the number of embedded rational curves in $\PP^2$ maximally tangent to $E$ \cite{TakahashiCurvesComplement}. The relationship between these classical enumerative counts and the Gromov--Witten invariants is governed by multiple cover formulae \cite[\S 6]{GPS} and logarithmic gluing results \cite{CvGKTDegenerateContributions}, though there remain many degenerate loci whose contributions are not yet understood.

Consider, as before, a degeneration of $E$ to $\Delta$. An embedded tangent curve to $E$ degenerates uniquely along with the divisor, and it is natural to ask what one obtains in the central fibre. This limiting curve must be contained entirely inside $\Delta$ (otherwise, the limit would have to intersect $\Delta$ in at least two distinct points, and then the same would be true on the general fibre) and so every embedded tangent curve to $E$ defines a unique limiting multi-degree. Determining which curves in the general fibre limit to which multi-degrees in the central fibre, however, is a rather subtle problem.

In this section we uncover this limiting behaviour, using the above Gromov--Witten calculations (on the central fibre) together with known multiple cover formulae (on the general fibre) to unravel the behaviour of embedded curves. We obtain a complete description for $d=1,2,3$, and partial information for $d \geq 4$. These are results in classical enumerative geometry, but are, as far as we are aware, new. We do not know of a proof which does not pass through logarithmic Gromov--Witten theory.

\subsubsection{Review: tangent curves and torsion points} We begin with a brief recap of the geometry of tangent curves to $E$. This is a vast subject with a long history, and we make no attempt at completeness; for a more detailed exposition, see for instance \cite[\S 0]{BousseauTakahashi}.

Fix $p_0 \in E$ a flex point of the cubic. It is easy to show that if $C \subseteq \PP^2$ is a degree $d$ curve maximally tangent to $E$ at a point $p$, then $p$ must be a $3d$-torsion point of the elliptic curve $(E,p_0)$. Hence, for each $d$ there are precisely $(3d)^2$ candidate points on $E$ which can support a tangent curve of degree $d$.

These points can be subdivided according to their order in the group $(E,p_0)$. For our purposes, it is only the divisibility by $3$ which is important. We therefore say that a point $p$ has \emph{index $3k$} if $k$ is the smallest integer such that the order of $p$ divides $3k$. For example, when $d=2$ there are $36$ $6$-torsion points, and these split up into $9$ points of index $3$ and $27$ points of index $6$. Of the latter, $3$ have order $2$, while the remaining $24$ have order $6$, but this further refinement is not relevant to the discussion.

Given a point $p \in E$ of index $3k$ for $k|d$, we can ask for the number of embedded integral rational curves of degree $d$ intersecting $E$ with maximal tangency at the point $p$. It turns out that these numbers only depend on $d$ and $k$. They have been computed in low degrees in \cite{TakahashiCurvesComplement} (certain cases may also be deduced by combining Takahashi's formula \cite{BousseauTakahashi} with the Gross--Pandharipande--Siebert multiple cover formula \cite{GPS}).

For small $d$ it is therefore possible, summing over the $3d$-torsion points, to enumerate all the embedded tangent curves and describe precisely their contributions to the Gromov--Witten theory. In these cases, we can leverage our earlier Gromov--Witten calculations to study the degeneration behaviour of these curves.

\subsubsection{Degree $1$} This case is somewhat trivial. On the general fibre, there are $9$ $3$-torsion points which each support a unique tangent line. On the central fibre, the only valid multi-degree is $(1,0,0)$. There is, however, a finer decomposition given by the ordered multi-degree, which in this case records which of the co-ordinate lines $D_0,D_1,D_2$ support the limit curve. This decomposes the central fibre moduli space into $3$ connected components, each with a contribution of $3$ to the Gromov--Witten invariant.

Thus we see that, of the $9$ flex lines in the general fibre, $3$ of them limit to each of $D_0,D_1,D_2$. In the more complicated cases to follow, we will ignore this $3$-fold symmetry.

\subsubsection{Degree $2$}\label{sec: degree 2 degenerations} \label{sec: deg 2 degeneration} This is the first interesting case. There are $36$ $6$-torsion points, $9$ of which have index $3$ and the remaining $27$ of which have index $6$. By \cite[Proposition 1.4]{TakahashiCurvesComplement} we know that:
\begin{itemize}
\item each index $3$ point supports $1$ tangent line and no tangent conic;
\item each index $6$ point supports $1$ tangent conic.	
\end{itemize}
The general fibre moduli space $\Kup^{\log}(\PP^2|E)$ therefore consist of $27$ isolated points parametrising the tangent conics, together with $9$ one-dimensional components parametrising ramified double covers of the flex lines. Each of these one-dimensional components contributes $3/4$ to the Gromov--Witten invariant \cite[Proposition 6.1]{GPS}.

The central fibre moduli space $\Kup^{\log}(\Xcal_0)$, on the other hand, decomposes according to the multi-degrees $(2,0,0)$ and $(1,1,0)$ which we refer to as double covers and split curves, respectively. Of course, double covers in the general fibre must limit to double covers in the central fibre, but it is not clear how many of the $27$ conics in the general fibre limit to double covers, and how many limit to split curves. We thus have the following partially-completed degeneration picture
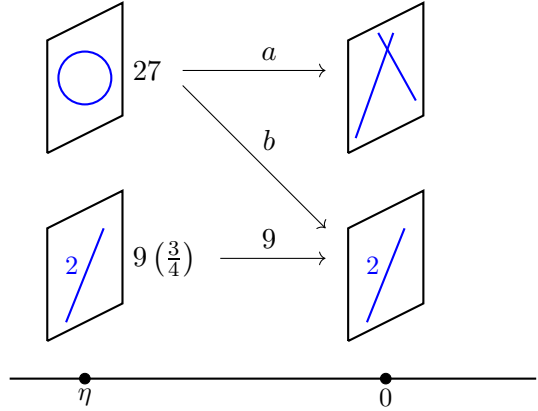
\begin{center}
\begin{tikzpicture}
\draw [thick] (0,0) -- (7,0);
\draw [fill] (1,0) circle[radius=2pt];
\draw (1,0) node[below]{\small{$\eta$}};
\draw [fill] (5,0) circle[radius=2pt];
\draw (5,0) node[below]{\small{$0$}};

\draw [thick] (0.5,0.5) -- (1.5,1) -- (1.5,2.5) -- (0.5,2) -- (0.5,0.5);
\draw [thick,blue] (0.75,0.75) -- (1.25,2);
\draw [blue] (1.05,1.5) node[left]{\small$2$};
\draw (1.5,1.6) node[right]{$9\left(\frac{3}{4}\right)$};
\draw [thick] (0.5,3) -- (1.5,3.5) -- (1.5,5) -- (0.5,4.5) -- (0.5,3);
\draw [thick,blue] (1,4) circle[radius=10pt];
\draw (1.5,4.1) node[right]{$27$};

\draw [thick] (4.5,0.5) -- (5.5,1) -- (5.5,2.5) -- (4.5,2) -- (4.5,0.5);
\draw [thick,blue] (4.75,0.75) -- (5.25,2);
\draw [blue] (5.05,1.5) node[left]{\small$2$};
\draw [thick] (4.5,3) -- (5.5,3.5) -- (5.5,5) -- (4.5,4.5) -- (4.5,3);
\draw [thick,blue] (4.6,3.2) -- (5.1,4.6);
\draw [thick,blue] (4.9,4.6) -- (5.4,3.7);

\draw [->] (2.3,4.1) -- (4.2,4.1);
\draw (3.45,4.1) node[above]{$a$};

\draw [->] (2.3,3.9) -- (4.2,2);
\draw (3.45,2.9) node[above]{$b$};

\draw [->] (2.8,1.6) -- (4.2,1.6);
\draw (3.45,1.6) node[above]{$9$};

\end{tikzpicture}	
\end{center}
with $a$ and $b$ unknown. But we have calculated the contributions of the components on the right-hand side (see \S \ref{sec: tables}): they are $18$ and $63/4$, respectively. This allows us to uniquely solve the above picture for $a$ and $b$:
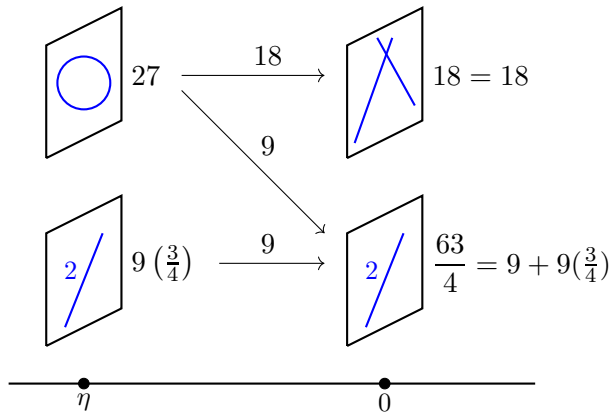
\begin{center}
\begin{tikzpicture}
\draw [thick] (0,0) -- (7,0);
\draw [fill] (1,0) circle[radius=2pt];
\draw (1,0) node[below]{\small{$\eta$}};
\draw [fill] (5,0) circle[radius=2pt];
\draw (5,0) node[below]{\small{$0$}};

\draw [thick] (0.5,0.5) -- (1.5,1) -- (1.5,2.5) -- (0.5,2) -- (0.5,0.5);
\draw [thick,blue] (0.75,0.75) -- (1.25,2);
\draw [blue] (1.05,1.5) node[left]{\small$2$};
\draw (1.5,1.6) node[right]{$9\left(\frac{3}{4}\right)$};
\draw [thick] (0.5,3) -- (1.5,3.5) -- (1.5,5) -- (0.5,4.5) -- (0.5,3);
\draw [thick,blue] (1,4) circle[radius=10pt];
\draw (1.5,4.1) node[right]{$27$};

 Double lines: multi-degree (2)
\draw [thick] (4.5,0.5) -- (5.5,1) -- (5.5,2.5) -- (4.5,2) -- (4.5,0.5);
\draw [thick,blue] (4.75,0.75) -- (5.25,2);
\draw [blue] (5.05,1.5) node[left]{\small$2$};
\draw [thick] (4.5,3) -- (5.5,3.5) -- (5.5,5) -- (4.5,4.5) -- (4.5,3);
\draw [thick,blue] (4.6,3.2) -- (5.1,4.6);
\draw [thick,blue] (4.9,4.6) -- (5.4,3.7);

\draw [->] (2.3,4.1) -- (4.2,4.1);
\draw (3.45,4.1) node[above]{$18$};

\draw [->] (2.3,3.9) -- (4.2,2);
\draw (3.45,2.9) node[above]{$9$};

\draw [->] (2.8,1.6) -- (4.2,1.6);
\draw (3.45,1.6) node[above]{$9$};

\draw (5.5,4.1) node[right]{$18=18$};
\draw (5.5,1.6) node[right]{$\dfrac{63}{4} = 9+9(\frac{3}{4})$};
\end{tikzpicture}	
\end{center}
Thus, of the $27$ embedded conics tangent to $E$, $18$ limit to split curves and $9$ limit to double lines.

We note that the decomposition of the $27$ index $6$ points into $18+9$ can also be obtained by examining the sizes of their orbits under the monodromy action. However, this information is not sufficient to unravel the degeneration behaviour for $d\geq 3$.

\subsubsection{Degree $3$} \label{sec: deg 3 degeneration} In this case there are $81$ $9$-torsion points, which split into $9$ points of index $3$ and $72$ of index $9$. By \cite[Proposition 1.5]{TakahashiCurvesComplement} (see also \cite{RanUnisecant}) we know that:
\begin{itemize}
\item each index $3$ point supports $1$ tangent line and $2$ tangent cubics;
\item each index $9$ point supports $3$ tangent cubics.	
\end{itemize}
Thus in total there are $9\cdot 2 + 72 \cdot 3 = 234$ tangent cubics in the general fibre moduli space. There are also $9$ two-dimensional components parametrising triple covers of the flex lines, each contributing $10/9$ to the Gromov--Witten invariant. The central fibre moduli space decomposes according to the possible multi-degrees $(3,0,0)$, $(2,1,0)$ and $(1,1,1)$. We thus have the following degeneration picture\newpage
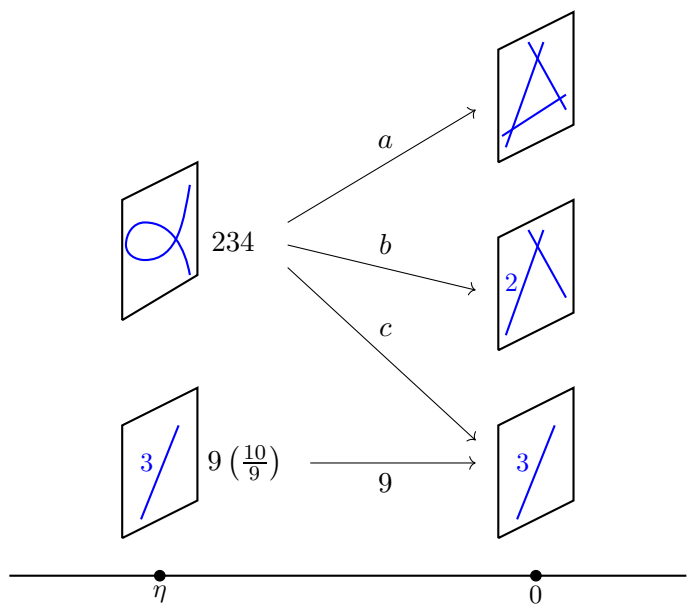
\begin{center}
\begin{tikzpicture}

\draw [thick] (-2,-0.5) -- (7,-0.5);
\draw [fill] (0,-0.5) circle[radius=2pt];
\draw (0,-0.5) node[below]{\small{$\eta$}};
\draw [fill] (5,-0.5) circle[radius=2pt];
\draw (5,-0.5) node[below]{\small{$0$}};

\draw [thick] (-0.5,0) -- (0.5,0.5) -- (0.5,2) -- (-0.5,1.5) -- (-0.5,0);
\draw [thick,blue] (-0.25,0.25) -- (0.25,1.5);
\draw [blue] (0.05,1) node[left]{\small$3$};
\draw (0.5,1) node[right]{$9\left(\frac{10}{9}\right)$};
\draw [thick] (-0.5,2.9) -- (0.5,3.5) -- (0.5,5) -- (-0.5,4.5) -- (-0.5,2.9);
\draw [color=blue,thick] (0.4,4.7) to [out=260,in=0] (-0.2,3.7) to [out = 180,in=270] (-0.45,3.9) to [out=90,in=180] (-0.2,4.2) to [out=0, in=100] (0.4,3.5);

\draw (0.55,3.95) node[right]{$234$};

\draw [thick] (4.5,0) -- (5.5,0.5) -- (5.5,2) -- (4.5,1.5) -- (4.5,0);
\draw [thick,blue] (4.75,0.25) -- (5.25,1.5);
\draw [blue] (5.05,1) node[left]{\small$3$};
\draw [thick] (4.5,2.5) -- (5.5,3) -- (5.5,4.5) -- (4.5,4) -- (4.5,2.5);
\draw [thick,blue] (4.6,2.7) -- (5.1,4.1);
\draw [blue] (4.9,3.4) node[left]{\small$2$};
\draw [thick,blue] (4.9,4.1) -- (5.4,3.2);
\draw [thick] (4.5,5) -- (5.5,5.5) -- (5.5,7) -- (4.5,6.5) -- (4.5,5);
\draw [thick,blue] (4.6,5.2) -- (5.1,6.6);
\draw [thick,blue] (4.9,6.6) -- (5.4,5.7);
\draw [thick,blue] (4.55,5.35) -- (5.4,5.9);

\draw [->] (1.7,4.2) -- (4.2,5.7);
\draw (3,5.05) node[above]{$a$};

\draw [->] (1.7,3.9) -- (4.2,3.3);
\draw (3,3.65) node[above]{$b$};

\draw [->] (1.7,3.6) -- (4.2,1.3);
\draw (3,2.55) node[above]{$c$};

\draw [->] (2,1) -- (4.2,1);
\draw (3,1) node[below]{$9$};

\end{tikzpicture}
\end{center}
with $a,b,c$ unknown. Once again, we can solve this uniquely using our knowledge of the central fibre contributions:
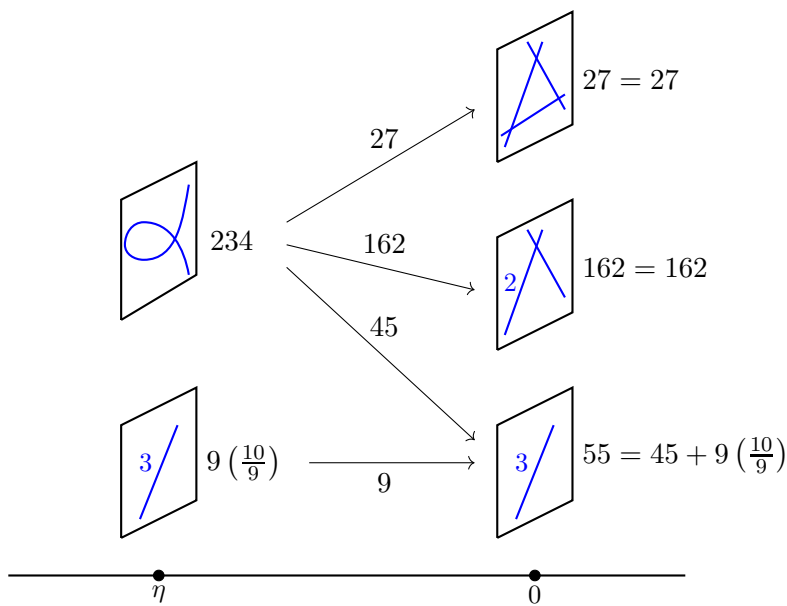
\begin{center}
\begin{tikzpicture}

\draw [thick] (-2,-0.5) -- (7,-0.5);
\draw [fill] (0,-0.5) circle[radius=2pt];
\draw (0,-0.5) node[below]{\small{$\eta$}};
\draw [fill] (5,-0.5) circle[radius=2pt];
\draw (5,-0.5) node[below]{\small{$0$}};

\draw [thick] (-0.5,0) -- (0.5,0.5) -- (0.5,2) -- (-0.5,1.5) -- (-0.5,0);
\draw [thick,blue] (-0.25,0.25) -- (0.25,1.5);
\draw [blue] (0.05,1) node[left]{\small$3$};
\draw (0.5,1) node[right]{$9\left(\frac{10}{9}\right)$};
\draw [thick] (-0.5,2.9) -- (0.5,3.5) -- (0.5,5) -- (-0.5,4.5) -- (-0.5,2.9);
\draw [color=blue,thick] (0.4,4.7) to [out=260,in=0] (-0.2,3.7) to [out = 180,in=270] (-0.45,3.9) to [out=90,in=180] (-0.2,4.2) to [out=0, in=100] (0.4,3.5);

\draw (0.55,3.95) node[right]{$234$};

\draw [thick] (4.5,0) -- (5.5,0.5) -- (5.5,2) -- (4.5,1.5) -- (4.5,0);
\draw [thick,blue] (4.75,0.25) -- (5.25,1.5);
\draw [blue] (5.05,1) node[left]{\small$3$};
\draw [thick] (4.5,2.5) -- (5.5,3) -- (5.5,4.5) -- (4.5,4) -- (4.5,2.5);
\draw [thick,blue] (4.6,2.7) -- (5.1,4.1);
\draw [blue] (4.9,3.4) node[left]{\small$2$};
\draw [thick,blue] (4.9,4.1) -- (5.4,3.2);
\draw [thick] (4.5,5) -- (5.5,5.5) -- (5.5,7) -- (4.5,6.5) -- (4.5,5);
\draw [thick,blue] (4.6,5.2) -- (5.1,6.6);
\draw [thick,blue] (4.9,6.6) -- (5.4,5.7);
\draw [thick,blue] (4.55,5.35) -- (5.4,5.9);

\draw [->] (1.7,4.2) -- (4.2,5.7);
\draw (3,5.05) node[above]{$27$};

\draw [->] (1.7,3.9) -- (4.2,3.3);
\draw (3,3.65) node[above]{$162$};

\draw [->] (1.7,3.6) -- (4.2,1.3);
\draw (3,2.55) node[above]{$45$};

\draw [->] (2,1) -- (4.2,1);
\draw (3,1) node[below]{$9$};

\draw (5.5,6.1) node[right]{$27=27$};
\draw (5.5,3.6) node[right]{$162=162$};
\draw (5.5,1.1) node[right]{$55 = 45+9\left(\frac{10}{9}\right)$};
\end{tikzpicture}
\end{center}

\subsubsection{Degree $4$ and beyond} For $d=4$, there are $144$ $12$-torsion points, which split into $9$ points of index $3$, $27$ points of index $6$ and $108$ points of index $12$. We know that \cite{TakahashiCurvesComplement}:
\begin{itemize}
\item each index $3$ point supports $1$ tangent line, no tangent conics and $8$ tangent quartics;
\item each index $6$ point supports $1$ tangent conic and $14$ tangent quartics;
\item each index $12$ point supports $16$ tangent quartics.	
\end{itemize}
There are thus $9 \cdot 8 + 27 \cdot 14 + 108 \cdot 16 =  2178$ embedded rational tangent quartics in the general fibre. In addition we have the following components parametrising degenerate maps:
\begin{itemize}
\item $9$ three-dimensional components, parametrising quadruple covers of flex lines;
\item $27$ one-dimensional components, parametrising double covers of embedded conics;
\item $9\cdot 2 = 18$ zero-dimensional components, parametrising reducible maps whose image is the union of a flex line and a tangent cubic passing through a flex point.
\end{itemize}
The multiple cover components contribute $35/16$ and $9/4$, respectively. On the other hand, logarithmic gluing considerations \cite{CvGKTDegenerateContributions} show that each of the $18$ components parametrising reducible curves contributes $3$ (more precisely, each ``component'' is actually made up of $3$ isolated points). We arrive at the following illustration of the general fibre moduli space:
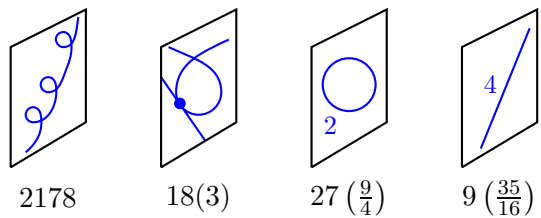
\begin{center}
\begin{tikzpicture}
\draw [thick] (-2.5,2.9) -- (-1.5,3.5) -- (-1.5,5) -- (-2.5,4.5) -- (-2.5,2.9);
\draw [thick,blue] (-1.6,4.9) to [out=260,in=0] (-1.8,4.5) to [out=180,in=280] (-1.9,4.6) to [out=90,in=180] (-1.8,4.7) to [out=0, in = 100] (-1.7,4.6) to [out=280, in=70] (-1.8,4.1) to [out=250,in=0] (-2,3.9) to [out=180, in=270] (-2.1,4) to [out=90,in=180] (-2,4.1) to [out=0,in=100] (-1.9,4) to [out=280,in=0] (-2.2,3.5) to [out=180, in=270] (-2.3,3.6) to [out=90,in=180] (-2.2,3.7) to [out=0,in=100] (-2.1,3.6) to [out=280, in=40] (-2.3,3.1);
\draw (-2,2.5) node{$2178$};

\draw [color=blue,thick] (-0.4,4.5) to [out=335,in=90] (0.3,3.9) to [out=270,in=0] (0,3.6) to [out=180,in=270] (-0.3,3.9) to [out=90, in=205] (0.4,4.6);
\draw [color=blue,thick] (-0.5,4.1) -- (0.09,3.25);
\draw [color=blue,fill] (-0.25,3.75) circle[radius=2pt];
\draw [thick] (-0.5,2.9) -- (0.5,3.5) -- (0.5,5) -- (-0.5,4.5) -- (-0.5,2.9);

\draw (0,2.5) node{$18(3)$};

\draw [thick] (1.5,2.9) -- (2.5,3.5) -- (2.5,5) -- (1.5,4.5) -- (1.5,2.9);
\draw [thick,blue] (2,4) circle[radius=10pt];
\draw [blue] (1.75,3.7) node[below]{\small$2$};
\draw (2,2.5) node{$27\left(\frac{9}{4}\right)$};

\draw [thick] (3.5,2.9) -- (4.5,3.5) -- (4.5,5) -- (3.5,4.5) -- (3.5,2.9);
\draw [thick,blue] (3.75,3.15) -- (4.4,4.75);
\draw [blue] (4.1,4) node[left]{\small$4$};
\draw (4,2.5) node{$9\left(\frac{35}{16}\right)$};


\end{tikzpicture}	
\end{center}
We wish to describe the degeneration behaviour of the $2178$ integral quartics. In order to do this, it is first necessary necessary to describe the degeneration behaviour of the multiple covers and reducible curves. The degenerations of multiple covers are determined by the previous calculations for $d=1$ and $d=2$.

The degenerations of the reducible curves, however, cannot be deduced from previous calculations. The problem is that, although we know how many of the $234$ tangent cubics limit to each of the multi-degrees $(3,0,0)$,$(2,1,0),(1,1,1)$, we are not able to separate out the limiting behaviour of cubics passing through an index $3$ point from those of cubics passing through an index $9$ point. This further information is crucial here, since only the cubics passing through an index $3$ point appear in the $d=4$ moduli space.

We hope to address this issue in the future by refining our techniques to separate out the contributions of different torsion points, possibly via a synthesis with the methods of \cite{Graefnitz}. For the moment, however, we are unable to make progress beyond $d=3$.

As a final remark it is worth pointing out that for $d > 4$ further difficulties arise, due to components of the \emph{general} fibre whose contributions to the Gromov--Witten invariants are not yet known. These include stable maps obtained by gluing two or more multiple covers, as well as those obtained by gluing three or more embedded tangent curves (where one also has moduli for the contracted component of the source curve). The degeneration questions we consider here may serve as motivation for the determination of such contributions.

\footnotesize
\bibliographystyle{alpha}
\bibliography{Bibliography}

\bigskip\bigskip

\noindent Lawrence Jack Barrott \\
Boston College, USA \\
\href{mailto:barrott@bc.edu}{barrott@bc.edu}\\

\noindent Navid Nabijou \\
University of Cambridge, UK \\
\href{mailto:nn333@cam.ac.uk}{nn333@cam.ac.uk}

\end{document}